\documentclass[11pt]{article}

\usepackage{preamble}
\title{Multipoint connectivity in the branching interlacement process}
\author{Louis Vanhaelewyn \thanks{Université Lyon 1, Centrale Lyon, INSA Lyon, Université Jean Monnet, CNRS, ICJ UMR5208, 69622 Villeurbanne, France.\\ \textit{email:} vanhaelewyn@math.univ-lyon1.fr}}

\begin{document}

\maketitle

\abstract{
  We consider the branching interlacement model introduced by Zhu as an analog of Sznitman's random interlacement for branching random walks. We show that two points of the interlacement are connected via at most $\lceil d/4 \rceil$ trajectories of the interlacement, using a different proof than Procaccia and Zhang. This upper bound is sharp, in the sense that almost surely there exist two points not connected by $\lceil d/4\rceil - 1$ trajectories. We extend this result by proving that $k$ points of the interlacement are connected via at most $\lceil d(k-1)/4\rceil -(k-2)$ trajectories, and that this bound is also sharp. \\

  MSC 2020 \textit{subject classifications.} 60K35, 60J80, 60D05
}

\section{Introduction}

Although the random interlacement model was orignally introduced in connection with the study of the trace left by a simple random walk on a torus, it also appeared to be an interesting object to study in its own. The interlacement can be seen either as the support of a Poisson point process on the set of doubly infinite trajectories on $\mathbb{Z}^d$, or as the subset of $\mathbb{Z}^d$ formed by the union of the ranges of these trajectories. In his seminal paper \cite{sznitman2010}, Sznitman was already interested in the existence of an infinite cluster in the vacant set, which led to many questions linked to percolation, see e.g.~\cite{dcgrst2022}, \cite{prevost2025}. Nonetheless, the model also shows strong connections with other random models such as the Gaussian Free Field or the loop soup model (\cite{sznitman2012}, \cite{lupu2014}). For our purpose, the structure of the interlacement seen as a graph is also interesting. Indeed, we can consider the graph whose vertices are trajectories of the interlacement, which are neighbours if they intersect. This has been studied by \cite{pt2011}, \cite{cp2012}, \cite{rs2011}, and \cite{rs2013} for example.

Recently, Zhu in \cite{zhu2016a} and \cite{zhu2016b} extended the result of Le Gall and Lin from \cite{ll2016} on the infinite invariant tree, and developed a theory of capacity for random walks indexed by this tree. Eventually, he also introduced in \cite{zhu2018} a counterpart of the random interlacement model in the branching random walk setting: the branching interlacement. Similar questions as for Sznitman's model have already been explored, either in the field of percolation (\cite{schapira2025}) or on the graph structure of the model (\cite{pz2016}).

Throughout this article, we will be interested in the connectivity properties of the branching interlacement model. Simultaneously, R\'ath and Sapozhnikov \cite{rs2011}, and Procaccia and Tykesson \cite{pt2011} proved that the diameter of the random interlacement is $\lceil d/2\rceil - 1$. This result was translated to the branching random interlacement by Procaccia and Zhang \cite{pz2016}, who showed that the diameter is $\lceil d/4\rceil-1$ using the concept of stochastic dimension as in \cite{pt2011}. In this paper, we extend their result to a more general class of branching random walks by developing an analogue of the tools introduced by R\'ath and Sapozhnikov for the branching random walk case. Afterwards we extend this result to the connection of $k$ points of the interlacement.

For completeness, let us state the theorem proved by R\'ath and Sapozhnikov, and Procaccia and Tykesson.

\begin{theorem}[R\'ath and Sapozhnikov \cite{rs2012} -- Procaccia and Tykesson \cite{pt2011}]
  Let $d\geq 3$. Denote by $G$ the graph whose vertices are trajectories of the interlacement, which are neighbours if they intersect each other, and denote by $\mathrm{diam}(G)$ its diameter. Then, almost surely,
  \[
    \mathrm{diam}(G) = \left\lceil \frac{d}{2} \right\rceil - 1.
  \]
\end{theorem}

Another formulation of this result is that almost surely one can connect two points of the interlacement with $\lceil d/2 \rceil$ trajectories, and find two points of the interlacement not connected by $\lceil d/2 \rceil -1$ trajectories. This result has been extended by Lacoin and Tykesson who answered the question on the number of trajectories needed to connect $k$ points. The theorem is the following.

\begin{theorem}[Lacoin and Tykesson \cite{lt2013}]
  Let $d \geq 5$, and $k \geq 3$. Then almost surely, for every $k$ points of the interlacement, there exist $\lceil d(k-1)/2\rceil - (k-2)$ trajectories of the interlacement that connect them.
  Furthermore, almost surely, there exist $k$ points of the interlacement that are not connected by $\lceil d(k-1)/2\rceil - (k-2) -1$ trajectories.
\end{theorem}

\subsection{Statement of results}
We consider here the branching interlacement setting, and get an analog of the result of R\'ath, Sapozhnikov, Procaccia and Tykesson.

\begin{theorem}
  \label{thm:2_points_branching}
  For $d\geq 5$. Let
  \begin{equation}
    \label{def:s_d}
    s_d = \left\lceil\frac{d}{4}\right\rceil - 1.
  \end{equation}
  Assume that the reproduction law is a mean $1$ distribution with finite third moment, and that the jump law of the walk is centered, has finite range and its support does not generate any strict subgroup of $\mathbb{Z}^d$. Then,
  \[
    \mathrm{diam}(G) = s_d.
  \]
\end{theorem}

This result has been proved in a less general setting by Procaccia and Zhang in \cite{pz2016}. They stated their theorem for the geometric reproduction law, and for the case where the steps of the walk are chosen uniformly over the unit basis of $\mathbb{Z}^d$. Similarly to Lacoin and Tykesson \cite{lt2013}, we can also address the question of the connection of $k$ points of the interlacement, for general $k \geq 3$.

\begin{theorem}
  \label{thm:k_points_branching}
  For $d \geq 9$, and $k \geq 3$. Let
  \begin{equation}
    \label{def:n_k_d}
    n(k,d) = \left\lceil\frac{d(k-1)}{4}\right\rceil - (k-2).
  \end{equation}
  Assume that the reproduction law is a mean $1$ distribution with finite $2\lceil d(k-1)/4\rceil$  moment, and that the jump law is centered, has finite range and its support does not generate any strict subgroup of $\mathbb{Z}^d$. Then almost surely, for every $k$ points of the interlacement, there exist $n(k,d)$ trajectories of the interlacement that connect them. Futhermore, almost surely, there exist $k$ points of the interlacement that are not connected using less than $n(k,d)$ trajectories.
\end{theorem}

Note that for $d \in \{5,\dots,8\}$, $s_d = 1$, meaning that almost surely any two trajectories intersect, hence $k$ points are connected using at most $k$ trajectories.

Note also that Theorem \ref{thm:k_points_branching} implies Theorem \ref{thm:2_points_branching}, but we state and prove the two successively for pedagogical reasons.

As we outline below, the proof strategy of Theorems \ref{thm:2_points_branching} and \ref{thm:k_points_branching} are similar to the random walk case, but the branching structure of the trajectories in our case induces important new difficulties, especially for the proof of Theorem \ref{thm:k_points_branching}.

\subsection{Sketch of proof}
\subsubsection{Heuristics for Theorem \ref{thm:2_points_branching}}
When $d$ is between $5$ and $8$, the theorem boils down to proving that in this case two branching random walks indexed by an infinite critical tree almost surely intersect. The proof is similar to the one of Lawler in \cite{lawler80}, and relies on a second moment method to bound from below the probability that two branching random walks intersect at a given point. We conclude using a version of the Borel-Cantelli lemma due to Kochen and Stone.

The proof in higher dimension is the most interesting and delicate part. The upper bound is proved by extending the concept of trace sets of R\'ath and Sapozhnikov to the branching random walk case. We recursively build a cluster of $\mathbb{Z}^d$ from a branching random walk with a scale parameter. We start by setting the cluster to a certain part of the walk. Then, we add pieces of trajectories that intersect the cluster, bounded in space and time relatively to the parameter. The branching capacity of the aggregate increases until step $s_d$ where a saturation phenomenon appears. At that point the cluster typically fills a ball, and the probability to be hit by an independent branching random walk is lower bounded by a constant. Altogether with the Borel-Cantelli lemma and by increasing the scale parameter, we show that this walk almost surely hits an aggregate of this form. With this result, we observe that considering two random walks of the interlacement, the first one almost surely hits the $s_d$-cluster of the other, therefore they are at graph distance $s_d$ from each other.

The lower bound is proved by contradiction. Suppose that, with probability greater than a constant, every two points of the interlacement need less than $s_d$ walks to be connected. Fix two points $x,y \in \mathbb{Z}^d$ of the interlacement and observe that there exist $s_d$ trajectories, such that two successive trajectories intersect at a point $z_i$, and $x$ and $y$ belong to the first and last trajectories. This creates a ``chain'' of random walk trajectories from $x$ to $y$ with ``links'' $z_i$. Each walk has approximately a probability $\|z_i - z_{i+1}\|^{4-d}$ to realise this event. We show that consequently the probability that $x$ and $y$ are connected with less than $n(k,d)$ trajectories is upper bounded by a constant times $\|x-y\|^{-1}$. Taking two points sufficiently far from each other yields a contradiction.

\subsubsection{Heuristics for Theorem \ref{thm:k_points_branching}}
The proof of the upper bound relies on the tools developed in Section \ref{sec:diameter}. Given $k$ points of the interlacement, we want to find $n(k,d)$ trajectories connecting those points. The key idea is to form overlapping groups of $5$ trajectories, then grow $5$ clusters around each of them, and find a walk of the interlacement that connects the clusters. To do so, we prove that we can connect $5$ walks with $d - 8$ trajectories. Then the walks remaining from the division by $5$ are connected in a similar way. Eventually, we find $n(k,d)$ trajectories connecting the points.

Proving the lower bound is much harder and is actually the most innovative part of this paper. Fix $k$ points $x_1,\dots,x_k$ of the interlacement and suppose that they are connected by $n < n(k,d)$ trajectories. Then there exists a graph $G$ with $k$ leaves $x_1,\dots,x_k$ and $n-1$ internal nodes $y_1,\dots,y_{n-1}$ indexed by elements of $\mathbb{Z}^d$, in which two points are neighbours if there exists a trajectory of the interlacement that connects them. Whereas this graph was a simple line in the two points case, we now have to deal with a generic graph with $k$ leaves and $n+k-1$ vertices. Moreover, to better reproduce the behavior of a branching random walk, we have to add ``branching'' vertices $z_1,\dots,z_m$ that appear naturally when asking a branching walk to connect $3$ points or more (see Section \ref{sec:n_pts_connection_walk}). This construction leads to a complex graph $G(\overline{x},\overline{y},\overline{z})$ with vertices in $\mathbb{Z}^d$. We reduce the graph into a single edge with two operations on graphs decreasing the number of vertices and applied successively. At the end, we get
\[
  \P(x_1,\dots,x_k \text{ connected by } n \text{ trajectories}) \lesssim \max_{i\neq j} \|x_i-x_j\|^{-1+\varepsilon},
\]
for some arbitrarily small constant $\varepsilon > 0$. This allows us to conclude by the same final argument as in the two points case.

\subsection{Structure of the article}
The article is structured as follows. In Section \ref{sec:def}, we introduce the definitions and notations relative to the infinite critical Galton-Watson trees (see Section \ref{sec:gw}) and the branching random walks (see Section \ref{sec:brw}). In particular, we present some of the results of Zhu, among which the branching capacity and the hitting probability for branching random walks. In the meantime, we develop results on the connection of $n$ points by branching random walks (see Section \ref{sec:n_pts_connection_walk}). Finally, we formally introduce the branching interlacement model (see Section \ref{sec:branching_interlacements}).

The paper is then divided into two non-independent parts, Section \ref{sec:diameter} proves Theorem \ref{thm:2_points_branching} and develop analogs of the results of R\'ath and Sapozhnikov in \cite{rs2011}, and Section \ref{sec:k_points_connection} proves Theorem \ref{thm:k_points_branching}, using the tools of Section \ref{sec:diameter}.

\section{Definitions, notations and known results}
\label{sec:def}
\subsection{Notations and a useful result}
\label{sec:notations}
For a radius $r > 0$, and a point $x \in \mathbb{Z}^d$, we denote by $\mathrm{B}(x,r)$ the (open) ball of radius $r$ centered at $x$ that is the set $\mathrm{B}(x,r) = \{y \in \mathbb{Z}^d, \|x - y\| < r \}$, where $\| \cdot \|$ is the Euclidean norm. In particular we write $\mathrm{B}(r)$ for $\mathrm{B}(0,r)$, the ball centered at $0$. For a set $K \subset \mathbb{Z}^d$ and a point $x \in \mathbb{Z}^d$ the distance from $x$ to $K$ is defined by $\mathrm{d}(x,K) = \inf_{y\in K} \|x-y\|$, and the diameter of $K$, when $K$ is finite, is defined by $\mathrm{diam}(K) = \max_{x,y\in K} \|x-y\| + 1$.

For two probability measures $\P$ and $\mathbb{Q}$ on the spaces $(\Omega_1, \mathcal{F}_1)$ and $(\Omega_2, \mathcal{F}_2)$ respectively, we denote by $\P \otimes \mathbb{Q}$ the product measure of $\P$ and $\mathbb{Q}$ defined on the space $(\Omega_1\times\Omega_2, \mathcal{F}_1\otimes\mathcal{F}_2)$, where $\mathcal{F}_1\otimes\mathcal{F}_2$ is the product $\sigma-$algebra of $\mathcal{F}_1$ and $\mathcal{F}_2$.

If $T$ is a rooted tree, and $u$ is a vertex of $T$ different from the root, we denote by $\parent{u}$ the parent of $u$ in $T$, that is the vertex visited just before $u$ in the geodesic from the root to $u$. For two vertices $u,v$, we write $u\wedge v$ for the greatest common ancestor of $u$ and $v$. If $u$ is a vertex, we denote by $\nu_u$ the number of children of $u$ in $T$. Furthermore, we denote by $E(T)$ the set of edges of $T$.

If $f$ and $g$ are two real functions, we write $f \lesssim g$ (resp.~$f \gtrsim g$) if there exists a constant $C$ that does not depend on the parameters, such that, $f \leq Cg$ (resp.~$f \geq Cg$). We write $f \asymp g$ if $f \lesssim g$ and $g \lesssim f$. If $f$ and $g$ are real functions defined on $\mathbb{Z}^d$, we denote by $f * g$ the convolution of $f$ and $g$, that is for $x\in\mathbb{Z}^d$, $f*g(x) = \sum_{y\in\mathbb{Z}^d} f(y) g(x-y)$.

We introduce here an inequality which will be useful for estimating convolutions of functions of the form $x \mapsto \|x\|^a$ (see e.g. \cite[Proposition 1.7]{hhs2003} for a proof).
\begin{lemma}
  \label{lem:magic_ineq}
  Let $d \geq 2$, $a \geq b > 0$. If $f,g : \mathbb{Z}^d \to \mathbb{R}$ are such that $|f(x)| \lesssim \|x\|^{-a}$ and $|g(x)| \lesssim \|x\|^{-b}$, then there exists a constant $C = C(a,b,d)$ such that for all $x \in \mathbb{Z}^d$,
  \[
    |(f*g)(x)| \leq
    \begin{cases}
      C\|x\|^{d-(a+b)} & \text{if } a < d \text{ and } a + b > d,\\
      C\|x\|^{-b} & \text{if } a > d.
    \end{cases}
  \]
\end{lemma}

\subsection{Infinite critical Galton-Watson trees}
\label{sec:gw}
We define here two infinite random trees, which can both be seen as limits of critical Galton-Watson trees conditioned to be large. Let $\mu$ be a probability measure over $\mathbb{N}$ with mean $1$ and positive variance $\sigma^2$, and denote by $\mu_{sb}$ the size biased distribution of $\mu$, defined by $\mu_{sb}(i) = i\mu(i)$, for $i \geq 1$. \\
The \textit{infinite invariant tree} is a plane tree constructed as follows:
\begin{itemize}
\item the root generates $i$ children with probability $\mu(i-1)$, the first one is \textit{special}, others are \textit{normal},
\item \textit{special} vertices generate $i$ children with probability $\mu_{sb}(i)$, one of them is chosen uniformly at random to be special,
\item \textit{normal} vertices generate $i$ children with probability $\mu(i)$, each of them is normal.
\end{itemize}
This construction is due to Le Gall and Lin \cite{ll2016}, Zhu \cite{zhu2018} and Bai and Wan \cite{bw2022}. If we forget the planar structure of the tree we get the infinite tree described by Aldous in \cite{aldous1991}, which can be seen as the Benjamini-Schramm limit of a Galton-Watson tree conditioned to have size $n$ when $n$ goes to infinity.

\noindent We will also consider \textit{Kesten's tree}, which is constructed as the infinite invariant tree, except that the root generates $i$ children with probability $\mu_{sb}(i)$. In the following, an \textit{infinite critical tree} refers to either a Kesten's tree or an infinite invariant tree.

In these two trees the set of special vertices (together with the root) forms a semi-infinite line called the spine, on which types of critical Galton-Watson trees are attached. 

One important remark is that if we truncate an infinite critical tree by cutting the edge between two nearest neighbor vertices on the spine, the remaining infinite component is a Kesten's tree.

The vertices of an infinite critical tree $\mathcal{T}$ can be labeled by elements of $\mathbb{Z}$ using a depth-first search counter-clockwise from infinity (see Figure \ref{fig:labelling}). Vertices with negative labels are called the \textit{past} of $\mathcal{T}$, denoted $\mathcal{T}_-$, and vertices with non-negative labels are called the \textit{future} of $\mathcal{T}$, denoted $\mathcal{T}_+$. The law of the resulting labeled tree is invariant by the shift transformation $\theta : k \mapsto k+1$ and its inverse, hence the name of the tree, see \cite{zhu2018} or \cite{bw2022} for a full proof.

Denote by $\widetilde{\mathcal{T}}_c$ the Galton-Watson tree with reproduction law $\mu$. Let $\widetilde{\mathcal{T}}_c$ be the \textit{adjoint critical tree}, defined as a Galton-Watson tree where each node has the same reproduction law $\mu$, except the root which generates offspring with law $\widetilde{\mu}$, where $\widetilde{\mu}(i) = \sum_{j \geq i+1} \mu(j)$. We can decompose $\mathcal{T}_-$ into a spine of special vertices on which are attached independent adjoint critical trees, except if $\mathcal{T}$ is an infinite invariant tree, then there is no tree attached to the root. We will denote by $\widetilde{\mathcal{T}}_c^{(n)}$ the adjoint critical tree attached to the $n^{\text{th}}$ node of the spine from the origin in the past.

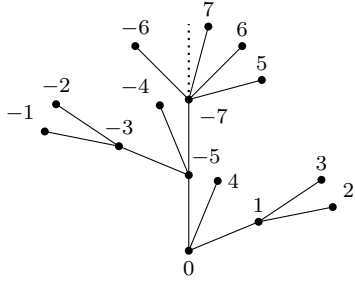
\begin{figure}[H]
  \centering
  \begin{tikzpicture}
    \node[dot=3] (0) at (0.000000,0.000000) {};
    \node[below] (0) at (0,0) {\scriptsize $0$};
    \node[dot=3] (0) at (0.923880,0.382683) {};
    \node[above] (0) at (0.923880,0.382683) {\scriptsize $1$};
    \draw (0.000000,0.000000) -- (0.923880,0.382683);
    \node[dot=3] (0) at (1.904665,0.577774) {};
    \node[above right] (0) at (1.904665,0.577774) {\scriptsize $2$};
    \draw (0.923880,0.382683) -- (1.904665,0.577774);
    \node[dot=3] (0) at (1.755349,0.938254) {};
    \node[above] (0) at (1.755349,0.938254) {\scriptsize $3$};
    \draw (0.923880,0.382683) -- (1.755349,0.938254);
    \node[dot=3] (0) at (0.382683,0.923880) {};
    \node[right] (0) at (0.382683,0.923880) {\scriptsize $4$};
    \draw (0.000000,0.000000) -- (0.382683,0.923880);
    \node[dot=3] (0) at (0.000000,1.000000) {};
    \draw (0.000000,1.000000) -- (0.000000,0.000000);
    \node[above right] (0) at (0.000000,1.000000) {\scriptsize $\!\!-5$};
    \node[dot=3] (0) at (-0.382683,1.923880) {};
    \node[above left] (0) at (-0.382683,1.923880) {\scriptsize $-4$};
    \draw (0.000000,1.000000) -- (-0.382683,1.923880);
    \node[dot=3] (0) at (-0.923880,1.382683) {};
    \node[above] (0) at (-0.923880,1.382683) {\scriptsize $-3$};
    \draw (0.000000,1.000000) -- (-0.923880,1.382683);
    \node[dot=3] (0) at (-1.755349,1.938254) {};
    \node[above] (0) at (-1.755349,1.938254) {\scriptsize $-2$};
    \draw (-0.923880,1.382683) -- (-1.755349,1.938254);
    \node[dot=3] (0) at (-1.904665,1.577774) {};
    \node[above left] (0) at (-1.904665,1.577774) {\scriptsize $-1$};
    \draw (-0.923880,1.382683) -- (-1.904665,1.577774);
    \node[dot=3] (0) at (0.000000,2.000000) {};
    \node[below right] (0) at (0.000000,2.000000) {\scriptsize $-7$};
    \draw (0.000000,2.000000) -- (0.000000,1.000000); 
    \node[dot=3] (0) at (-0.707107,2.707107) {};
    \node[above] (0) at (-0.707107,2.707107) {\scriptsize $-6$};
    \draw (0.000000,2.000000) -- (-0.707107,2.707107);
    \node[dot=3] (0) at (0.965926,2.258819) {};
    \node[above] (0) at (0.965926,2.258819) {\scriptsize $5$};
    \draw (0.000000,2.000000) -- (0.965926,2.258819);
    \node[dot=3] (0) at (0.707107,2.707107) {};
    \node[above] (0) at (0.707107,2.707107) {\scriptsize $6$};
    \draw (0.000000,2.000000) -- (0.707107,2.707107);
    \node[dot=3] (0) at (0.258819,2.965926) {};
    \node[above] (0) at (0.258819,2.965926) {\scriptsize $7$};
    \draw (0.000000,2.000000) -- (0.258819,2.965926);
    \draw[thick,dotted] (0.000000,2.000000) -- (0.000000,3.000000);
  \end{tikzpicture}
  \caption{Labeling of an infinite invariant tree.}
  \label{fig:labelling}
\end{figure}

For an infinite critical tree, denote by $(\xi_i)_{i \geq 1}$ the vertices of the spine ordered by a depth-first search from the origin. Then for a vertex $u$ in $\mathcal{T}$ we introduce the notation $n_u$, for the index of the node $\xi_{n_u}$ on the spine that is the nearest vertex from $u$ located on the spine (see Figure \ref{fig:spine_decomposition}).

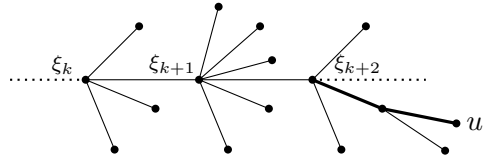
\begin{figure}[H]
  \centering
  \begin{tikzpicture}
    \begin{scope}[rotate=-90]
      \draw[thick,dotted] (0.000000,-1.000000) -- (0.000000,0.000000);
      \node[dot=3] (0) at (0.000000,0.000000) {};
      \node (0) at (-0.2000000,-0.30000) {\scriptsize $\xi_k$};
      \node[dot=3] (0) at (-0.707107,0.707107) {};
      \draw (0.000000,0.000000) -- (-0.707107,0.707107);
      \node[dot=3] (0) at (0.923880,0.382683) {};
      \draw (0.000000,0.000000) -- (0.923880,0.382683);
      \node[dot=3] (0) at (0.382683,0.923880) {};
      \draw (0.000000,0.000000) -- (0.382683,0.923880);
      \node[dot=3] (0) at (0.000000,1.500000) {};
      \node (0) at (-0.200000,1.150000) {\scriptsize $\xi_{k+1}$};
      \draw (0.000000,1.500000) -- (0.000000,0.000000);
      \node[dot=3] (0) at (-0.258819,2.465926) {};
      \draw (0.000000,1.500000) -- (-0.258819,2.465926);
      \node[dot=3] (0) at (-0.707107,2.307107) {};
      \draw (0.000000,1.500000) -- (-0.707107,2.307107);
      \node[dot=3] (0) at (-0.965926,1.758819) {};
      \draw (0.000000,1.500000) -- (-0.965926,1.758819);
      \node[dot=3] (0) at (0.923880,1.882683) {};
      \draw (0.000000,1.500000) -- (0.923880,1.882683);
      \node[dot=3] (0) at (0.382683,2.423880) {};
      \draw (0.000000,1.500000) -- (0.382683,2.423880);
      \node[dot=3] (0) at (0.000000,3.000000) {};
      \node (0) at (-0.20000,3.600000) {\scriptsize $\xi_{k+2}$};
      \draw (0.000000,3.000000) -- (0.000000,1.500000);
      \node[dot=3] (0) at (-0.707107,3.707107) {};
      \draw (0.000000,3.000000) -- (-0.707107,3.707107);
      \node[dot=3] (0) at (0.923880,3.382683) {};
      \draw (0.000000,3.000000) -- (0.923880,3.382683);
      \node[dot=3] (0) at (0.382683,3.923880) {};
      \draw[very thick] (0.000000,3.000000) -- (0.382683,3.923880);
      \node[dot=3] (0) at (0.938254,4.755349) {};
      \draw (0.382683,3.923880) -- (0.938254,4.755349);
      \node[dot=3] (0) at (0.577774,4.904665) {};
      \node[right] (0) at (0.577774,4.904665) {$u$};
      \draw[very thick] (0.382683,3.923880) -- (0.577774,4.904665);
      \draw[thick,dotted] (0.000000,3.000000) -- (0.000000,4.5000000);
    \end{scope}
  \end{tikzpicture}
  \caption{Example of a spine decomposition and nearest vertex on the spine ($n_u = k+2$).}
  \label{fig:spine_decomposition}
\end{figure}

\subsection{Tree-indexed random walks}
\label{sec:brw}
A \textit{random walk} with jump distribution $\theta$ on $\mathbb{Z}^d$ indexed by a tree $T$ with root $\rho$, is a family $(X_u)_{u\in T}$ of $\mathbb{Z}^d$-valued random variables indexed by the set of vertices of $T$, such that the variables $(X_u - X_{\parent{u}})_{\substack{u\in T\backslash\{\rho\}}}$ are independent, and follow the same law $\theta$.

The random walks we are interested in are the one indexed by a random tree $\mathcal{T}$ (a critical Galton-Watson or an infinite critical tree), where the law $\theta$ satisfies some conditions described later. For $x \in \mathbb{Z}^d$ denote by $X^x$ the branching random walk started from $x$, i.e. such that if $\rho$ is the root of $\mathcal{T}$, $X_\rho^x = x$. If no point is specified we assume that the walk starts from the origin.

We introduce the following notation for the range of a random walk $X$ indexed by a tree $\mathcal{T}$,
\[
  \mathcal{T}^x = \{X_u^x : u \in \mathcal{T}\},
\]
and we also write sometimes $\mathrm{Range}(X)$ for the range of the $\mathbb{Z}^d-$valued random process.
We omit the index $x$ if the walk starts from the origin, or if the starting point is already specified.

If $\mathcal{T}$ is an infinite critical tree, we will sometimes refer to $X(k)$ as the position of the walk at the vertex labeled $k \in \mathbb{Z}$ in the tree. Furthermore, we introduce the following notations for ranges of the branching random walk,
\[
  \mathcal{T}^x = \{X^x(k) : k \in \mathbb{Z}\}, \quad \mathcal{T}_+^x = \{X^x(k) : k \geq 0\}, \quad \mathcal{T}_-^x = \{X^x(k) : k < 0\},
\]
i.e. $\mathcal{T}^x$ denotes the range of the whole walk, while $\mathcal{T}_+^x$ and $\mathcal{T}_-^x$ denote the ranges in the future and the past respectively. 

We denote by $\P_x$ (respectively $\overline{\P}_x$) the law of a random walk indexed by an infinite invariant (respectively Kesten's) tree started at point $x$. Similarly we introduce $\E_x$ and $\overline{\E}_x$ for the respective expectation. We omit the index $x$ when it start from the origin.

In this whole work we assume the following hypotheses on $\theta$, which will allow us to apply the results of Zhu from \cite{zhu2016a} \cite{zhu2016b} \cite{zhu2018}, that we state later: \\[-1.75em]
\begin{itemize}  \setlength\itemsep{0em}
\item it has finite range,
\item it is centered (with zero mean),
\item it is not supported on any strict subgroup of $\mathbb{Z}^d$.
\end{itemize}

\noindent This random walk is transient in dimension $d \geq 5$ (see for example \cite{ll2016}), which allows us to define the following function introduced by Zhu in \cite{zhu2016b}.

\begin{definition}[Green's function]
  The \textit{Green's function} of the branching random walk $(X_u)_{u\in\mathcal{T}}$ is defined for $x \in \mathbb{Z}^d$ by
  \[
    G(x) = \E\Bigg[\sum_{u\in\mathcal{T}_-} \indic{X_u = x}\Bigg].
  \]
  Furthermore, we define $G(x,y) = G(y-x)$, for $x,y \in \mathbb{Z}^d$.
\end{definition}
\noindent As a direct consequence of the Markov inequality, one has for any $x,y \in \mathbb{Z}^d$,
\begin{equation}
  \label{eq:hit_T_past}
  \P_x(\mathcal{T}_- \text{ hits } y) \leq G(x,y).
\end{equation}
We also consider the Green's function of a standard random walk $(S_n)_{n\geq 0}$ with jump distribution $\theta$, which is defined by
\[
  g(x) = \E\Bigg[\sum_{n\geq 0} \indic{S_n = x}\Bigg].
\]
Similarly we define for $x,y \in\mathbb{Z}^d$, $g(x,y) = g(y-x)$. One important property is that (see e.g. \cite{ass2023})
\begin{equation}
  \label{eq:G_g_convolution}
  G(x,y) \asymp g * g (x,y) = \sum_{z\in\mathbb{Z}} g(x,z)g(y,z),
\end{equation}
from which we can deduce the following asymptotic expansion at infinity:
\begin{equation}
  \label{eq:estim_green}
  G(x,y) = G(y-x) \asymp \frac{1}{1+\|x-y\|^{d-4}}.
\end{equation}
Note also that by the Markov property and the definition of $\mathcal{T}^x$,
\begin{equation}
  \label{eq:hit_T}
  \P(\mathcal{T}^{x} \text{ hits } y) \leq \P(\mathcal{T}_-^x \text{ hits } y) + \P(\mathcal{T}_+^x \text{ hits } y) \leq 2G(x,y) + g(x,y) \lesssim G(x,y).
\end{equation}
The Green's function of the critical adjoint tree is equal to
\begin{equation}
  \label{eq:green_adjoint}
  \E_x\Bigg[\sum_{\ u\in\widetilde{\mathcal{T}}_c} \indic{X_u = y}\Bigg] = \delta_{x,y} + \frac{\sigma^2}{2}\left(g(x,y) - \delta_{x,y}\right) \asymp g(x,y).
\end{equation}

\begin{theorem}[\cite{zhu2016a}, \cite{zhu2016b}]
  \label{thm:bound_hitting_probability}
  Let $\varepsilon > 0$. There exist positive constants $c_1, c_2$, such that for any nonempty $K \subset \mathbb{Z}^d$ and any $x \in \mathbb{Z}^d$, with $\mathrm{d}(x,K) \geq \varepsilon \cdot \mathrm{diam}(K)$, one has,
  \begin{equation}
    \label{thm:bound_hitting_probability_infinite}
    c_1\frac{\mathrm{BCap}(K)}{\mathrm{d}(x,K)^{d-4}} \leq \P(\mathcal{T}_-^x \cap K \neq \emptyset) \leq c_2\frac{\mathrm{BCap}(K)}{\mathrm{d}(x,K)^{d-4}}.
  \end{equation}
  \begin{equation}
    \label{thm:bound_hitting_probability_critical}
    c_1\frac{\mathrm{BCap}(K)}{\mathrm{d}(x,K)^{d-2}} \leq \P(\mathcal{T}_c^x \cap K \neq \emptyset) \leq c_2\frac{\mathrm{BCap}(K)}{\mathrm{d}(x,K)^{d-2}}.
  \end{equation}
\end{theorem}
\noindent The same result holds when replacing $\mathcal{T}_-$ by $\mathcal{T}_+$ in \eqref{thm:bound_hitting_probability_infinite}.

\begin{definition}[Equilibrium measure and branching capacity]
  Let $K$ be a finite subset of $\mathbb{Z}^d$. The \textbf{equilibrium measure} of $K$ is the measure defined by
  \[
    e_K(x) = \indic{x \in K}\ \P(\mathcal{T}_-^x \cap K = \emptyset),\quad x \in \mathbb{Z}^d.
  \]
  The \textbf{branching capacity} of $K$ is defined by
  \[
    \mathrm{BCap}(K) = \sum_{x \in K} e_K(x).
  \]
  We can normalize the equilibrium measure and define a probability measure $\widetilde{e}_K(x)$ by
  \[
    \widetilde{e}_K(x) = \frac{e_K(x)}{\mathrm{BCap}(K)},\quad x \in \mathbb{Z}^d.
  \]
\end{definition}

From this we can define the law of a random walk indexed by an infinite invariant tree, starting from a point chosen according to $\widetilde{e}_K$ and conditioned to not hit $K$ on negative times. This will be useful later for the construction of the branching interlacement.
\begin{definition}
  For $K \subset \mathbb{Z}^d$, we define a probability measure $\widetilde{P}_{\widetilde{e}_K}$ by
  \[
    \widetilde{P}_{\widetilde{e}_K}(\,\cdot\,) = \sum_{x \in K}\widetilde{e}_K(x)\ \P(\,\cdot \mid \mathcal{T}_-^x \cap K = \emptyset),
  \]
  where $\mathbb{P}$ is the law of $\mathcal{T}^x$.
\end{definition}

\begin{proposition}[Branching capacity of a ball \cite{zhu2016a}]
  There exist constants $c, C > 0$, such that for every $R > 0$,
  \label{res:bcap_ball}
  \begin{equation}
    \label{eq:bcap_ball}
    cR^{d-4} \leq \mathrm{BCap}(\mathrm{B}(R)) \leq C R^{d-4}.
  \end{equation}
\end{proposition}

If we denote by $(S_n)_{n\in\mathbb{N}}$ the simple random walk induced by the spine of the infinite critical tree, we can define
\begin{equation}
  \label{eq:exit_time_br}
  \tau_R = \inf\{n \in \mathbb{N} : S_n \notin \mathrm{B}(R)\},
\end{equation}
that is the exit time of the ball of radius $R$ by the walk $(S_n)_{n\in\mathbb{N}}$. 

\subsubsection{Connection of $n$ points by a branching random walk}
\label{sec:n_pts_connection_walk}
In this section, we prove two lemmas regarding the connection probability of $n$ points by a branching random walk, which extend the simple upper bound from \eqref{eq:hit_T_past} for $2$ points. More precisely we prove two upper bounds, first when the walk is indexed by a critical Galton-Watson tree and then by an infinite critical tree.

\begin{definition}[Generated tree, see Figure \ref{fig:generated_tree}]
  Let $T$ be a finite planar tree with more than $n$ vertices, and $v_1,\dots,v_n$ distinct vertices of $T$. The \textbf{tree generated by $v_1,\dots,v_n$ in $T$} is the tree $T\langle v_1,\dots,v_n\rangle$ with vertex set
  \[
    \{v_1,\dots,v_n\} \cup \{v_i \wedge v_j, i \neq j\},
  \]
  and where two vertices are connected by an edge if the geodesic between the two in $T$ does not contain any other vertex of $T\langle v_1,\dots,v_n\rangle$.
\end{definition}

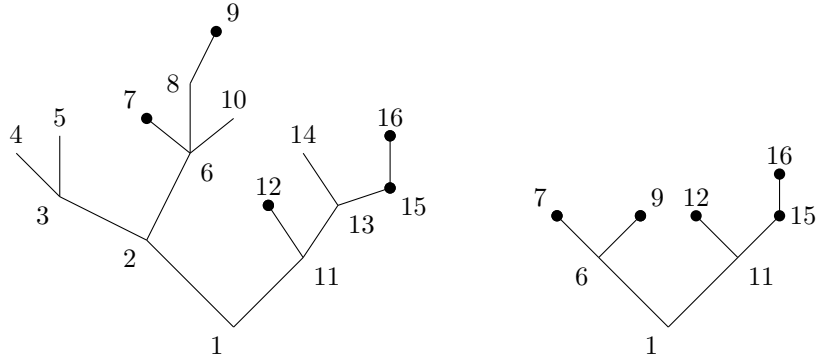
\begin{figure}[H]
  \centering
  \begin{tikzpicture}[scale=2.3]
    \begin{scope}
      \coordinate (1) at (0,0);
      \coordinate (2) at (-0.5,0.5);
      \coordinate (3) at (-1,0.75);
      \coordinate (4) at (-1.25, 1);
      \coordinate (5) at (-1, 1.1);
      \coordinate (6) at (-0.25,1);
      \coordinate (7) at (-0.5, 1.2);
      \coordinate (8) at (-0.25,1.4);
      \coordinate (9) at (-0.1,1.7);
      \coordinate (10) at (0,1.2);
      \coordinate (11) at (0.4,0.4);
      \coordinate (12) at (0.2,0.7);
      \coordinate (13) at (0.6,0.7);
      \coordinate (14) at (0.4,1);
      \coordinate (15) at (0.9,0.8);
      \coordinate (16) at (0.9,1.1);

      \node at (7) [circle,fill,inner sep=1.5pt] {};
      \node at (9) [circle,fill,inner sep=1.5pt] {};
      \node at (12) [circle,fill,inner sep=1.5pt] {};
      \node at (15) [circle,fill,inner sep=1.5pt] {};
      \node at (16) [circle,fill,inner sep=1.5pt] {};

      \node at (1) [below left] {\small$1$};
      \node at (2) [below left] {\small$2$};
      \node at (3) [below left] {\small$3$};
      \node at (4) [above] {\small$4$};
      \node at (5) [above] {\small$5$};
      \node at (6) [below right] {\small$6$};
      \node at (7) [above left] {\small$7$};
      \node at (8) [left] {\small$8$};
      \node at (9) [above right] {\small$9$};
      \node at (10) [above] {\small$10$};
      \node at (11) [below right] {\small$11$};
      \node at (12) [above] {\small$12$};
      \node at (13) [below right] {\small$13$};
      \node at (14) [above] {\small$14$};
      \node at (15) [below right] {\small$15$};
      \node at (16) [above] {\small$16$};
      
      \draw (1) -- (2);
      \draw (2) -- (3);
      \draw (3) -- (4);
      \draw (3) -- (5);
      \draw (2) -- (6);
      \draw (6) -- (7);
      \draw (6) -- (8);
      \draw (8) -- (9);
      \draw (6) -- (10);
      \draw (1) -- (11);
      \draw (11) -- (12);
      \draw (11) -- (13);
      \draw (13) -- (14);
      \draw (13) -- (15);
      \draw (15) -- (16);
    \end{scope}
    \begin{scope}[xshift=2.5cm, scale=0.8]
      \coordinate (1) at (0,0);
      \coordinate (6) at (-0.5, 0.5);
      \coordinate (7) at (-0.8, 0.8);
      \coordinate (9) at (-0.2, 0.8);
      \node at (7) [circle,fill,inner sep=1.5pt] {};
      \node at (9) [circle,fill,inner sep=1.5pt] {};
      \node at (12) [circle,fill,inner sep=1.5pt] {};
      \node at (15) [circle,fill,inner sep=1.5pt] {};
      \node at (16) [circle,fill,inner sep=1.5pt] {};

      \coordinate (11) at (0.5, 0.5);
      \coordinate (12) at (0.2, 0.8);
      \coordinate (15) at (0.8, 0.8);
      \coordinate (16) at (0.8, 1.1);

      \node at (1) [below left] {\small$1$};
      \node at (6) [below left] {\small$6$};
      \node at (7) [above left] {\small$7$};
      \node at (9) [above right] {\small$9$};
      \node at (11) [below right] {\small$11$};
      \node at (12) [above] {\small$12$};
      \node at (15) [right] {\small$15$};
      \node at (16) [above] {\small$16$};

      \node at (7) [circle,fill,inner sep=1.5pt] {};
      \node at (9) [circle,fill,inner sep=1.5pt] {};
      \node at (12) [circle,fill,inner sep=1.5pt] {};
      \node at (15) [circle,fill,inner sep=1.5pt] {};
      \node at (16) [circle,fill,inner sep=1.5pt] {};
      
      \draw (1) -- (6);
      \draw (6) -- (7);
      \draw (6) -- (9);
      \draw (1) -- (11);
      \draw (11) -- (12);
      \draw (11) -- (15);
      \draw (15) -- (16);
    \end{scope}
  \end{tikzpicture}
  \caption{Example of a tree $T$ and the generated tree $T\langle 7,9,12,15,16\rangle$}
  \label{fig:generated_tree}
\end{figure}

\begin{definition}[Multipoint connection following a tree]
  \label{def:connects_following_critical}
  Let $X$ be a random walk indexed by a critical tree $\mathcal{T}_c$. Let $x_0,\dots,x_n \in \mathbb{Z}^d$, and $T$ be a planar tree with vertex set $\{x_0,\dots,x_n\}$, rooted at $x_0$. We say that $X$ \textbf{connects} $x_0,\dots,x_n$ \textbf{following} $T$, if $X_\varnothing = x_0$, and there exist $v_1,\dots,v_n \in \mathcal{T}_c$ such that $X_{v_i} = x_i$, and $\mathcal{T}_c\langle \varnothing,v_1,\dots,v_n\rangle$ is isomorphic to $T$, with $v_i$ mapped to $x_i$.
\end{definition}

Note that in the previous definition the vertex set of $\mathcal{T}_c\langle \varnothing,v_1,\dots,v_n\rangle$ is exactly $\{\varnothing,v_1,\dots,v_n\}$.

\begin{lemma}
  \label{lem:connection_critical}
  Let $X$ be a random walk indexed by a critical tree $\mathcal{T}_c$, $x_0,\dots,x_n \in \mathbb{Z}^d$, and $T$ a rooted planar tree with vertex set $\{x_0,\dots,x_n\}$, and rooted at $x_0$. If $\mu$ has a finite moment of order $n$, then
  \[
    \P_{x_0}\big(X \text{ connects } x_0,\dots,x_n \text{ following } T\big) \lesssim \prod_{\{x_i,x_j\} \in E(T)} \frac{1}{1+\|x_i - x_j\|^{d-2}}.
  \]
\end{lemma}

\begin{proof}
  We prove the result by induction on $n \geq 0$. If $n=0$, there is nothing to prove. Suppose now that $T$ is not reduced to its root $x_0$. Then the graph $T \backslash \{x_0\}$ is made of trees $T_1, \dots, T_k$. For $j \leq k$, let $I_j$ be the vertex set of $T_j$, and $x_{i_j}$ be the root of $T_j$. We have
  \[
    \P_{x_0}\left(X \text{ follows } T\right) = \sum_{m \geq k}\mu(m) \binom{m}{k} \prod_{j=1}^k \P_{x_0}\big(X \text{ connects } (x_i)_{i\in I_j^*} \text{ following } T^*_j \mid \nu_\varnothing = 1\big),
  \]
  where $I_j^* = I_j \cup \{x_0\}$, $T_j^* = T_j \cup \{x_0\}$ is the tree rooted at $x_0$, and in which $x_0$ is connected to the root of $T_j$.
  Let $g_*(x) = \E\Big[\sum_{\substack{v \in \mathcal{T}_c\\ v \neq \varnothing}} \indic{X_v = x} \mid \nu_\varnothing = 1\Big] = g(x) - \indic{x = 0}$. Then
  \[
    \P_{x_0}\left(X \text{ follows } T \right) \leq \sum_{m \geq k}\mu(m) \binom{m}{k} \prod_{j=1}^k g_*(x_{i_j}-x_0) \P_{x_{i_j}}\big(X \text{ connects } (x_i)_{i\in I_j} \text{ following } T_j\big).
  \]
  Recall that $\mu$ has a finite moment of order $n$. By induction and the fact that ${g_*(x) \lesssim 1/(1+\|x\|^{d-2})}$ and $\sum_{m\geq k} \mu(m)\binom{m}{k} \leq \sum_{m\geq 1} \mu(m) m^n$ for all $k \leq n$, this yields 
  \[
    \P_{x_0}\left(X \text{ follows } T\right) \lesssim \prod_{j=1}^k \frac{1}{1+\|x_0-x_{i_j}\|^{d-2}} \prod_{\{x,y\} \in E(T_j)} \frac{1}{1+\|x-y\|} = \!\!\prod_{\{x_i,x_j\} \in E(T)} \frac{1}{1+\|x_i-x_j\|^{d-2}}.
  \]
    
\end{proof}

The same proof applies for a random walk indexed by an adjoint critical tree defined in Section \ref{sec:gw} with a finite $n+1$ moment, or for other critical Galton-Watson trees with modified root laws.

We now prove a similar result for branching random walks indexed by an infinite critical tree. For $m \leq n \in \mathbb{N} \cup \{\infty\}$, let $\mathbb{T}_{m,n}$ be the set of planar rooted trees with $n$ nodes, where $m$ vertices form a nearest-neighbor path (called the \textbf{spine}) and such that the root is one of the two ends of the spine. For $T\in\mathbb{T}_{m,n}$, and $u \in T$, denote by $s_u$ the nearest vertex from $u$ on the spine of $T$.

\begin{definition}[Generated tree with a spine]
  Let $1 \leq k,m \leq n$. Let $T \in \mathbb{T}_{m,n}$ and $v_1,\dots,v_k$ distinct vertices of $T$. The \textbf{tree generated by $v_1,\dots,v_k$ in $T$} is the tree $T\langle v_1,\dots,v_k\rangle$ with vertex set
  \[
    \{v_1,\dots,v_k\} \cup \{v_i \wedge v_j, i \neq j\} \cup \{s_{v_i}, i \leq k\},
  \]
  and where two vertices are connected if the geodesic between the two in $T$ does not contain any other vertex of $T\langle v_1,\dots,v_k\rangle$.
\end{definition}

\begin{definition}[Multipoint connection following a spined tree]
  \label{def:connects_following}
  Let $X$ be a random walk indexed by an infinite critical tree $\mathcal{T}$. Let $1\leq m \leq n+1$, and $T \in \mathbb{T}_{m,n+1}$ with vertex set $x_0,\dots,x_n \in \mathbb{Z}^d$ rooted at $x_0$. We say that $X$ \textbf{connects} $x_0,\dots,x_n$ \textbf{following} $T$, if $X_\varnothing = x_0$, and there exist $v_1,\dots,v_n \in \mathcal{T}$  such that $X_{v_i} = x_i$ and $\mathcal{T}\langle \varnothing,v_1,\dots,v_n\rangle$ is isomorphic to $T$ with $v_i$ mapped to $x_i$, and preimages of elements of the spine of $T$ on the spine of $\mathcal{T}$.
\end{definition}

\begin{lemma}
  \label{lem:connection_infinite}
    Let $X$ be a random walk indexed by an infinite critical tree $\mathcal{T}$, $x_0,\dots,x_n \in \mathbb{Z}^d$, and $T \in \mathbb{T}_{m,n+1}$ a rooted planar tree with vertex set $\{x_0,\dots,x_n\}$. If $\mu$ has a finite moment of order $n+1$, then
  \[
    \P_{x_0}\big(X \text{ connects } x_0,\dots,x_n \text{ following } T\big) \lesssim \prod_{\{x_i,x_j\} \in E(T)} \frac{1}{1+\|x_i - x_j\|^{d-2}}.
  \]
\end{lemma}
\begin{proof}
  We prove this lemma by decomposing the trees $T$ and $\mathcal{T}$ according to their spine. Let $(T_j^-)_{1\leq j\leq m-1}$ and $(T_j^+)_{1\leq j\leq m-1}$ be the trees attached respectively to the left and right of the $j^{\text{th}}$ element of the spine of $T$ (starting from the root), and denote by $T_m$ the tree attached to the last vertex of the spine. For $\varepsilon \in \{-,+\}$, let $I_j^\varepsilon$ be the set of indices such that $V(T_j^\varepsilon) = \{x_i, i \in I_j^\varepsilon\}$, and let $x_{i_j}$ be the root of $T_j^\varepsilon$. Define $I_m$ and $x_{i_m}$ the same way for $T_m$. Since the random walk induced by the spine of $\mathcal{T}$ on $X$ is a simple random walk, and if $\widetilde{\P}_x$ denotes the law of a random walk indexed by an adjoint critical tree starting at $x\in\mathbb{Z}^d$, we get
  \begin{align*}
    &\P_{x_0}\left(\!\!\begin{array}{c} X \text{ connects } x_0,\dots,x_n \\ \text{ following } T \end{array}\!\!\right) \\
    &\hspace{1.5cm}\lesssim \prod_{j=1}^{m-1} g(x_{i_{j+1}}-x_{i_j}) \prod_{\varepsilon\in\{-,+\}}\widetilde{\P}_{x_{i_j}}\left(\!\!\begin{array}{c} X \text{ connects } (x_i)_{i\in I_j^\varepsilon} \\ \text{ following } T_j^\varepsilon \end{array}\!\!\right) \times \widetilde{\P}_{x_{i_m}}\left(\!\!\begin{array}{c} X \text{ connects } (x_i)_{i\in I_m} \\ \text{ following } T_m \end{array}\!\!\right).
  \end{align*}
  Finally, we apply Lemma \ref{lem:connection_critical} for an adjoint critical tree with the $n+1$ finite moment hypothesis, which gives
  \begin{align*}
    &\P_{x_0}\left(\!\!\begin{array}{c} X \text{ connects } x_0,\dots,x_n \\ \text{ following } T \end{array}\!\!\right) \\
    &\hspace{2cm}\lesssim \prod_{j=1}^{m-1} \frac{1}{1+\|x_{i_{j+1}}-x_{i_j}\|^{d-2}} \prod_{\varepsilon\in\{-,+\}}\prod_{\{x,y\} \in E(T_j^\varepsilon)} \frac{1}{1 + \|x-y\|^{d-2}} \times \prod_{\{x,y\} \in E(T_m)} \frac{1}{1 + \|x-y\|^{d-2}}\\
    &\hspace{2cm}= \prod_{\{x_i,x_j\}\in E(T)} \frac{1}{1+\|x_i-x_j\|^{d-2}}.
  \end{align*}
\end{proof}

As a corollary, we can deduce an upper bound on the probability that a branching random walk hits $n$ given points. First of all, let $\mathbb{T}_n^k$ be the set of planar rooted trees with $n$ vertices among which $k$ are distinguished, and with a spine of length smaller than or equal to $n$ starting from the root, satisfying
\begin{itemize}  \setlength\itemsep{0em}
\item the root and the leaves are distinguished vertices,
\item undistinguished vertices are all of degree at least $3$, except possibly the one at the end of the spine which may also have degree $2$.
\end{itemize}
For $\overline{x}_k = (x_1,\dots,x_k) \in \mathbb{Z}^d$ and $\overline{z}_{n-k} = (z_1,\dots,z_{n-k})$ denote by $\mathbb{T}_n^k(\overline{x}_k, \overline{z}_{n-k})$ the subset of trees in $\mathbb{T}_n^k$ with distinguished vertices $x_1,\dots,x_k$ and undistinguished vertices $z_1,\dots,z_{n-k}$.

\begin{lemma}
  \label{lem:n_points_connection}
  Let $x_1,\dots,x_k \in \mathbb{Z}^d$. Let $X$ be a random walk indexed by an infinite critical tree $\mathcal{T}$. Then under $\P_{x_1}$,
  \begin{equation}
    \{x_1,\dots,x_k \in \mathrm{Range}(X)\} = \bigcup_{m=0}^{k-1}\bigcup_{z_1,\dots,z_m \in \mathbb{Z}^d} \bigcup_{T \in \mathbb{T}_{k+m}^k(\overline{x}_k,\overline{z}_m)} \left\{\begin{array}{c} X \text{ connects } x_1,\dots,x_k, z_1,\dots,z_m\\ \text{ following } T\end{array}\right\}.
  \end{equation}
  Consequently, if $\mu$ has a finite moment of order $2k-1$,
  \begin{equation}
    \P_{x_1}\big(x_1,\dots,x_k \in \mathrm{Range}(X)\big) \lesssim \sum_{m=0}^{k-1}\sum_{z_1,\dots,z_m \in \mathbb{Z}^d} \sum_{T \in \mathbb{T}_{k+m}^k(\overline{x}_k,\overline{z}_m)} \prod_{\{u,v\}\in E(T)}\frac{1}{1+\|u-v\|^{d-2}}.
  \end{equation}
\end{lemma}
\begin{proof}
  If $x_1,\dots,x_k \in \mathrm{Range}(X)$ then for each $1 \leq i \leq k,$ there exists $u_i \in \mathcal{T}$, such that $X_{u_i} = x_i$. Let $T = \mathcal{T}\langle u_1,\dots,u_k\rangle$, $\{v_1,\dots,v_m\} = V(T) \backslash\{u_1,\dots,u_k\}$, and $z_i = X_{v_i}$. By definition, under $\P_{x_1}$, $X \text{ connects } x_1,\dots,x_k, z_1,\dots,z_m \text{ following } \widetilde{T}$, obtained from $T$ by replacing $v_i$ with $z_i$. Furthermore, observe that each $z_i$ is of degree at least $3$ except possibly the last vertex on the spine of $T$, of degree at least $2$. Thus $T \in \mathbb{T}_{k+m}^k(\overline{x}_k,\overline{z}_m)$. Since we have
  \[
    2(k+m-1) = 2 \#E(T) = \sum_{z \in V(T)} \mathrm{deg}(z) \geq k + 3(m-1) + 2,
  \]
  we deduce that $m \leq k-1$. Thus under $\P_{x_1}$,
  \[
    \{x_1,\dots,x_n \in \mathrm{Range}(X)\} \subset \bigcup_{m=0}^{k-1}\bigcup_{z_1,\dots,z_m \in \mathbb{Z}^d} \bigcup_{T \in \mathbb{T}_{k+m}^k(\overline{x}_k,\overline{z}_m)} \left\{\begin{array}{c} X \text{ connects } x_1,\dots,x_k, z_1,\dots,z_m \\\text{ following } T\end{array}\right\}.
  \]
  The last assertion of the lemma follows by a union bound and Lemma \ref{lem:connection_infinite}.
\end{proof}

\subsection{Branching interlacements}
\label{sec:branching_interlacements}
Following the idea of Sznitman in \cite{sznitman2010} for simple random walks, Zhu introduced in \cite{zhu2018} the branching interlacement model, a Poisson point process on the set of branching random walks.

Let $W = \{w : \mathbb{Z} \to \mathbb{Z}^d,\, \lim_{n\to\pm\infty} \|w(n)\| = +\infty \}$. We denote by $\mathcal{W}$ the cylinder $\sigma-$algebra on $W$. Then we consider the space,
\[
  W^* = W/_\sim, \quad \text{with } w \sim w' \Leftrightarrow \exists\,k \in \mathbb{Z},\, w(\cdot) = w'(k + \cdot).
\]
We let $\pi_*$ be the canonical projection from $W$ to $W^*$, and define the $\sigma-$algebra $\mathcal{W}^*$ as the push forward of $\mathcal{W}$ by $\pi_*$. If $K \subset \mathbb{Z}^d$ is nonempty and finite, we define $W_K$ the space of trajectories intersecting $K$ as
\[
  W_K = \{w \in W : \exists\,n\in\mathbb{Z},\ w(n) \in K\}.
\]
We define the first entrance time of an element $w \in W_K$ to be,
\[
  H_K(w) = \inf\{n \in \mathbb{Z} : w(n) \in K\},
\]
and let $W_K^0$ be the set of trajectories whose first entrance time is $0$,
\[
  W_K^0 = \{w \in W_K,\, H_K(w) = 0\}.
\]
For $K \subset \mathbb{Z}^d$ finite and nonempty, we can now define a measure $Q_K$ supported on $W_K^0$, such that
\[
  \forall A \in \mathcal{W},\quad Q_K(A) = \sum_{x\in K} \mathbb{P}\left((X^x(k))_{k\in\mathbb{Z}} \in A,\, \mathcal{T}_-^x \cap K = \emptyset\right),
\]
where $X^x$ is a random walk indexed by an infinite invariant tree $\mathcal{T}$ starting from $x$.
This family of measures is compatible, which allows to define a $\sigma$-finite measure $\nu$ on $(W^*,\mathcal{W}^*)$, by
\[
  \forall A\in\mathcal{W}^*,\quad \nu(A) = \limsup_{|K|\to\infty}Q_K(\pi_*^{-1}(A)\cap W_K).
\]

The \textit{branching interlacement} is the Poisson point process on $W^* \times \mathbb{R}_+$ with intensity measure $\nu \times \lambda$, where $\lambda$ is the Lebesgue measure. Its law is a probability measure $\sigma$ on the space of locally finite point measures on $W^* \times \mathbb{R}_+$,
\[
  \Omega = \left\{\omega = \sum_{n\in\mathbb{N}} \delta_{(w_n,u_n)} : \omega(W_K^* \times [0,u]) < \infty,\, \forall K \subset \mathbb{Z}^d \text{ finite},\, \forall u \geq 0\right\}.
\]

For $u > 0$ and $K$ a finite subset of $\mathbb{Z}^d$, we define the random variable $\sigma_{K,u}$ as the restriction of the branching interlacement to the trajectories intersecting $K$ that appear before time $u$, that is
\[
  \sigma_{K,u} \colon \omega \mapsto \sum_{n\in\mathbb{N}}\delta_{w_n}\indic{u_n \leq u}\indic{w_n \cap K \neq \varnothing},\quad \text{where } \omega = \sum_{n\in\mathbb{N}} \delta_{(w_n,u_n)} \in \Omega.
\]
Thus $\sigma_{K,u}$ is a Poisson point process on $W$ with intensity $u\cdot\mathrm{BCap}(K) \cdot \widetilde{P}_{\widetilde{e}_K}$. The following proposition gives another description of $\sigma_{K,u}$.

\begin{proposition}[{\cite[Proposition 2]{zhu2018}}]
  \label{prop:desc_interlacement_compact}
  Let $N_K$ be a Poisson random variable with parameter $u \cdot \mathrm{BCap}(K)$, and consider $(X^{(i)})_{i\geq 1}$ independent random walks indexed by infinite invariant trees with starting point chosen according to $e_K$, conditioned not to touch $K$ on negative times. Then $\sigma_{K,u}$ has the same law as $\sum_{i=1}^{N_K}\delta_{X^{(i)}}$, where the $(X^{(i)})_{i\geq 1}$ are identified with elements of $W^*$.
\end{proposition}

\noindent For $u > 0$, we denote by $\mathrm{Poi}(u,W^*)$ the law of the push forward of the branching interlacement process by the map
\[
  \sigma_{u} \colon \omega \mapsto \sum_{n\in\mathbb{N}}\delta_{w_n}\indic{u_n \leq u},\quad \text{where } \omega = \sum_{n\in\mathbb{N}} \delta_{(w_n,u_n)} \in \Omega.
\]
For a general point process $\omega = \sum_{n\in\mathbb{N}} \delta_{x_n}$, we denote by $\mathrm{Supp}(\omega)$ the support of $\omega$ that is the set $\{x_n, n\in\mathbb{N}\}$. If $\omega \in \Omega$, we denote by $\mathcal{I}(\omega)$ the subset of $\mathbb{Z}^d$ defined by $\mathcal{I}(\omega) = \bigcup_{w \in \mathrm{Supp}(\omega)}\mathrm{Range}(w)$, where $\mathrm{Range}(w)$ is the set of points of $\mathbb{Z}^d$ visited by $w$. In particular we will use the notation $\mathcal{I}^u$ for $\mathcal{I}(\sigma_u)$.

\section{Diameter of the branching interlacement graph}
\label{sec:diameter}
In this section we prove Theorem \ref{thm:2_points_branching}. First we consider the case $5 \leq d \leq 8$ in Section \ref{sec:5_d_8}, we show that the diameter is almost surely one by proving that two branching random walks indexed by infinite critical trees almost surely intersect. Then we prove that in dimension $d \geq 9$, the diameter is upper bounded by $s_d$ in Section \ref{sec:upper_bound_2}, and that it is lower bounded by $s_d$ in Section \ref{sec:lower_bound_2_pt}.

\subsection{Proof for $5 \leq d \leq 8$}
\label{sec:5_d_8}
A result of Erd\"os and Taylor \cite{et1960} states that two simple random walks in dimension less than or equal to $4$ almost surely intersect. The same result holds for branching random walks in dimension $d \leq 8$. The case $d \leq 4$ follows from the recurrence of the walk (see Le Gall and Lin  \cite{ll2016}).

\begin{proposition}
  \label{prop:2_connection_as}
  Let $5 \leq d \leq 8$.
  Consider two independent random walks indexed by infinite critical trees $\mathcal{T}$ and $\widetilde{\mathcal{T}}$ respectively. Then for every $x_1, x_2 \in \mathbb{Z}^d$,
  \[
    \P\big(\mathcal{T}^{x_1} \cap \widetilde{\mathcal{T}}^{x_2} \neq \varnothing \big) = 1.
  \]
\end{proposition}
\begin{proof}
  We follow here the idea of Lawler from \cite[Theorem 2.6]{lawler80}. Without loss of generality assume that $x_1 = 0$, and relabel $x_2$ to $x \in \mathbb{Z}^d$.
  Let $X_1$ and $X_2$ be random walks indexed by infinite critical trees starting respectively from the origin and $x$. For $y \in \mathbb{Z}^d$ consider the event
  \[
    A_y = \{X_1 \text{ and } X_2 \text{ hit } y\}.
  \]
  
  We will use a version of the Borel-Cantelli lemma due to Kochen and Stone \cite{ks1964} which states that if a sequence of events $(E_k)_{k\in\mathbb{N}}$ satisfies
  \[
    \limsup_{n\to\infty} \frac{\left(\sum_{1\leq k\leq n} \P(E_k)\right)^2}{\sum_{1\leq i,j \leq n} \P(E_i \cap E_j)} > 0,
  \]
  then the events $(E_k)_{k\in\mathbb{N}}$ occur infinitely often. In our case by independence between $X_1$ and $X_2$, and \eqref{thm:bound_hitting_probability_infinite},
  \[
    \P(A_y) = \P(X_1 \text{ hits } y) \P(X_2 \text{ hits } y) \asymp G(0,y) G(y,x).
  \]
  Moreover, on the event $\{y,z \in \mathcal{T}^x\}$ there exist $u, v \in \mathcal{T}$ such that $X_{u} = y$ and $X_{v} = z$. If we distinguish between the cases $n_u < n_v$, $n_u = n_v$, and $n_u > n_v$, we get
  \[
    \P(y,z \in \mathcal{T}^x) \lesssim f(x,y,z) + f(y,z,x) + f(z,x,y),
  \]
  where $f : (x,y,z) \mapsto \sum_{w \in\mathbb{Z}^d} G(x,w)g(w,z)g(w,y)$. Let $r = \min(\|x-y\|,\|x-z\|)/2$. If $w \in \mathrm{B}(x,r)$, then $g(w,y) \asymp g(x,y)$ and $g(w,z) \asymp g(x,z)$, while if $w \notin \mathrm{B}(x,r)$, then $G(x,w) \lesssim r^{4-d}$. Altogether we get by Lemma \ref{lem:magic_ineq} and the fact that $\sum_{w \in \mathrm{B}(x,r)} G(x,w) \lesssim r^4$,
  \[
    f(x,y,z) \lesssim G(x,y) G(x,z) + G(x,y)G(y,z) + G(x,z) G(z,y).
  \]
  As a consequence,
  \begin{align*}
    \P(A_{y} \cap A_{z})
    &= \P(y,z \in \mathcal{T}) \P(y,z \in \widetilde{\mathcal{T}}^x) \\
    &\lesssim \Big(G(y)G(z) + G(y,z)G(y) + G(y,z)G(z)\Big)\Big(G(x,y)G(x,z) + G(y,z)G(y,x) + G(y,z)G(z,x)\Big).
  \end{align*}
  \noindent Now, as long as we choose $\|y\|$ large enough, we can approximate $G(x,y)$ by $G(0,y)$. More precisely, there exists $r_0 > 0$, such that for $\|y\| \geq r_0$,
  $G(0,y)/2 \leq G(x,y) \leq 3 G(0,y) /2$.
  If $\|y\|, \|z\| \geq r_0$, we get
  \begin{align}
    \P(A_{y} \cap A_{z})
    &\lesssim (G(y)G(y,z) + G(y)G(z) + G(z)G(y,z))^2 \nonumber \\
    &\lesssim (G(y)G(y,z))^2 + (G(y)G(z))^2 + (G(z)G(y,z))^2, \label{eq:A1_A2_three}
  \end{align}
  and likewise, if $\|y\| \geq r_0$,
  \[
    \P(A_y) \asymp G(0,y)^2.
  \]
  Consider the annulus $C_{r_0,R} = \{y \in \mathbb{Z}^d : r_0 \leq \|y\| \leq R\}$. Then, as $R \to +\infty$, 
  \[
    \sum_{y \in C_{r_0,R}} \P(A_y) \asymp \int_{r_0}^R \frac{r^{d-1}}{r^{2d-8}}\, \mathrm{d}r = \int_{r_0}^R \frac{1}{r^{d-7}}\,\mathrm{d}r \asymp f(R) - f(r_0),
  \]
  where
  \[
    f(R) = \begin{cases}R^{8-d} & \text{ if } d \leq 7, \\ \ln\left(R\right) & \text{ if } d = 7.\end{cases}
  \]
  We bound each term of the sum in \eqref{eq:A1_A2_three} separately. For the first term, we get
  \[
    \sum_{y,z\in C_{r_0,R}} (G(0,y)G(y,z))^2 \leq \sum_{y \in C_{r_0,R}} \sum_{z \in \mathrm{B}(2R)} \frac{1}{(1 + \|y\|^{d-4}) (1 + \|z\|^{d-4})} \lesssim f(R)f(2R) \asymp f(R)^2.
  \]
  By symmetry, the same bound is true for the third term in \eqref{eq:A1_A2_three}. For the second one, we just note that,
  \[
    \sum_{y,z\in C_{r_0,R}} (G(0,y)G(0,z))^2 = \bigg(\sum_{y \in C_{r_0,R}} G(0,y)^2\bigg)^2 \lesssim f(R)^2.
  \]

  \noindent Altogether, we infer
  \[
    \lim_{R \to \infty} \frac{\left(\sum_{y\in C_{r_0,R}}\P(A_y)\right)^2}{\sum_{y,z\in C_{r_0,R}}\P(A_{y}\cap A_{z})} \gtrsim 1.
  \]
  Hence, we can apply the version of the Borel-Cantelli lemma alluded to earlier, and we deduce that the events $(A_y)_{y\in\mathbb{Z}^d}$ occur infinitely often, which in turn proves the proposition.
\end{proof}

Theorem \ref{thm:2_points_branching} follows from this result for dimensions $5$ to $8$, as it implies that the diameter of the graph $G$ is exactly $1$ in these cases, and we can also verify that $s_d = 1$.

\subsection{Proof of the upper bound $\mathrm{diam}(G) \leq s_d$}
\label{sec:upper_bound_2}
In this section, we prove the upper bound over the diameter of the graph in the case $d\geq 9$. To do so we adapt the strategy of R\'ath and Sapozhnikov from \cite{rs2012}. Recall the definition of $s_d$ in \eqref{def:s_d} and let $X$ and $X^*$ be two fixed branching random walks in the branching interlacement process. We want to find a family of branching random walks in the interlacement, say $X^1,\dots,X^k$ with $k \leq s_d-1$, such that $X^i$ and $X^{i+1}$ intersect each other for $1\leq i \leq k-1$, and such that $X$ and $X^*$ also intersect one of the branching random walks in this family.

In order to find such walks, we recursively construct sets $(A^{(s)})_{s \geq 1}$ (see Subsection \ref{sec:visible_sets}) which are built using an application $\Phi$ (see \eqref{eq:def_phi}) and the branching random walk $X$, such that if a point lies in $A^{(s)}$, there exist $s-1$ branching random walks in the interlacement that link it to $X$. With branching capacity estimates of the sets $A^{(s)}$ (see Lemmas \ref{lem:phi_capacity_bound}, \ref{lem:lower_bound_psi}, \ref{lem:bcap_moments_bound} and \ref{lem:lower_bound_capacity_final}), we show that the probability that a branching random walk hits the set $A^{(s_d)}$ is bounded from below by a positive constant. Combining this with a Borel-Cantelli lemma (see Lemma \ref{lem:bc_cond}), we are able to conclude.

\subsubsection{Trace sets}
\label{sec:trace_sets}

In this subsection, we introduce an application $\Phi$ that selects parts of traces of walks that are not too far from their starting point both spatially and temporally.

Let $(X^i)_{1 \leq i \leq N}$ be a sequence of random walk trajectories indexed by infinite critical trees $(\mathcal{T}^{(i)})_{1 \leq i \leq N}$. For $R > 0$, let $\overline{X}_N = (X^1,\dots,X^N)$, and
\begin{equation}
  \label{eq:def_phi}
  \Phi(\overline{X}_N,R) = \bigcup_{i=1}^N \left(\left\{X^i_u,\, u \in \mathcal{T}_-^{(i)} : n_u \leq R^2\right\}\ \cap\ \mathrm{B}(X^i_\varnothing, R)\right).
\end{equation}

\begin{figure}[H]
  \centering
  \begin{tikzpicture}[scale=1.5,use Hobby shortcut]
    \draw (0,0)  circle (2);
    \draw[above right] (1,1.73) node {$R$};
    \draw[below] (0,0) node {\small $X_{\varnothing}$};
    \draw[very thick] (0,0) .. (0.1,0.2) .. (0.5,-0.9) .. (-0.3,0.4) .. (-0.5,0) .. (-1.73,-1);
    \draw[very thick, dotted] (-1.73,-1) .. (-2.5,-1.2) .. (-2.6, 0) .. (-2.2,0.3) .. (-3,-0.4) .. (-2.8,-0.8) .. (-1.5,-2) .. (-0.35,-1.97);
    \draw[very thick] (-0.35,-1.97) .. (-0.15,-1.8) .. (0,-1.7) .. (-0.3,-1.4) .. (-0.2,-1) .. (0.2, -1.2) .. (1.29, -1.53);
    \draw[very thick, dotted] (1.29, -1.53) .. (2,-1.3) .. (2.5,-1.6) .. (2.6, -2);
    \draw[right] (2.6,-2) node {\small$X_{R^2}$};
    \draw (0,0.8) node {$\Phi(X,R)$};
  \end{tikzpicture}

  \caption{Example of a set $\Phi(X,R)$ for a single random walk $X$. We only draw the spine of the branching random walk for simplicity.}
  \label{fig:ex_phi}
\end{figure}
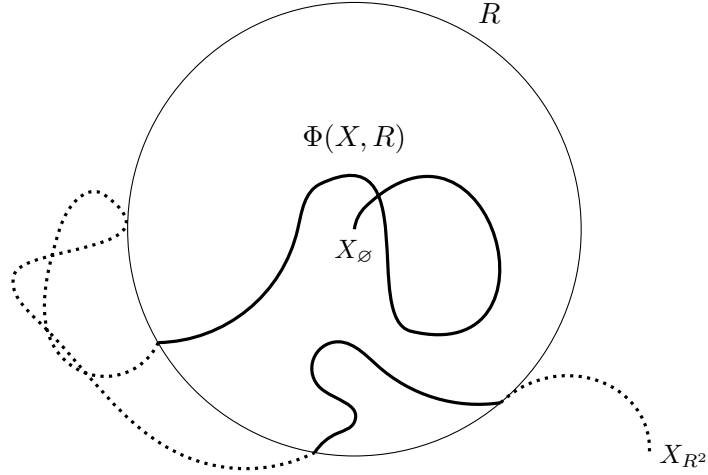

The following lemma provides an upper bound on the number of vertices of a branching random walk which fall in a ball of radius $R$. We need this result in order to control the capacity of the set $\Phi$ for $N$ independent branching random walks.

\begin{lemma}
  \label{lem:bound_z_moments}
  Let $R > 0$, and let $X$ be a random walk on $\mathbb{Z}^d$ indexed by an infinite critical tree starting from the origin, whose reproduction law $\mu$ has finite third moment.
  Let
  \begin{equation}
    \label{eq:def_jX}
    J_X(R) =  \sum_{u \in \mathcal{T}_-} \indic{X_u \in \mathrm{B}(R)}.
  \end{equation}
  There exist positive constants $c$ and $C$ such that,
  \[
    \textit{(i)}\ cR^4 \leq \E[J_X(R)] \leq CR^4, \quad\quad\quad \textit{(ii)}\ \E[J_X(R)^2] \leq CR^8.
  \]
\end{lemma}
\begin{proof}
  We only prove the result for a random walk indexed by an infinite invariant tree, but the proof works the same when the walk is indexed by Kesten's tree.
  
  \noindent\textit{Proof of (i).} The first part of the lemma boils down to the following computation,
  \[
    \E[J_X(R)] = \E\Bigg[\sum_{u\in\mathcal{T}_-} \indic{X_u \in \mathrm{B}(R)}\Bigg] = \sum_{x\in\mathrm{B}(R)} \E\Bigg[\sum_{u\in\mathcal{T}_-} \indic{X_u=x}\Bigg] = \sum_{x\in\mathrm{B}(R)} G(x) \stackrel{\eqref{eq:estim_green}}{\asymp} R^4.
  \]

  \medskip

  \noindent\textit{Proof of (ii).}
  Denote by $(S_n)_{n\in\mathbb{N}}$ the random walk induced by the spine. We use the decomposition of the past of the infinite tree using the spine and the adjoint critical trees located on the past. Letting $\varphi : x \mapsto \indic{x \in \mathrm{B}(R)}$, we have by using \eqref{eq:green_adjoint} for the last step,
  \begin{align}
    & \E\Bigg[\Bigg(\sum_{u\in\mathcal{T}_-}\indic{X_u \in \mathrm{B}(R)}\Bigg)^2\Bigg] = \E\Bigg[\Bigg(\sum_{u\in\mathcal{T}_-}\varphi(X_u)\Bigg)^2\Bigg] = \E\Bigg[\Bigg(\sum_{n\in\mathbb{N}}\sum_{u\in\widetilde{\mathcal{T}}_c^{(n)}} \varphi(S_n+X_u^{(n)})\Bigg)^2\Bigg] \nonumber \\
    =&\ \E\Bigg[\sum_{n\in\mathbb{N}}\Bigg(\sum_{u\in\widetilde{\mathcal{T}}_c^{(n)}} \varphi(S_n+X_u^{(n)})\Bigg)^2\Bigg]
       + \E\Bigg[\sum_{n_1 \neq n_2}\Bigg(\sum_{u\in\widetilde{\mathcal{T}}_c^{(n_1)}} \varphi(S_{n_1}+X_u^{(n_1)})\Bigg)\Bigg(\sum_{v\in\widetilde{\mathcal{T}}_c^{(n_2)}} \varphi(S_{n_2}+X_v^{(n_2)})\Bigg)\Bigg]. \label{eq:proof_bound_Jx_ii}
  \end{align}

  \noindent Concerning the first term in \eqref{eq:proof_bound_Jx_ii}, we note that,
  \begin{align*}
    \E\Bigg[\sum_{n\in\mathbb{N}}\Bigg(\sum_{u\in\widetilde{\mathcal{T}}_c^{(n)}} \varphi(S_n+X_u^{(n)})\Bigg)^2\Bigg]
    = \E\Bigg[\sum_{n\in\mathbb{N}} \E_{S_n}\Bigg[\Bigg(\sum_{u\in\widetilde{\mathcal{T}}_c} \varphi(X_u)\Bigg)^2\Bigg]\Bigg]
    = \sum_{x\in\mathbb{Z}^d} \E_x\Bigg[\Bigg(\sum_{u\in\widetilde{\mathcal{T}}_c} \varphi(X_u)\Bigg)^2\Bigg] g(x).
  \end{align*}
  Moreover,
  \begin{align*}
    \E_x\Bigg[\Bigg(\sum_{u\in\widetilde{\mathcal{T}}_c} \varphi(X_u)\Bigg)^2\Bigg]
    = \E_x\Bigg[\sum_{u,v \in \widetilde{\mathcal{T}}_c} \varphi(X_u)\varphi(X_v)\Bigg]
    = \E_x\Bigg[\sum_{w \in \widetilde{\mathcal{T}}_c}\ \sum_{\substack{u\neq v\in\widetilde{\mathcal{T}}_c \\ u\wedge v = w}} \varphi(X_u) \varphi(X_v) + \sum_{u\in\widetilde{\mathcal{T}}_c} \varphi(X_u)^2\Bigg].
  \end{align*}
  On the one hand, since $\mu$ has a finite third moment, recalling the notation $\nu$ from Section \ref{sec:notations},
  \begin{align*}
    \E_x\Bigg[\sum_{w \in \widetilde{\mathcal{T}}_c}\ \sum_{\substack{u\neq v\in\widetilde{\mathcal{T}}_c \\ u\wedge v = w}} \varphi(X_u) \varphi(X_v)\Bigg]
    &\leq \sum_{y_1,y_2\in\mathrm{B}(R)} \E_x\Bigg[\sum_{w\in\widetilde{\mathcal{T}}_c} \nu_w (\nu_w-1) g(y_1,X_w)g(y_2,X_w)\Bigg] \\
    &\lesssim \sum_{y_1,y_2\in\mathrm{B}(R)} \sum_{z \in \mathbb{Z}^d} g(y_1,z)g(y_2,z) \E_x\Bigg[\sum_{w\in\widetilde{\mathcal{T}}_c}\indic{X_w = z}\Bigg],
  \end{align*}
  which by \eqref{eq:green_adjoint} gives,
  \[
    \E_x\Bigg[\sum_{w \in \widetilde{\mathcal{T}}_c}\ \sum_{\substack{u\neq v\in\widetilde{\mathcal{T}}_c \\ u\wedge v = w}} \varphi(X_u) \varphi(X_v)\Bigg] \lesssim \sum_{y_1,y_2\in\mathrm{B}(R)} \sum_{z\in\mathbb{Z}^d} g(y_1,z)g(y_2,z)g(x,z).
  \]
  On the other hand, again by \eqref{eq:green_adjoint},
  \[
    \E_x\Bigg[\sum_{u\in \widetilde{\mathcal{T}}_c} \varphi(X_u)^2\Bigg] \lesssim \sum_{y\in\mathrm{B}(R)} g(x,y),
  \]
  which gives us for the first term of \eqref{eq:proof_bound_Jx_ii}, using Lemma \ref{lem:magic_ineq} at the end,
  \begin{align*}
    \E\Bigg[\sum_{n\in\mathbb{N}}\Bigg(\sum_{u\in\widetilde{\mathcal{T}}_c^{(n)}} \varphi(S_n+X_u^{(n)})\Bigg)^2\Bigg]
    &\lesssim \sum_{y\in\mathrm{B}(R)}\sum_{x\in\mathbb{Z}^d} g(x) g(x,y) + \sum_{y,z\in\mathrm{B}(R)} \sum_{x,t\in\mathbb{Z}^d} g(y,t)g(z,t)g(x,t)g(x) \\
    &\overset{\eqref{eq:G_g_convolution}}{\lesssim} \sum_{y\in\mathrm{B}(R)} G(y) + \sum_{y,z\in\mathrm{B}(R)}\sum_{t\in\mathbb{Z}^d} g(y,t)g(z,t)G(t) \\
    &\lesssim R^4 + R^2 \sum_{y\in\mathrm{B}(R)} (G * g)(0,y) \lesssim R^4 + R^2 \cdot R^6
      \lesssim R^8.
  \end{align*}

  \noindent Now for the second term of \eqref{eq:proof_bound_Jx_ii},
  \begin{align*}
    \E\Bigg[&\sum_{n_1 \neq n_2}\Bigg(\sum_{u\in\widetilde{\mathcal{T}}_c^{(n_1)}} \varphi(S_{n_1}+X_u^{(n_1)})\Bigg)\Bigg(\sum_{v\in\widetilde{\mathcal{T}}_c^{(n_2)}} \varphi(S_{n_2}+X_v^{(n_2)})\Bigg)\Bigg] \\
            &\asymp 2\E\Bigg[\sum_{x,y\in \mathrm{B}(R)} \sum_{w,w'\in\mathbb{Z}^d} g(w) g(w,w') g(w,x) g(w',y)\Bigg] \asymp \sum_{w\in\mathbb{Z}^d} g(w) \sum_{x\in\mathrm{B}(R)} g(x,z) \sum_{y\in\mathbb{Z}^d} G(w,y) \\
            &\lesssim R^8,
  \end{align*}
  using \eqref{eq:G_g_convolution} and $\sum_{y\in\mathrm{B}(R)} G(w,y) \lesssim R^4$ twice for the last inequality.
\end{proof}

\begin{lemma}
  \label{lem:phi_capacity_bound}
  There exist positive constants $c,C > 0$, such that for $N \geq 1$, $R > 0$, $(x_i)_{1 \leq i \leq N} \in \mathbb{Z}^d$, and $(X^i)_{1\leq i \leq N}$ a sequence of independent random walks indexed by infinite critical trees starting respectively from $(x_i)_{1\leq i\leq N}$,
  \begin{enumerate}[(i)]
  \item \label{bound_1_i} \[
      \E\bigg[\mathrm{BCap}\big(\Phi(\overline{X}_N, R)\big)\bigg] \leq CNR^4.
    \]
  \item \label{bound_1_ii} For $d \geq 9$,
    \[
      \E\bigg[\mathrm{BCap}\big(\Phi(\overline{X}_N, R)\big)\bigg] \geq\ c \min\{NR^4, R^{d-4}\}.
    \]
  \end{enumerate}
\end{lemma}
\begin{proof}[Proof of (i)]
  Fix a family of random walks $(X^i)_{1\leq i \leq N}$ starting respectively from $(x_i)_{1\leq i\leq N} \in \mathbb{Z}^d$, and indexed by infinite critical trees $(\mathcal{T}^{(i)})_{1\leq i\leq N}$.

  \noindent By applying Lemma \ref{lem:bound_z_moments} \textit{(i)} to each of these walks and a union bound, we get,
  \[
    \E\big[\#\Phi(\overline{X}_N,R)\big] \leq CNR^4,
  \]
  \noindent which implies, by subadditivity of the branching capacity,
  \[
    \E\big[\mathrm{BCap}(\Phi(\overline{X}_N,R))\big] \leq CNR^4.
  \]

  \bigskip

  \noindent\textit{Proof of (ii).} Fix a positive integer $n$, that we will choose later, and let $(S_t^{(i)})_{t\in\mathbb{N}}$ be the random walk induced by the spine of $\mathcal{T}^{(i)}$. We define the probability measure
  \begin{equation}
    \label{eq:def_nu}
    \nu(x) = \frac{1}{Y} \sum_{i=1}^N \sum_{t = n+1}^{2n} \indic{S^{(i)}_t\in\mathrm{B}(x_i,R)} \sum_{u \in \widetilde{\mathcal{T}}_c^{(t,i)}} \indic{X^i_u = x} \indic{x\in\mathrm{B}(x_i,R)},
  \end{equation}
  where $\widetilde{\mathcal{T}}_c^{(t,i)}$ denotes the $t$-th adjoint critical tree in the spinal decomposition of $\mathcal{T}^{(i)}_-$, and where the normalization factor $Y$ is defined by
  \[
    Y = \sum_{i=1}^N Y_i, \quad \text{with } Y_i = \sum_{t=n+1}^{2n}\indic{S^{(i)}_t\in\mathrm{B}(x_i,R)} \sum_{u\in\widetilde{\mathcal{T}}_c^{(t,i)}} \indic{X^i_u \in \mathrm{B}(x_i,R)}.
  \]
  Observe that the definition of this measure does not depend on the type of trees that we choose, since the attached trees at the root are not taken into account.

  We use Corollary 1.4 from Asselah, Schapira and Sousi \cite{ass2023} which provides a general lower bound on the branching capacity in terms of the inverse of an energy, under a third moment assumption on $\mu$. In particular, it shows that, for any probability measure $\rho$ supported on $\Phi(\overline{X}_N,R)$,
  \begin{equation}
    \label{eq:lower_bound_energy}
    \E[\mathrm{BCap}(\Phi(\overline{X}_N,R))] \gtrsim \E[\mathcal{E}(\rho)^{-1}],
  \end{equation}
  where the energy of the probability measure $\rho$ is defined to be
  \[
    \mathcal{E}(\rho) = \sum_{x,y\in\mathbb{Z}^d} G(x,y) \rho(x)\rho(y).
  \]
  Specifying to $\nu$ as defined in \eqref{eq:def_nu} it becomes,
  \begin{equation}
    \label{eq:energy_definition}
    \mathcal{E}(\nu) = \frac{1}{Y^2} \sum_{1\leq i,j \leq N} \sum_{n+1 \leq s,t \leq 2n} \indic{\substack{S^{(i)}_t\in\mathrm{B}(R) \\ S^{(j)}_s\in\mathrm{B}(R)}} \sum_{\substack{X^i_u \in \widetilde{\mathcal{T}}_c^{(t,i)} \\ X^j_v \in \widetilde{\mathcal{T}}_c^{(s,j)}}} G(X^i_u, X^j_v) \indic{\substack{X^i_u \in \mathrm{B}(x_i,R) \\ X^j_v \in \mathrm{B}(x_j,R)}}.
  \end{equation}
  In order to bound this quantity, we work on a particular event where $Y$ is large enough to offer an interesting lower bound. We will use the Paley-Zigmund inequality, which states that for any $\theta \in (0,1)$,
  \[
    \P(Y \geq \theta \E[Y]) \geq (1-\theta)^2 \frac{\E[Y]^2}{\E[Y^2]}.
  \]
  Let $A = \big\{Y \geq \E[Y]/2\big\}$. By Cauchy-Schwarz inequality,
  \begin{equation}
    \label{eq:bound_cauchy_schwarz}
    \E[\mathcal{E}(\nu)^{-1}] \geq \E[\mathcal{E}(\nu)^{-1}\sindic{A}] \geq \P(A)^2  \ \E[\mathcal{E}(\nu)\sindic{A}]^{-1}.
  \end{equation}

  \noindent It amounts to compute the first two moments of $Y$. Using \eqref{eq:green_adjoint} and the central limit theorem, we get
  \[
    \E[Y_i]
    = \sum_{t=n+1}^{2n}\ \E\Bigg[\indic{S^{(i)}_t\in\mathrm{B}(R)}\sum_{u\in\widetilde{\mathcal{T}}_c^{(t,i)}} \indic{X^i_u \in \mathrm{B}(x_i,R)}\Bigg]
    \asymp \sum_{t = n+1}^{2n}\P(S^{(i)}_t \in \mathrm{B}(R))\ \E\Bigg[\sum_{x\in\mathrm{B}(x_i,R)}g(S^{(i)}_t,x)\Bigg] 
    \asymp nR^2.
  \]
  For the second moment, we have
  \begin{align*}
    \E[Y_i^2]
    &\leq \sum_{n+1 \leq s,t \leq 2n} \E\Bigg[\sum_{\ \ u \in \widetilde{\mathcal{T}}_c^{(t,i)}} \sum_{\ v \in \widetilde{\mathcal{T}}_c^{(s,i)}}\indic{\substack{X^i_u \in \mathrm{B}(x_i,R) \\ X^i_v\in \mathrm{B}(x_i,R)}} \Bigg] \\
    &\lesssim \sum_{n+1\leq s,t \leq 2n}\E\Bigg[\Bigg(\sum_{x\in\mathrm{B}(x_i,R)} g(S^{(i)}_t,x)\Bigg)\Bigg(\sum_{x\in\mathrm{B}(x_i,R)}g(S_s^{(i)},y)\Bigg)\Bigg] 
      \lesssim R^4 n^2.
  \end{align*}
  Those calculations allow us to apply the Paley-Zigmund inequality. Indeed,
  \[
    \E[Y] \asymp nNR^2 \quad\text{and}\quad \E[Y^2] \lesssim N\E[Y_1^2] + \E[Y]^2 \lesssim Nn^2R^4 + \E[Y]^2 \lesssim \E[Y]^2.
  \]
  We deduce that there exists a constant $c > 0$ which does not depend on $R$ such that,
  \[
    \P(A) = \P\left(Y \geq \frac{1}{2}\E[Y]\right) \geq c > 0.
  \]
  Thus,
  \begin{equation}
    \label{eq:ineq_energy}
    \E\Big[\mathcal{E}(\nu)^{-1}\Big] \geq c^2 \E\Big[\mathcal{E}(\nu)\sindic{A}\Big]^{-1}.
  \end{equation}

  \noindent Concerning the expectation on the right-hand side of \eqref{eq:ineq_energy}, first note that for all $i \neq j$,
  \begin{align*}
    \sum_{n+1 \leq s,t \leq 2n} \E&\Bigg[\sum_{\ \ u \in \widetilde{\mathcal{T}}_c^{(t,i)}}\sum_{\ \ v \in \widetilde{\mathcal{T}}_c^{(s,j)}} G(X^i_u, X^j_v) \indic{\substack{X^i_u \in \mathrm{B}(x_i,R) \\ X^j_v \in \mathrm{B}(x_j,R)}}\Bigg]\\ 
    &= \sum_{n+1 \leq s,t \leq 2n}\sum_{\substack{x\in\mathrm{B}(x_i,R)\\y\in\mathrm{B}(x_j,R)}} G(x,y)\ \E\Bigg[\sum_{\ u \in \widetilde{\mathcal{T}}_c^{(t,i)}} \indic{X^i_u = x}\Bigg]\ \E\Bigg[\sum_{\ v \in \widetilde{\mathcal{T}}_c^{(s,j)}} \indic{X^i_v = y}\Bigg] \\
    &\asymp\sum_{\substack{x\in\mathrm{B}(x_i,R)\\y\in\mathrm{B}(x_j,R)}} G(x,y)\ \E\Bigg[\Bigg(\sum_{t=n+1}^{2n}g(S_t^{(i)},x)\Bigg) \Bigg(\sum_{s=n+1}^{2n}g(S_s^{(j)},y)\Bigg)\Bigg].
  \end{align*}
  Furthermore, according to \cite{rs2012} (4.2), we have uniformly in $x \in\mathbb{Z}^d$, for $d \geq 3$,
  \begin{equation}
    \label{eq:ineq_sum_green_srw}
    \sum_{t=n+1}^{2n} \E\big[g(S_t,x)\big] \leq Cn^{2-\frac{d}{2}}.
  \end{equation}
  Thus, for $i\neq j$,
  \begin{align}
    \sum_{n+1 \leq s,t \leq 2n} \E\Bigg[\indic{\substack{S_t^{(i)}\in\mathrm{B}(R) \\ S_s^{(j)}\in\mathrm{B}(R)}}\sum_{\substack{u \in \widetilde{\mathcal{T}}_c^{(t,i)} \\ v \in \widetilde{\mathcal{T}}_c^{(s,j)}}} G(X^i_u, X^j_v) \indic{\substack{X^i_u \in \mathrm{B}(x_i,R) \\ X^j_v \in \mathrm{B}(x_j,R)}}\Bigg]
    &\lesssim \ \ n^{4-d}\sum_{\substack{x\in\mathrm{B}(x_i,R)\\y\in\mathrm{B}(x_j,R)}} G(x,y) \nonumber\\
    &\lesssim  n^{4-d}R^{4+d}, \label{eq:ineqjlast}
  \end{align}
  using that $\sum_{\substack{x\in\mathrm{B}(x_i,R)\\y\in\mathrm{B}(x_j,R)}} G(x,y) \lesssim \sum_{x\in\mathrm{B}(x_i,R)} R^4 \lesssim R^{d+4}$.

  \noindent Now if $i = j$, we consider two cases. First, if $s \neq t$, assume without loss of generality that $t < s$, and denote by $\widetilde{X}$ and $\widetilde{X}'$ two independent random walks starting from $0$, independent from $(S_n)_{n\in\mathbb{N}}$, and indexed by adjoint critical trees, respectively $\widetilde{\mathcal{T}}_c^{(t)}$ and $\widetilde{\mathcal{T}}_c^{(s)}$. Let also $S'_n = S_{m+n} - S_m$. Then we have,
  \begin{align*}
    \E\Bigg[\indic{\substack{S_t \in \mathrm{B}(R) \\ S_s \in \mathrm{B}(R)}}\sum_{\substack{u \in \widetilde{\mathcal{T}}_c^{(t)} \\ v\in\widetilde{\mathcal{T}}_c^{(s)}}} G(X_u,X_v) \indic{\substack{X_u \in \mathrm{B}(R) \\ X_v \in \mathrm{B}(R)}}\Bigg]
    &=\ \E\Bigg[\indic{\substack{S_t \in \mathrm{B}(R) \\ S_s \in \mathrm{B}(R)}}\sum_{\substack{u \in \widetilde{\mathcal{T}}_c^{(t)} \\ v\in\widetilde{\mathcal{T}}_c^{(s)}}} G(\widetilde{X}_u + S_t,\widetilde{X}'_v + S_s)\indic{\substack{\widetilde{X}_u+S_t \in \mathrm{B}(R) \\ \widetilde{X}'_v + S_s \in \mathrm{B}(R)}}\Bigg] \\
    &\leq\ \E\Bigg[\indic{S'_{s-t} \in\mathrm{B}(2R)}\sum_{\substack{u \in \widetilde{\mathcal{T}}_c^{(t)} \\ v\in\widetilde{\mathcal{T}}_c^{(s)}}} G(\widetilde{X}_u-\widetilde{X}'_v, S'_{s-t}) \indic{\substack{\widetilde{X}_u \in \mathrm{B}(2R) \\ \widetilde{X}'_v \in \mathrm{B}(2R)}}\Bigg] \\
    &\lesssim \E\Bigg[\sum_{x,y \in \mathrm{B}(2R)}\sum_{z \in \mathbb{Z}^d} g(x) g(y) g(z-y)G(x-y,z)\Bigg].
  \end{align*}
  \noindent By translating $z$ by $y$ and using the fact that $g * g \asymp G$ by \eqref{eq:G_g_convolution},
  \begin{align*}
    \E\Bigg[\sum_{x,y \in \mathrm{B}(2R)}\sum_{z \in \mathbb{Z}^d} g(x) g(y) g(z-y)G(x-y,z)\Bigg]
    &\lesssim \ \E\Bigg[\sum_{x \in \mathrm{B}(2R)} \sum_{y,z \in \mathbb{Z}^d} g(x) g(y) g(z-y)G(x,z)\Bigg] \\
    &\lesssim \ \E\Bigg[\sum_{x \in \mathrm{B}(2R)} \sum_{z \in \mathbb{Z}^d} g(x) G(z)G(x,z)\Bigg].
  \end{align*}
  \noindent Since $d \geq 9$,
  \[
    \sum_{\substack{z \in \mathbb{Z}^d}} G(x,z)G(z) \lesssim \sum_{z \in \mathbb{Z}^d} \frac{1}{1 + \|x-z\|^{d-4}}\frac{1}{1 + \|z\|^{d-4}} \lesssim \sum_{r \in \mathbb{N}^*} \frac{r^{d-1}}{r^{2d-8}} = \sum_{r\in\mathbb{N}^*} \frac{1}{r^{d-7}} < +\infty.
  \]
  Thus,
  \begin{equation}
    \label{eq:ieqjsneqt}
    \E\Bigg[\sum_{n+1\leq t < s \leq 2n}\indic{\substack{S_t \in \mathrm{B}(R) \\ S_s \in \mathrm{B}(R)}}\sum_{\substack{u \in \widetilde{\mathcal{T}}_c^{(t)} \\ v\in\widetilde{\mathcal{T}}_c^{(s)}}} G(X_u,X_v) \indic{\substack{X_u \in \mathrm{B}(R) \\ X_v \in \mathrm{B}(R)}}\Bigg]
    \lesssim\ n\E\Bigg[\sum_{x\in\mathrm{B}(2R)} g(x)\Bigg] \lesssim nR^2.
  \end{equation}

  \noindent Now consider the case $s = t$ in \eqref{eq:ineq_energy}. If we sum over the most recent common ancestor of $u$ and $v$ in $\widetilde{\mathcal{T}}_c^{(t)}$,
  \[
    \E\Bigg[\sum_{u \in \widetilde{\mathcal{T}}_c^{(t,i)}} \sum_{v \in \widetilde{\mathcal{T}}_c^{(t,i)}}\indic{\substack{X^i_u = x\\X^i_v = y}}\Bigg] \asymp \E\Bigg[\sum_{z\in\mathbb{Z}^d} g(S_t,z) g(y,z) g(x,z)\Bigg],
  \]
  and consequently,
  \begin{align}
    \sum_{x,y\in\mathrm{B}(x_i,R)}&\sum_{t=n+1}^{2n} G(x,y)\ \E\Bigg[\sum_{u,v \in \widetilde{\mathcal{T}}_c^{(t,i)}}\indic{\substack{X^i_u = x\\X^i_v = y}}\Bigg]
    \asymp \sum_{x,y\in\mathrm{B}(x_i,R)} G(x,y)\ \E\Bigg[\sum_{z \in \mathbb{Z}^d}\Bigg(\sum_{t=n+1}^{2n}g(S_t,z)\Bigg) g(y,z) g(x,z)\Bigg] \nonumber\\
                                  & \stackrel{\eqref{eq:ineq_sum_green_srw}}{\lesssim} n^{2-\frac{d}{2}} \Bigg(\sum_{x,y\in\mathrm{B}(x_i,R)} G(x,y)^2\Bigg)
                                    \stackrel{d\geq 6}{\lesssim} n^{2-\frac{d}{2}} \Bigg(\sum_{x,y\in\mathrm{B}(x_i,R)} g(x,y)\Bigg)
                                    \lesssim n^{2-\frac{d}{2}} R^{d+2}.\label{eq:ieqjseqt}
  \end{align}

  \noindent Altogether, with \eqref{eq:ineqjlast}, \eqref{eq:ieqjsneqt} and \eqref{eq:ieqjseqt}, we get
  \[
    \sum_{1\leq i,j \leq N} \sum_{n+1 \leq s,t \leq 2n} \sum_{\substack{u \in \widetilde{\mathcal{T}}_c^{(t,i)} \\ v \in \widetilde{\mathcal{T}}_c^{(s,j)}}} G(X^i_u, X^j_v) \indic{\substack{u \in \mathrm{B}(x_i,R) \\ v \in \mathrm{B}(x_j,R)}} \lesssim N^2 R^{4+d}n^{4-d} + N(nR^2 + R^{d+2}n^{2-\frac{d}{2}}).
  \]
  Setting $n = \left\lfloor R^2/2\right\rfloor$ yields
  \[
    \sum_{1\leq i,j \leq N} \sum_{n+1 \leq s,t \leq 2n} \sum_{\substack{u \in \widetilde{\mathcal{T}}_c^{(t,i)} \\ v \in \widetilde{\mathcal{T}}_c^{(s,j)}}} G(X^i_u, X^j_v) \indic{\substack{u \in \mathrm{B}(x_i,R) \\ v \in \mathrm{B}(x_j,R)}} \lesssim N^2 R^{12-d} + NR^4,
  \]
  and finally,
  \[
    \E\Big[\mathrm{BCap}(\Phi(\overline{X}_N,R))\Big] \gtrsim \frac{N^2R^8}{N^2R^{12-d} + NR^4} \gtrsim \min\{NR^4, R^{d-4}\}.
  \]
\end{proof}

Next we define a map $\Psi$ as follows. Let $A$ be a finite set of $\mathbb{Z}^d$. For a point measure $\omega = \sum_{i\in\mathbb{N}} \delta_{w_i}$ with $w_i \in W^*$, denote by $N_A(\omega)$ the number of trajectories of $\mathrm{Supp}(\omega)$ that intersect $A$. Let ${\overline{X}_A(\omega)=(X^i)_{1\leq i \leq N_A(\omega)}}$ be the family of branching random walks associated to the trajectories $(w_i)_{1\leq i\leq N_A(\omega)}$ parametrized in such a way that $w_i(0) \in A$ and for all $j < 0$, $w_i(j) \notin A$. Recall the definition of $\Phi$ in \eqref{eq:def_phi} and let now
\begin{equation}
  \label{eq:def_psi}
  \Psi(\omega, A, R) = \Phi(\overline{X}_A(\omega), R) = \bigcup_{i=1}^{N_A(\omega)}\left(\left\{X^i_u,\, u\in\mathcal{T}_-^{(i)},\, n_u \leq R^2\right\}\cap\mathrm{B}\left(X^i_\varnothing,R\right)\right).
\end{equation}

\begin{figure}[H]\label{fig:ex_psi}
  \centering
  \begin{tikzpicture}[scale=1.5,use Hobby shortcut]
    \draw[left] (-0.1,0.5) node {$A$};
    \draw ([closed]0,0) .. (1,-1) .. (2,-0.5) .. (2.25,0) .. (3,2);
    
    \draw[very thick, dotted] (-2.8,1.5) .. (-2.5,1.5) .. (-2,2) .. (-1.5,1.8) .. (-1,2.3) .. (0,3.5);

    \draw (2.65,2.5) circle (1);
    \draw[very thick, dotted] (1,1) .. (1,1.5) .. (1.2,1.6) .. (2.65,2.5);
    \draw[very thick] (2.65, 2.5) .. (2.8, 3.2) .. (2.4, 3.3) .. (2.1,3.35);
    \draw[very thick, dotted] (2.1,3.35) .. (2, 3.4) .. (2.4, 3.8) .. (2.8,3.7) .. (3.3,3.5) .. (4,3.2) .. (3.5,3);
    \draw[very thick] (3.5,3) .. (3.2,2.8) .. (3.4,2.5) .. (3.2,2.3);
    \draw[above] (2.65, 2.5) node {\scriptsize$X^{(1)}_\varnothing$};
    \draw (3.3,2.1) node {\scriptsize$X_{R^2}^{(1)}$};

    \draw (2,-0.5) circle (1);
    \draw[very thick] (2,-0.5) .. (2.3, -0.9) .. (2.2, -1) .. (2.5,-1.2) .. (2.7,-0.5) .. (2.8,0.1);
    \draw[very thick, dotted] (2.8,0.1) .. (3,0.6) .. (3.7,0.4) .. (3.6,-0.6) .. (4.2, -0.2);
    \draw[very thick, dotted] (2,-0.5) .. (1.5,-0.6) .. (1.3,-0.4) .. (0,-1.5);
    \draw(1.8,-0.3) node {\scriptsize$X^{(2)}_\varnothing$};

    \draw[left] (2,3.3) node {\scriptsize$R$};
    \draw[right] (2.5,-1.37) node {\scriptsize$R$};
  \end{tikzpicture}
  \caption{Example of a set $\Psi(\omega, A, R)$.}
\end{figure}
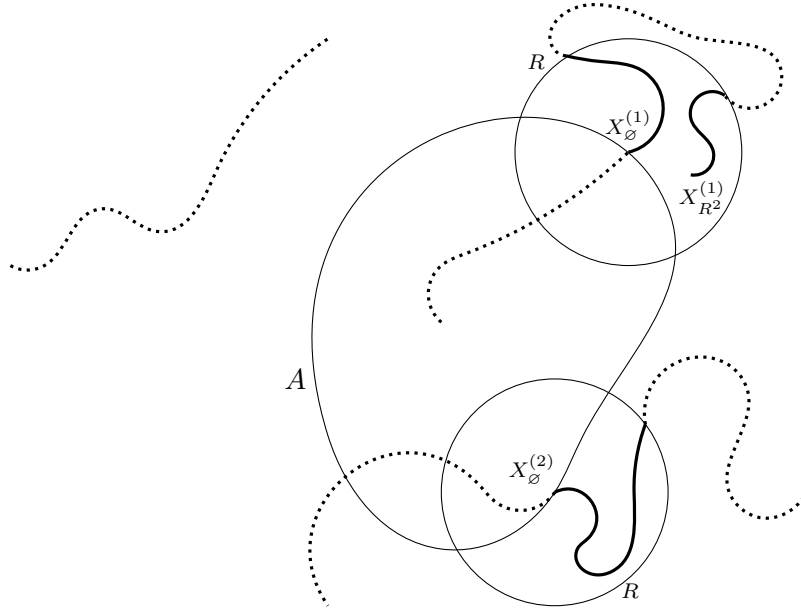

\begin{lemma}
  \label{lem:lower_bound_psi}
  Let $\sigma_u$ be a random point measure with distribution $\mathrm{Poi}(u,W^*)$. Then there exists $c > 0$, such that for all finite subsets $A$ of $\mathbb{Z}^d$, and for all positive $R \geq 1$, we have
  \[
    \E\Big[\mathrm{BCap}(\Psi(\sigma_u,A,R))\Big] \geq c\min\{u\mathrm{BCap}(A)R^4, R^{d-4}\}.
  \]
\end{lemma}
\begin{proof}
  The independence between $N_A = N_A(\sigma_u)$ and $(X_i)_{1\leq i \leq N_A}$, and Lemma \ref{lem:phi_capacity_bound} imply that
  \[
    \E\Big[\mathrm{BCap}(\Psi(\sigma_u, A, R))\Big] \geq c \E\Big[\min(N_AR^4,R^{d-4})\Big].
  \]
  Recall that $N_A$ follows a Poisson law with parameter $u\mathrm{BCap}(A)$. Consider two cases, either $u\mathrm{BCap}(A) \leq 1$, then
  \begin{align*}
    \E\Big[\min(N_AR^4,R^{d-4})\Big]
    &\geq \E\Big[\indic{N_A\geq 1} \min(N_AR^4,R^{d-4})\Big] \geq R^4 \P(N_A\geq 1) = R^4 (1-e^{-u\mathrm{BCap}(A)}) \\
    &\hspace{-0.75cm}\stackrel{u\mathrm{BCap}(A)\leq 1}{\leq} R^4 u\mathrm{BCap}(A)/2,
  \end{align*}
  or $u\mathrm{BCap}(A) > 1$, and then
  \[
    \E\Big[\min(N_AR^4,R^{d-4})\Big] \geq \min(R^4u\mathrm{BCap}(A)/2, R^{d-4}) \P(N_A \geq u\mathrm{BCap}(A)/2).
  \]
  Let $\lambda = \E[N_A] = u\mathrm{BCap}(A)$. Then $\E[N_A^2] = \lambda^2 + \lambda$, and we conclude by using the Paley-Zigmund inequality,
  \[\
    \P(N_A \geq u\mathrm{BCap}(A)/2) \geq \frac{\lambda^2}{4(\lambda^2+\lambda)} \geq \frac{1}{8}.
  \]
\end{proof}

Given $0 < r < R \leq \infty$, and $\sigma_u$ a Poisson point process with distribution $\mathrm{Poi}(u,W^*)$, we write $\sigma^u_r$ for the restriction of $\sigma_u$ to the set of trajectories that intersect $\mathrm{B}(r)$, and $\sigma^u_{r,R}$ for the restriction of $\sigma_u$ to the set of trajectories that intersect $\mathrm{B}(R)$ but not $\mathrm{B}(r)$. By general properties of Poisson point processes, $\sigma_r^u$ and $\sigma_{r,R}^u$ are independent, and we have $\sigma_u = \sigma_r^u + \sigma_{r,\infty}^u$.

\begin{lemma}
  \label{lem:lower_bound_psi_r_inf}
  There exist positive $c$ and $C$ such that for every finite subset $A$ of $\mathbb{Z}^d$, and for all $r,R$ with $0 < r < R$,
  \[
    \E\Big[\mathrm{BCap}(\Psi(\sigma^u_{r,\infty},A,R))\Big] \geq c\min\left(u\mathrm{BCap}(A)R^4, R^{d-4}\right) - Cur^{d-4}R^4.
  \]
\end{lemma}
\begin{proof}
  By subadditivity of the capacity,
  \[
    \E\Big[\mathrm{BCap}(\Psi(\sigma^u_{r,\infty},A,R))\Big] \geq \E\Big[\mathrm{BCap}(\Psi(\sigma_u,A,R))\Big] - \E\Big[\mathrm{BCap}(\Psi(\sigma^u_r,A,R))\Big].
  \]
  For the first term we use Lemma \ref{lem:lower_bound_psi}, and for the second one note that $|\mathrm{Supp}(\sigma_r)| = N_{\mathrm{B}(r)}$ follows a Poisson law with parameter $u\mathrm{BCap}(\mathrm{B}(r)) \asymp u r^{d-4}$, and thus by Lemma \ref{lem:phi_capacity_bound}\eqref{bound_1_i} and \eqref{eq:bcap_ball},
  \[
    \E\Big[\mathrm{BCap}(\Psi(\sigma_r,A,R))\Big] \lesssim R^4\E[N_{\mathrm{B}(r)}] \asymp u R^4 r^{d-4},
  \]
  which concludes the proof of the lemma.
\end{proof}

\subsubsection{Construction of visible sets}
\label{sec:visible_sets}
We now iterate the application $\Psi$ introduced in \eqref{eq:def_psi}. Starting from $\Phi(X,R)$ with a random walk $X$, indexed by an infinite critical tree, we obtain a sequence of subsets $(A^{(s)})_{s\geq 1}$ of $\mathbb{Z}^d$ whose capacity increases by Lemma \ref{lem:lower_bound_psi}. The study of the branching capacity of these sets shows a change of behavior and a saturation phenomenon: when $s \geq s_d$ the set $A^{(s)}$ typically fills a ball, and thus its hitting probability is lower bounded by a positive constant (see Lemma \ref{lem:lower_bound_capacity_final}).

Let $X$ be a random walk on $\mathbb{Z}^d$ starting from $x \in \mathbb{Z}^d$ and indexed by an infinite invariant tree $\mathcal{T}$. We think of it as an element in the support of the branching interlacement. Let $\sigma^{(2)}, \sigma^{(3)},\dots$ be independent random point measures with distribution $\mathrm{Poi}(u,W^*)$, and also independent from $X$. Denote by $\P_x$ their joint law.

Let $r < R$ with $R > \|x\|$.  Let $\xi_{\tau_R}$ be the vertex on the spine of $\mathcal{T}$ at the exit time $\tau_R$ of $\mathrm{B}(R)$ as defined in \eqref{eq:exit_time_br}, and let $Y$ be the random walk indexed by the tree $\mathcal{T}^*$ of descendants of $\xi_{\tau_R}$ in $\mathcal{T}$. Recall the definition of $\Psi$ in \eqref{eq:def_psi} and $\Phi$ in \eqref{eq:def_phi}. We define the following sequence $\big(A^{(s)}(r,R)\big)_{s\geq 1}$ of random subsets of $\mathbb{Z}^d$,
\begin{equation}
  \label{eq:def_A_1}
  A^{(1)}(r,R) = A^{(1)}(R) = \Phi(Y,R/2) = \left\{Y_u,\, u\in\mathcal{T}^*_-,\, n_u \leq R^2/4\right\} \cap \mathrm{B}\left(Y_\varnothing,R/2\right),
\end{equation}
and for $s \geq 2$,
\begin{equation}
  \label{eq:def_A_s}
  A^{(s)}(r,R) = \Psi\left(\sigma_{r,\infty}^{(s)},\, A^{(s-1)}(r,R), R/2\right) = \Psi\left(\sigma_{r,sR}^{(s)},\, A^{(s-1)}(r,R), R/2\right).
\end{equation}
Note that $A^{(1)}(r,R)$ does not depend on $r$, we only make it appear for consistency in the notation, to not distinguish between the cases $s = 1$ and $s \geq 2$ later. We consider parts of ranges of walks in the interlacement that do not hit $\mathrm{B}(r)$ (to guarantee some independence later), bounded in space and time (to efficiently control their capacity, thanks to the previous lemmas). The factor $1/2$ in front of $R$ will only be useful in Section \ref{sec:upper_bound_k} to prove the upper bound on the $k$ points connection probability.

\begin{figure}[H]
  \centering
  \begin{subfigure}{0.49\textwidth}
    \centering
    \begin{tikzpicture}[scale=1.3,use Hobby shortcut]
      \draw[above] (0,0) node {\scriptsize$0$};
      \draw (0,0) node[cross=2pt,rotate=45]{};
      \draw (0,0) circle (2);
      \draw[above left] (-1.4,1.4) node {\scriptsize$R$};
      \draw[very thick] (-1.29,-1.53) .. (-1.7,-1.3) .. (-1.8, -1.2) .. (-2.3, -1.5) .. (-2.7, -1.6) .. (-3, -2.2) .. (-3.1, -2.35);
      \draw[very thick, dotted] (-3.1, -2.35) .. (-3.2, -2.9) .. (-3, -3.4) .. (-2.7, -3.2) .. (-2.5, -3.4) .. (-2.3, -3.2);
      \draw[very thick] (-2.3, -3.2) .. (-2.2, -3) .. (-2, -2.7) .. (-1.8, -2.4) .. (-1.6, -2.7) .. (-1.2, -2.4) .. (-1, -2.7) .. (-0.5, -3.2) .. (-0.3, -3) .. (-0.4, -2.8);
      \draw[dotted] (-1.29,-1.53) circle (2);
      \draw[above left] (-3.17,-0.85) node {\scriptsize$R$};
      \draw (-1.8,-2) node {\scriptsize$A^{(1)}(r,R)$};
    \end{tikzpicture}
    \caption{$A^{(1)}(r,R)$}
  \end{subfigure}
  \hfill
  \begin{subfigure}{0.49\textwidth}
    \centering
    \begin{tikzpicture}[scale=1.3,use Hobby shortcut]
      \draw[above] (0,0) node {\scriptsize$0$};
      \draw (0,0) node[cross=2pt,rotate=45]{};
      \draw (0,0) circle (0.75);
      \draw[above right] (0.525,0.525) node {\scriptsize$r$};

      \draw[very thick,dashed] (-1.29,-1.53) .. (-1.7,-1.3) .. (-1.8, -1.2) .. (-2.3, -1.5) .. (-2.7, -1.6) .. (-3, -2.2) .. (-3.1, -2.35);
      \draw[very thick,dashed] (-2.3, -3.2) .. (-2.2, -3) .. (-2, -2.7) .. (-1.8, -2.4) .. (-1.6, -2.7) .. (-1.2, -2.4) .. (-1, -2.7) .. (-0.5, -3.2) .. (-0.3, -3) .. (-0.4, -2.8);

      \draw[dotted] (-1.8,-1.2) circle (2);
      \draw[very thick] (-1.8,-1.2) .. (-1.5,-1) .. (-1.3, -0.5) .. (-1.7, 0) .. (-1.9,-0.4) .. (-2.5,-0.6) .. (-2.7, -1) .. (-3.2, -0.6) .. (-3.3, -0.1) .. (-3,0.2) .. (-2.9,0.5);
      \draw[very thick, dotted] (-2.9,0.5) .. (-3,0.8) .. (-3.4, 0.7) .. (-3.4, 0.4) .. (-3.9, 0.6) .. (-4,0.8);

      \draw[dotted] (-1.2,-2.4) circle (2);
      \draw[very thick] (-1.2,-2.4) .. (-1.2,-2) .. (-0.8,-2.2) .. (-0.6, -2.5) .. (-0.2, -2) .. (0.2, -2.1) .. (0.3, -2.4) .. (0.2, -2.7) .. (-0.1, -3.1) .. (-0.3, -3.5) .. (-0.3, -3.9) .. (-0.6, -3.8) .. (-1, -4);

      \draw[dotted] (-2,-2.7) circle (2);
      \draw[very thick] (-2,-2.7) .. (-2.3,-2.7) .. (-2.5,-3) .. (-2.8, -3.3) .. (-3.2, -3) .. (-3,-3.5) .. (-3, -4) .. (-3.2, -4.3);
      \draw[very thick, dotted] (-3.2,-4.3) .. (-3.5, -4.7) .. (-3, -4.8) .. (-2.6, -4.8);

      \draw[very thick, dotted] (-0.5,1) .. (-0.7,0.5) .. (-0.5, -1) .. (-1,-1.3) .. (-1.4, -1.3) .. (-1.6, -1.5) .. (-1.9, -1.5) .. (-2.3, -1.4) .. (-2.8, -1.5) .. (-3.4, -1.8) .. (-3.7, -1.5) .. (-4,-1) .. (-4.4, -0.8);

      \draw (-2.2,-2) node {\scriptsize$A^{(1)}(r,R)$};
      \draw (0.2,-4) node {\scriptsize$A^{(2)}(r,R)$}; 
    \end{tikzpicture}
    \caption{$A^{(2)}(r,R)$}
  \end{subfigure}
  \caption{Example of the construction of $A^{(s)}(r,R)$.}
  \label{fig:A_s}
\end{figure}
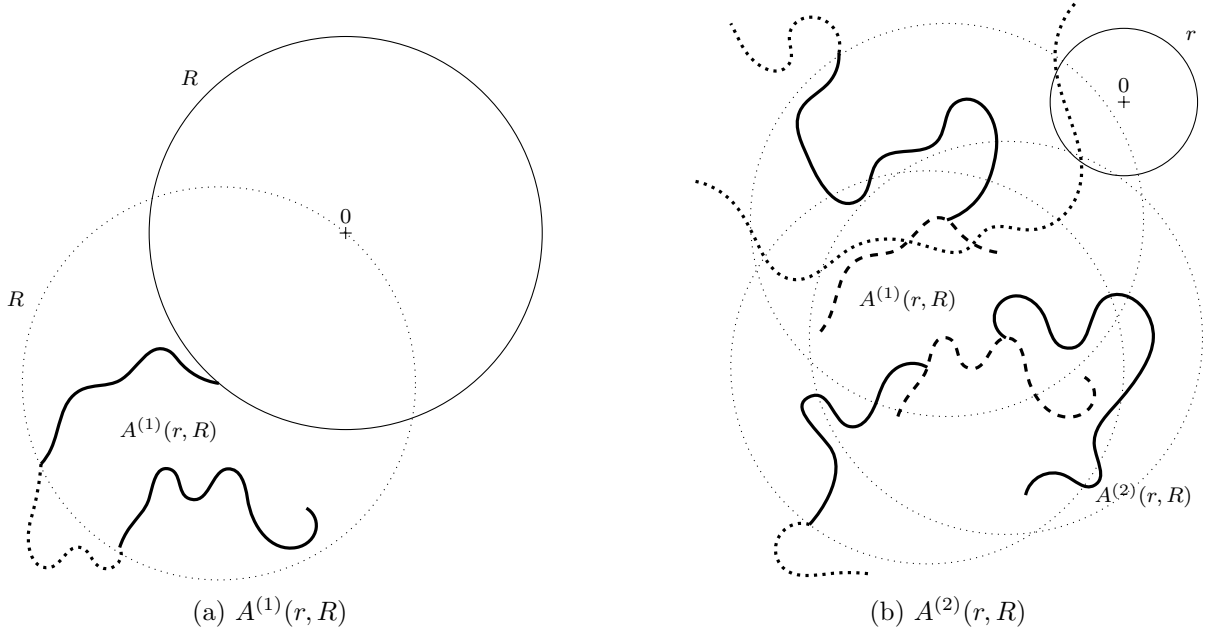

\begin{lemma}
  \label{lem:bcap_moments_bound}
  Let $s$ be a positive integer. There exists a finite constant $C_s = C(u,d,s)$ such that for all positive $r < R$, and $x \in \mathrm{B}(R)$,
  \[
    \begin{array}[H]{c l}
      \textit{(i)} & \E_x\Big[\mathrm{BCap}\left(A^{(s)}(r,R)\right)\Big] \leq C_s R^{\min(d-4,4s)},\\[0.5cm]
      \textit{(ii)} & \E_x\Big[\mathrm{BCap}\left(A^{(s)}(r,R)\right)^2\Big] \leq C_s R^{2\min(d-4,4s)}.
    \end{array}
  \]
\end{lemma}
\begin{proof}[Proof of (i)]
  We need to show that $\E_x\Big[\mathrm{BCap}\left(A^{(s)}(r,R)\right)\Big]$ is bounded by both $R^{d-4}$ and $R^{4s}$ up to a multiplicative constant. On the one hand, since $A^{(s)}(r,R)$ is a subset of $\mathrm{B}((s+1)R)$, we have by the monotonicity of the branching capacity, almost surely
  \begin{equation}
    \label{eq:bcap_monotonicity}
    \mathrm{BCap}\left(A^{(s)}(r,R)\right) \lesssim R^{d-4}.
  \end{equation}
  To simplify the notation, we write $N$ for $N_{A^{(s-1)}(r,R)}(\sigma_{r,\infty}^{(s)})$.
  Let $(X^i)_{1 \leq i \leq N}$, be the random walks of $\sigma_{r,\infty}^{(s)}$ that intersect $A^{(s-1)}(r,R)$.
  On the other hand, by the subadditivity of the capacity and the definition of $J_X$ in \eqref{eq:def_jX}, almost surely,
  \begin{equation}
    \label{eq:subadditivity_capacity}
    \mathrm{BCap}(A^{(s)}(r,R)) \leq \mathrm{BCap}(\{0\}) \sum_{i=1}^{N} J_{X^i}(R).
  \end{equation}
  Recall also that $N$ is almost surely smaller than the number of walks hitting $A^{(s-1)}(r,R)$ in $\sigma^{(s)}$ which follows a Poisson law with parameter $u\mathrm{BCap}({A^{(s-1)}(r,R)})$. Therefore, by independence between $N$ and the $(X^i)_{1\leq i \leq N}$, and Lemma \ref{lem:bound_z_moments},
  \[
    \E_x\Big[\mathrm{BCap}\left(A^{(s)}(r,R)\right)\Big] \lesssim R^4\E[N] \lesssim uR^4\ \E_x\Big[\mathrm{BCap}\left(A^{(s-1)}(r,R)\right)\Big].
  \]
  The result follows by induction.

  \medskip

  \noindent\textit{Proof of (ii).} On the one hand, by \eqref{eq:bcap_monotonicity},
  \[
    \mathrm{BCap}\left(A^{(s)}(r,R)\right)^2 \lesssim R^{2(d-4)}.
  \]
  On the other hand, if we write as before $N$ for $N_{A^{(s-1)}(r,R)}(\sigma_{r,\infty}^{(s)})$, then by \eqref{eq:subadditivity_capacity},
  \[
    \E_x\left[\mathrm{BCap}(A^{(s)}(r,R))^2\right] \leq \E_x\Bigg[\Bigg(\sum_{i=1}^NJ_{X^i}(R)\Bigg)^2\Bigg].
  \]
  Using the fact that if $U$ follows a Poisson law with parameter $\lambda$ then $\E[U^2] = \lambda^2 + \lambda$, we get
  \[
    \E_x\left[N^2\right] = \E_x\left[u^2 \mathrm{BCap}(A^{(s-1)}(r,R))^2\right] + \E_x\left[u\mathrm{BCap}(A^{(s-1)}(r,R))\right].
  \]
  By induction and the first part of the lemma, we deduce from this and from Lemma \ref{lem:bound_z_moments} \textit{(ii)}, that
  \[
    \E_x\left[\mathrm{BCap}(A^{(s)}(r,R))^2\right] \lesssim R^{8s}.
  \]
\end{proof}

\begin{lemma}
  \label{lem:lower_bound_capacity_final}
  Let $s$ be a positive integer. There exist $c_s = c(u,d,s)$ and $\varepsilon = \varepsilon(u,d,s)$, such that for all positive integers $r$ and $R$ with,
  \begin{equation}
    \label{eq:condition_r_R}
    r^{d-4} \leq \varepsilon R,
  \end{equation}
  and for all $x \in \mathrm{B}(R)$,
  \[
    \E_x\left[\mathrm{BCap}\left(A^{(s)}(r,R)\right)\right] \geq c_s R^{\min(d-4,4s)}.
  \]
\end{lemma}
\begin{proof}
  We prove the result by induction. If $s = 1$, we have by Lemma \ref{lem:phi_capacity_bound},
  \[
    \E_x\left[\mathrm{BCap}\left(A^{(1)}(r,R)\right)\right] \geq c_1 R^{\min(d-4,4)}.
  \]
  Let $s \geq 2$, and assume that,
  \[
    \E_x\left[\mathrm{BCap}\left(A^{(s-1)}(r,R)\right)\right] \geq c_{s-1} R^{\min(d-4,4s-4)}.
  \]
  With this lower bound and the bound on the moments established in Lemma \ref{lem:bcap_moments_bound}, we can use the Paley-Zigmund inequality, and obtain the existence of a constant $c' = c'(u,d,s)$, such that
  \begin{equation}
    \label{eq:lower_bound_proba_cap_min}
    \P\left(\mathrm{BCap}(A^{(s-1)}(r,R)) \geq c'R^{\min(d-4,4s-4)}\right) \geq c'.
  \end{equation}
  Note that since $d \geq 9$ and $s \geq 2$, we have $R^{\min(d-4,4s)} \geq R^5$. Then by Lemma \ref{lem:lower_bound_psi_r_inf} if $r^{d-4} \leq \varepsilon R$ and $\varepsilon$ is small enough,
  \begin{align*}
    \E\Bigg[\mathrm{BCap}(A^{(s)}(r,R))\Big]
    & \geq c\ \E\bigg[\min\left(u\mathrm{BCap}(A^{(s-1)}(r,R))R^4, R^{d-4}\right)\bigg] - Cur^{d-4}R^4 \\
    & \stackrel{\eqref{eq:lower_bound_proba_cap_min}}{\geq} cc'^2 \min(R^{\min(d-4,4s-4)}R^4,R^{d-4}) - Cur^{d-4}R^4 \\
    & \geq R^{\min(d-4,4s)} - C\varepsilon R^5 \\
    & \gtrsim R^{\min(d-4,4s)}.
  \end{align*}
\end{proof}

Recall the definition of $s_d$ from \eqref{def:s_d}, and the one of $\tau_R$ from \eqref{eq:exit_time_br}. We consider the event that the past of a branching random walk hits the set $A^{(s_d)}(r,R)$ at a vertex $u$ such that $n_u < \tau_{R^2}$, that is before the spine exits the ball of radius $R^2$.

\begin{lemma}
  \label{lem:probability_lower_bound}
  Let $X$ and $X^*$ be two independent random walks both indexed by infinite critical trees $\mathcal{T}$ and $\mathcal{T}^*$ respectively and independent from the branching interlacement. Let $\P_x$ denote the joint law of $X$ and $A^{(s_d)}(r,R)$, and $\P_z^*$ the law of $X^*$. Then there exist constants $c = c(u,d) > 0$, $\varepsilon = \varepsilon(u,d) > 0$ and $R' = R'(u,d)$, such that, for all $r$ and $R \geq R'$, satisfying $r^{d-4} \leq \varepsilon R$, and $r \geq \max(\|x\|,\|z\|)$,
  \[
    \P_z^* \otimes \P_x(\exists u \in \mathcal{T}^{*}_-:\, X^*_u \in A^{(s_d)}(r,R), n^*_u < \tau^*_{R^2}) \geq c,
  \]
  where the quantities with a * refer to $X^*$.
\end{lemma}
\begin{proof} Let $(S_n)_{n\in\mathbb{N}}$ (resp.~$(S_n^*)_{n\in\mathbb{N}}$) be the random walk induced by the spine of $\mathcal{T}$ (resp.~$\mathcal{T}^*$). Let $A = A^{(s_d)}(r,R)$ and $\widetilde{\tau}_R = \inf\{t \in \mathbb{N} : \mathrm{d}(S^{*}_t, A) \geq R\}$.
  Then by \eqref{thm:bound_hitting_probability_infinite}, conditioning on $\widetilde{\tau}_R$, and considering the random walk indexed by a Kesten's tree and starting at the random point $U = S^*_{\widetilde{\tau}_R}$.
  \[
    \P_z^* \otimes \P_x(\mathcal{T}^*_- \cap A \neq \emptyset) \geq \E_{z}^*\otimes \E_x \Big[\overline{\P}^*_{U} \otimes \P_x(\mathcal{T}^* \cap A \neq \emptyset)\Big] \geq c_1 \E_z^*\otimes \E_x\left[\frac{\mathrm{BCap}(A)}{\mathrm{d}\left(U,A\right)^{d-4}}\right].
  \]
   Since the jump law of our walks has finite range, by definition of $\widetilde{\tau}_R$, $\mathrm{d}(U,A) \asymp R$. Moreover, by Lemma \ref{lem:lower_bound_capacity_final}, $\mathrm{BCap}(A) \geq c_{s_d}R^{d-4}$, and there exists $C > 0$ only depending on the jump law, such that
  \[
    \P^*_z \otimes \P_x(\mathcal{T}^*_- \cap A \neq \emptyset) \geq Cc_1c_{s_d}.
  \]
  On the other hand, using again \eqref{thm:bound_hitting_probability_infinite}, if we let $V$ be the point of $\mathbb{Z}^d$ located at time $\tau_{R^2}$ on the spine of $\mathcal{T}^*$, the fact that $\mathrm{d}(V, A) \geq R^2 - s_dR$, entails
  \begin{align*}
    \E_z^*\otimes \E_x\Big[ \overline{\P}^*_V \otimes \P_x\big(\mathcal{T}^*_- \cap A \neq \emptyset\big)\Big]
    &\leq c_2\E_z^*\otimes\E_x\left[\frac{\mathrm{BCap}(A)}{\mathrm{d}\left(V,A\right)^{d-4}}\right] \\
    &\leq c_2c' \frac{(s_dR)^{d-4}}{(R^2-s_dR)^{d-4}} = \frac{c_2c's_d^{d-4}}{(R-s_d)^{d-4}}.
  \end{align*}
  By choosing $R$ large enough, we obtain
  \[
    \E_z^*\otimes \E_x\Big[ \overline{\P}^*_V \otimes \P_x\big(\mathcal{T}^*_- \cap A \neq \emptyset\big)\Big] \leq \frac{Cc_1c_{s_d}}{2}.
  \]
  Consequently,
  \begin{align*}
    \P_z^* \otimes \P_x&(\exists u \in \mathcal{T}^*_- : X^*_u \in A,\, n^*_u < \tau^*_{R^2}) \\
                    & \geq \P_z^* \otimes \P_x\big(\mathcal{T}^*_- \cap A \neq \emptyset\big) - \E_z^*\otimes \E_x\Big[ \overline{\P}^*_V \otimes \P_x\big(\mathcal{T}_- \cap A \neq \emptyset\big)\Big]
                      \geq \frac{Cc_1c_{s_d}}{2} > 0.
  \end{align*}
\end{proof}

\subsubsection{Construction of recurrent sets}
\label{sec:recurrent_sets}

Now that we know that the hitting probability of $A^{(s_d)}$ is lower bounded by a constant, we want to apply a Borel-Cantelli lemma to get an almost sure property. Until now, we have left the choices of $r$ and $R$ free (up to condition \eqref{eq:condition_r_R}). We will now construct sequences of integers $(r_k)_{k\geq 0}$ and $(R_k)_{k\geq 0}$ ensuring that the sequence $(A^{(s_d)}(r_k,R_k))_{k\geq 0}$ satisfies interesting measurable properties that allow us to apply a conditional Borel-Cantelli lemma (Lemma \ref{lem:bc_cond}).

Fix $\varepsilon > 0$ and $R'$ as given in Lemma \ref{lem:probability_lower_bound}. 
We define two growing sequences of radii $(r_k)_{k\in\mathbb{N}^*}$ and $(R_k)_{k\in\mathbb{N}^*}$, such that $r_0 = \max(|x|,|z|)$, $R_0 = \max\big(r_0^{d-4}/\varepsilon, R'\big)$, and for $k \geq 1$,
\begin{equation}
  \label{eq:def_rk_Rk}
  r_k = dR^2_{k-1}, \quad R_k = \frac{r_k^{d-4}}{\varepsilon}.
\end{equation}

They satisfy the properties we need to iterate Lemma \ref{lem:probability_lower_bound} with a conditional Borel-Cantelli lemma, that we state here in the version of Durrett (\cite[p.240]{durrett2019}) for completeness.

\begin{lemma}
  \label{lem:bc_cond}
  Let $(\Omega, \mathcal{F}, \P)$ be a probability space, $(\mathcal{F}_n)_{n\geq 0}$ be a filtration such that $\mathcal{F}_0 = \{\varnothing, \Omega\}$, and $(\Delta_n)_{n\geq 1}$ be a sequence of events such that $\Delta_n \in \mathcal{F}_n$. Then, almost surely,
  \[
    \{\Delta_n \text{ occurs infinitely often}\} = \Bigg\{\sum_{n=1}^\infty \P(\Delta_n\mid\mathcal{F}_{n-1}) = \infty \Bigg\}.
  \]
\end{lemma}

\begin{corollary}
  \label{lem:prob_limsup}
  Let $X, X^*, \mathcal{T}, \mathcal{T}^*, \P, \P^*$ be as defined in Lemma \ref{lem:probability_lower_bound}. For $x,z \in \mathbb{Z}^d$,
  \[
    \P_z^* \otimes \P_x\left(\limsup_{k}\left\{\mathcal{T}_-^* \cap \mathrm{A}^{(s_d)}(r_k,R_k) \neq \varnothing\right\}\right) = 1.
  \]
\end{corollary}
\begin{proof}
  \noindent Recall the definition of $\sigma_{r,R}$ before Lemma \ref{lem:lower_bound_psi_r_inf}. Consider the filtration $\left(\mathcal{F}_k\right)_{k\geq 0}$ defined by
  \[
    \mathcal{F}_k = \sigma\left(\{X_u, n_u < \tau_{r_{k+1}}\} \cup \{X^*_u: n_u^* < \tau^*_{r_{k+1}}\} \cup \sigma_{r_k,r_{k+1}}\right), \quad k \geq 0.
  \]
  Furthermore, note that for every simple random walk $(S_n)_{n\in\mathbb{N}}$ by equation \eqref{eq:def_rk_Rk}, we have that $\{S(t),\, t \leq \tau_{R_k^2} + R_k^2\} \subset \mathrm{B}(r_{k+1})$. Consequently for every $k \geq 1$ the event
  \[
    \Delta_k = \{\exists u \in \mathcal{T}^*_-:\, X^*_u \in A^{(s_d)}(r_k,R_k),\, n^*_u < \tau_{R_k^2}\},
  \]
  is measurable with respect to $\mathcal{F}_k$.

  \noindent To use Lemma \ref{lem:bc_cond}, we prove that there exists $c > 0$, such that for every $k \geq 1$,
  \begin{equation}
    \label{eq:conditional_eq}
    \P_z^* \otimes \P_x(\Delta_{k+1} \mid \mathcal{F}_k) \geq c.
  \end{equation}

  Knowing the positions of $X$ and $X^*$ up to a certain time allows us to apply two Markov properties. First of all for $X^*$, the random walk indexed by an infinite invariant random tree conditioned by $\mathcal{F}_k$ becomes a random walk indexed by a Kesten's tree that starts at a certain point. However for $X$, the conditioning does not alter the law of $A^{(s_d)}$. Indeed, recall that we already cut the branching random walk $X$ and defined the set $A^{(1)}$ using Kesten's branching random walk $Y$. Thus,
  \[
    \P_z^* \otimes \P_x(\Delta_{k+1} \mid \mathcal{F}_k) = \overline{\P}^*_{z'} \otimes \P_{x'}(\Delta_{k+1}),
  \]
  where $z' = S^*_{\tau^*_{r_k}}$ and $x' = S_{\tau_{r_k}}$. \\
  We can apply Lemma \ref{lem:probability_lower_bound} and deduce \eqref{eq:conditional_eq}. Finally, we conclude by applying Lemma \ref{lem:bc_cond}.
\end{proof}

\begin{proof}[Proof of the upper bound in Theorem \ref{thm:2_points_branching}.]
  Fix $u > 0$ and let $(u_N)_{N\geq 1}$ be an increasing sequence converging to $u$. Let $N \geq 1$, and $w, w^* \in \sigma_{u_n}$ intersecting $B_N = \mathrm{B}(0,N)$. By Proposition \ref{prop:desc_interlacement_compact}, they can respectively be described as trajectories of branching random walks $X$ and $X^*$, starting from the boundary of $B_N$. Consider the interlacement $\sigma_{u_N,u}$, restriction of $\sigma_u$ to the trajectories appearing between $u_N$ and $u$, and independent from $\sigma_{B_N,u_N}$ and thus from $X$ and $X^*$. Apply Corollary \ref{lem:prob_limsup} with $X,X^*$ and $\sigma_{u_N,u}$. Then there exist $s_d-1$ trajectories of $\sigma_{u_N,u}$ that connect $X$ and $X^*$. Therefore, $w$ and $w^*$ are at distance at most $s_d-1$ in $\sigma_u$. We conclude by observing that $\mathrm{Supp}(\sigma_u) = \bigcup_N \mathrm{Supp}(\sigma_{B_N,u_N})$.
\end{proof}
\subsection{Proof of the lower bound $\mathrm{diam}(G) \geq s_d$}
\label{sec:lower_bound_2_pt}
We want to prove that almost surely there exists a pair of trajectories of the interlacement which are not connected in less than $s_d-1$ steps. Recall the construction of the graph $G$ underneath the interlacement: the vertex set is $\mathrm{Supp}(\sigma_u)$, and two trajectories are neighbours if they intersect. Let
\begin{equation}
  \label{def:event_E}
  E = \{\mathrm{d}(w,w^*) \leq s_d-1,\, \forall w,w^*\in \mathrm{Supp}(\sigma_u)\},
\end{equation}
where $\mathrm{d}(w,w^*)$ denotes the graph distance between $w$ and $w^*$ in $G$.
Firstly, we prove that the probability that two points of the random interlacement are connected with less than $s_d-1$ trajectories decreases at least as the inverse of their distance. Afterwards, we suppose that the probability of $E$ is nonzero, and we reach a contradiction.

\noindent Before stating the next lemma, remember the definition of the measure $\nu$ from Section \ref{sec:branching_interlacements}, and let $S(x,y)$ be the set of trajectories that hit both $x$ and $y$,
\[
  S(x,y) = \{w \in W^*,\ x,y \in \mathrm{Range}(w)\}.
\]

\begin{lemma}
  \label{lem:product_decomposition}
  For any $x,y \in \mathbb{Z}^d$,
  \[
    \P(\{x,y \in \mathcal{I}\} \cap E) \leq \sum_{n=0}^{s_d-1} \sum_{z_1,\dots,z_n \in \mathbb{Z}^d} \prod_{i=0}^n \big(u\nu(S(z_i,z_{i+1}))\big).
  \]
  where $z_0 = x$ and $z_{n+1} = y$.
\end{lemma}
\begin{proof}
  Observe that
  \[
    \{x,y \in \mathcal{I}\} \cap E = \bigcup_{n=0}^{s_d-1} \bigcup_{z_1,\dots,z_n\in\mathbb{Z}^d} \bigcup_{w_1,\dots,w_n \in \sigma_u} \bigcap_{i=0}^n \{w_i \text{ hits } z_i \text{ and } z_{i+1}\}.
  \]
  Thus, by the Palm-Mecke's Theorem for general Poisson point processes (see \cite[Chapter 13.1]{dvj2008} for a proof)
  \[
    \P(\{x,y \in \mathcal{I}\} \cap E) \leq \sum_{n=0}^{s_d-1}\sum_{z_1,\dots,z_n\in\mathbb{Z}^d} \E\Bigg[\sum_{w_1,\dots,w_n\in\sigma_u} \prod_{i=0}^n \indic{w_i \text{ hits } z_i \text{ and } z_{i+1}}\Bigg] = \sum_{n=0}^{s_d-1}\sum_{z_1,\dots,z_n\in\mathbb{Z}^d} \prod_{i=0}^n \big(u\nu(S(z_i,z_{i+1}))\big),
  \]
  which concludes the proof of the lemma.
\end{proof}

\begin{lemma}
  \label{lem:2_points_lower_bound_probability}
  There exists a constant $C > 0$, such that, for every $x \neq y \in \mathbb{Z}^d$,
  \[
    \P(\{x,y \in \mathcal{I}\} \cap E) \leq \frac{C}{\|x-y\|}.
  \]
\end{lemma}
\begin{proof}
  \noindent Let $A = \{w \in W^* : y \in \mathrm{Range}(w)\}$. Then by \eqref{eq:hit_T},
  \[
    \nu(S(x,y)) = \P\left((X_{\mathcal{T}}^x(n))_{n\in\mathbb{Z}} \in A,\, x \notin \mathcal{T}_-^x\right) \leq \P(y \in \mathcal{T}^x) \lesssim G(x,y).
  \]

  \noindent Hence, according to Lemma \ref{lem:product_decomposition} and \eqref{eq:estim_green},
  \[
    \P(\{x,y\in \mathcal{I}\} \cap E) \leq \sum_{n=0}^{s_d-1} (uC')^{n+1}\hspace{-10pt} \sum_{z_1,\dots,z_n\in\mathbb{Z}^d} \prod_{i=0}^n \frac{1}{1 + \|z_{i+1}- z_i\|^{d-4}}.
  \]
  From Lemma \ref{lem:magic_ineq}, we have the following estimate for $4 \leq k < d-4$,
  \begin{equation}
    \label{eq:ineq_convo}
    \sum_{z\in\mathbb{Z}^d} \frac{1}{(1+\|x-z\|^{d-4})(1+\|z-y\|^{d-k})} \leq \frac{C''}{1+\|x-y\|^{d-4-k}},
  \end{equation}
  and then by induction, it follows that (recall that $4s_d \leq d-1$),
  \[
    \P(\{x,y\in \mathcal{I}\} \cap E) \leq \sum_{n=0}^{s_d-1} (uC')^{n+1}(C'')^n \frac{1}{1+\|x-y\|^{d-4(n+1)}} \leq C \|x-y\|^{-1}.
  \]
\end{proof}

\begin{proof}[Proof of lower bound in Theorem \ref{thm:2_points_branching}.]
  Suppose that $\P(E) > 0$ with $E$ as in \eqref{def:event_E}. Since,
  \[
    \P(\mathcal{I}\cap\mathrm{B}(R) \neq \varnothing) = \P(N_{\mathrm{B}(R)} \geq 1) = 1 - e^{-u\mathrm{BCap}(\mathrm{B}(R))},
  \]
  and $\mathrm{BCap}(\mathrm{B}(R)) \to \infty$ as $R \to \infty$, we can choose $R$ large enough, such that
  \begin{equation}
    \label{eq:lower_bound_intersect_ball}
    \P(\mathcal{I}\cap\mathrm{B}(R) \neq \varnothing) \geq 1 - \frac{\P(E)}{3}.
  \end{equation}
  Then for any $z \in \mathbb{Z}^d$,
  \[
    \P(\{\mathcal{I}\cap\mathrm{B}(R)\neq\varnothing\}\cap\{\mathcal{I}\cap\mathrm{B}(z,R)\neq\varnothing\} \cap E) \geq \P(E) - 2\P(\mathcal{I}\cap\mathrm{B}(R)=\varnothing) \geq \frac{\P(E)}{3}.
  \]
  However, for every $z \in\mathbb{Z}^d$ with $|z| > 3R$,
  \[
    \P(\{\mathcal{I} \cap \mathrm{B}(R) \neq \varnothing\} \cap \{\mathcal{I} \cap \mathrm{B}(z,R) \neq \varnothing\} \cap E) \leq \sum_{x \in \mathrm{B}(R)} \sum_{y\in\mathrm{B}(z,R)} \P(\{x,y\in\mathcal{I},\, E) \leq \frac{CR^{2d}}{\|z\|}.
  \]
  By sending $\|z\|$ to the infinity, we reach a contradiction. As a consequence, $\P(E) = 0$ and $\P$-a.s.~the diameter of $G$ is at least $s_d$.
\end{proof}

\section{Connecting $k$ points}
\label{sec:k_points_connection}
We prove Theorem \ref{thm:k_points_branching} here. Although the general strategy is similar to Lacoin and Tykesson \cite{lt2013} in the simple random walk case, the details are quite different and more involved here due to the branching structure of the trajectories.

Remember the definition of $n(k,d)$ from \eqref{def:n_k_d}. In order to prove the upper bound in the theorem, we fix $k$ points in the interlacement and we want to find $n(k,d)$ trajectories that connect them. We first provide a lemma (Lemma \ref{lem:t_connection} below) which controls the number of trajectories needed to connect any given set of $n\leq 5$ other trajectories. When applied successively to overlapping groups of $5$ walks (see Figure \ref{fig:sketch_proof_upper_bound}) this lemma proves the upper bound in the theorem. The main difficulty in proving Lemma \ref{lem:t_connection} is to circumvent the lack of Markov property for a branching random walk. This is why we prove Lemma \ref{lem:t_A_tau_alphaR}, in order to control the position of the walk induced by the spine at the time one of its descendent trees hits a certain set.

Similarly to what has been done in Section \ref{sec:lower_bound_2_pt}, we prove the lower bound by contradiction. We suppose that we can connect every $k$ points of the interlacement using less than $n(k,d)$ trajectories, and show that the probability for $k$ points to belong to the interlacement decreases as the inverse of the minimum of their mutual distances. Two new difficulties arise. First of all, whereas in Section \ref{sec:lower_bound_2_pt} the only way to connect $2$ points was by a path of pairwise intersecting trajectories of the interlacement, here the underlying graph is more intricate. Second, branching random walks do not connect points one after the other but simultaneously. This increases the complexity of this graph and is the reason why we introduce the set of trees $\mathbb{A}_{k,n,m}$ in Section \ref{sec:lower_bound_k}. The interest is that we can link the probability of connecting $k$ points to a certain weight on those trees (see Definition \ref{def:tree_weight}). We successively apply two operations (\textit{cut} and \textit{merge}) to reduce a tree of $\mathbb{A}_{k,n,m}$ to a single edge. By controlling the structure of the tree throughout the process, we also control the weight of the final edge and thus the probability of connecting $k$ points.

\medskip

In this whole section, we let $d \geq 9$ and $k \geq 3$, and suppose that the reproduction law $\mu$ has mean $1$ and a finite moment of order $2\lceil d(k-1)/4\rceil$. This moment assumption is used in the proof of the lower bound, whereas we only need a finite third moment for the upper bound.

\subsection{Proof of the upper bound}
\label{sec:upper_bound_k}
The goal of this subsection is to prove the upper bound in Theorem \ref{thm:k_points_branching}, i.e.~that for any $k$ points of the interlacement, one can find $n(k,d)$ branching random walks that connect the $k$ points. First we show how to connect $n\leq 5$ walks using a certain number of trajectories of the interlacement.

\begin{lemma}
  \label{lem:t_connection}
  Let $n \leq 5$, let $X^{(1)},\dots,X^{(n)}$ be $n$ independent random walks indexed by infinite invariant trees. Then almost surely there exist
  \[
    \left\lceil\frac{(n-1)d}{4}\right\rceil - 2(n-1)
  \]
  trajectories from an independent interlacement that connect $X^{(1)},\dots,X^{(n)}$.
\end{lemma}

We defer the proof after Lemma \ref{lem:t_connection_partial}, since it is a consequence of the latter. First of all, let $n\leq 5$ and fix $(n-1)$ integers $(k_i)_{1\leq i \leq {n-1}}$ (implicitely depending on $d$), such that
\[
  k_i = q + \indic{i \leq m},
\]
where $q$ and $m$ are defined by $\left\lceil (n-1)d/4 \right\rceil - n = q(n-1) + m$, and $m < n-1$. Note that they satisfy the following properties
\begin{equation}
  \label{eq:cond_k1_kt}
  k_1 + \dots + k_{n-1} = \left\lceil \frac{(n-1)d}{4} \right\rceil - n,
\end{equation}
\begin{equation}
  \forall 1 \leq i \leq n-1,\quad 4k_i \leq d. \label{eq:cond_ki}
\end{equation}

\noindent Then we need the following result. Remember the definition of the visible sets from \eqref{eq:def_A_1} and \eqref{eq:def_A_s}.
\begin{lemma}
  \label{lem:t_connection_partial}
  Let $n \leq 5$. Let $X^{(1)},\dots,X^{(n)}$ be $n$ independent random walks indexed by infinite critical trees starting respectively from $x_1, \dots, x_n$. Let $\sigma_u, \sigma_u^{(1)},\dots,\sigma_u^{(n)}$ be $n+1$ independent interlacements with law $\mathrm{Poi}(u,W^*)$ and independent from $X^{(1)},\dots,X^{(n)}$. For $1 \leq i \leq n$, denote by $(A_i^{(j)})_{j\geq 1}$ the visible sets associated to $X^{(i)}$ and constructed with the interlacement $\sigma_u^{(i)}$. There exist constants $c > 0$, $R_0 < \infty$ and $\varepsilon > 0$, such that, for any $r > \max(\|x_1\|,\dots,\|x_n\|)$ and $R > R_0$, satisfying $r^{d-4} \leq \varepsilon R$, we have
  \[
    \P\left(\exists\,\gamma \in \mathrm{Supp}(\sigma^u_{r,R}) : \gamma \text{ connects } A_1^{(k_1)}(r,R), \dots, A_{n-1}^{(k_{n-1})}(r,R) \text{ and } A_n^{(1)}(r,R)\right) \geq c.
  \]
\end{lemma}

We deduce now Lemma \ref{lem:t_connection} from Lemma \ref{lem:t_connection_partial}, the same way we deduced Corollary \ref{lem:prob_limsup} from Lemma \ref{lem:probability_lower_bound}, and prove Lemma \ref{lem:t_connection_partial} afterwards.

\begin{proof}[Proof of Lemma \ref{lem:t_connection}.]
  \noindent Fix $\varepsilon$ and $R_0$ as in Lemma \ref{lem:t_connection_partial}. Fix $\sigma_u$ with law $\mathrm{Poi}(u,W^*)$, and decompose it into $n+1$ independent random interlacements $\sigma'$ and $(\sigma^{(i)})_{1\leq i\leq n}$ with law $\mathrm{Poi}(u/(n+1),W^*)$. Let $r_1 = \max(\|x_1\|,\dots,\|x_n\|)$, $R_1 = \max(r_1^{d-4}/\varepsilon, R_0)$, and define recursively for $k \geq 2$,
  \[
    r_k = dR^2_{k-1}, \quad R_k = \frac{r_k^{d-4}}{\varepsilon}.
  \]
  Then for $k \geq 1$, we consider the event
  \[ 
    \Delta_k = \{\exists \gamma \in \mathrm{Supp}(\sigma'_{r_k,R_k}) : \gamma \text{ connects } A_1^{(k_1)}(r_k,R_k), \dots, A_{n-1}^{(k_{n-1})}(r_k,R_k) \text{ and } A_n^{(1)}(r_k,R_k)\},
  \]
  which is measurable with respect to the $\sigma-$algebra
  \[
    \mathcal{F}_k = \sigma\bigg(\{X_u^{(1)} : n_u^{(1)} < R_k^2/4\} \cup \dots \cup \{X_u^{(n)} : n_u^{(n)} < R_k^2/4\} \cup \sigma'_{r_k,R_k} \cup \bigcup_{i=1}^n \sigma^{(i)}_{r_k,r_{k+1}}\bigg),
  \]
  where $n_u^{(i)}$ refers to $X^{(i)}$. By Lemma \ref{lem:t_connection_partial} applied to $X^{(1)},\dots,X^{(n)}$ and $\sigma',\sigma^{(1)},\dots,\sigma^{(n)}$, we have that ${\P(\Delta_k \mid \mathcal{F}_{k-1}) \geq c}$, and we conclude the proof by applying Lemma \ref{lem:bc_cond}.
\end{proof}

The following lemma will be used in the proof of Lemma \ref{lem:t_connection_partial}. Recall the definition of the exit time $\tau_R$ of the ball of radius $R$ by the spine from \eqref{eq:exit_time_br}, and define
\[
  t_A = \inf\{n_u : u \in \mathcal{T}_-, X_u \in A\}.
\]

\begin{lemma}
  \label{lem:t_A_tau_alphaR}
  There exists $\alpha > 2$, such that for $R > 1$, $A \subseteq \mathrm{B}(R)$ and $X$ a random walk indexed by an infinite critical tree $\mathcal{T}$, 
  \[
    \inf_{x \in \mathrm{B}(R)} \P_x(t_A < \tau_{\alpha R}) \gtrsim \frac{\mathrm{BCap}(A)}{R^{d-4}}.
  \]
\end{lemma}
\begin{proof}
    Recall the constants $c_1$ and $c_2$ from Theorem \ref{thm:bound_hitting_probability}, and let $\alpha = 1 + 2(2c_2/c_1)^{1/(d-4)} > 2$. For $x \in \partial \mathrm{B}(2R)$, observe that
  \[
    \P_x(t_A < \infty) = \P_x(t_A \leq \tau_{\alpha R}) + \P_x(\tau_{\alpha R} < t_A < \infty).
  \]
  Since $x\in\partial\mathrm{B}(2R)$ and $A \subset \mathrm{B}(R)$, one has $R \leq \mathrm{dist}(x,A) \leq 2R$. Then Theorem \ref{thm:bound_hitting_probability} yields that,
  \[
    \P_x(t_A < \infty) \leq \frac{c_2\mathrm{BCap}(A)}{(2R)^{d-4}}.
  \]
  Moreover, the Markov property at time $\tau_{\alpha R}$ yields
  \[
    \P_x(\tau_{\alpha R} < t_A < \infty) \leq \sup_{y \in \partial\mathrm{B}(\alpha R)} \P_y(t_A < \infty) \leq c_2\frac{\mathrm{BCap}(A)}{(\alpha R-R)^{d-4}}.
  \]
  Therefore by the definition of $\alpha$,
  \begin{equation}
    \label{eq:lower_bound_ta_talphaR_boundary}
    \P_x(t_A \leq \tau_{\alpha R}) \geq \frac{\mathrm{BCap}(A)}{R^{d-4}} \left(\frac{c_1}{2^{d-4}} - \frac{c_2}{(\alpha-1)^{d-4}}\right) \geq \frac{c_1\mathrm{BCap}(A)}{2(2R)^{d-4}} \gtrsim \frac{\mathrm{BCap}(A)}{R^{d-4}}.
  \end{equation}
  Next, observe that for all $x\in\mathrm{B}(R)$,
  \[
    \P_x(t_A < \tau_{\alpha R}) \geq \P_x(\tau_{2R} < t_A < \tau_{\alpha R}) \geq \inf_{y \in \partial\mathrm{B}(2R)} \P_y(t_A < \tau_{\alpha R}).
  \]
  We conclude by applying \eqref{eq:lower_bound_ta_talphaR_boundary} to the previous equation.
\end{proof}

\begin{figure}[H]
  \centering
  \begin{tikzpicture}[scale = 0.7, use Hobby shortcut]
    \draw[left] (-0.1,0.5) node {$A_1$};
    \draw ([closed]0,0) .. (1,-1) .. (2,-0.5) .. (2.25,0) .. (3,2);

    \draw[below right] (6,-1.5) node {$A_2$};
    \draw ([closed]4,-2) .. (4.3,-2.5) .. (6,-1.5) .. (4.7,-0.75) .. (3.6,0);

    \draw[below] (9,3) node {$A_n$};
    \draw ([closed]7.5,2.5) .. (8,3) .. (9, 3) .. (8,4.5);

    \draw[above] (-0.3,3) node {$\gamma$};
    \draw (-0.3, 3) .. (1,2) .. (1.2,1) .. (1.5,0) .. (3,-0.5) .. (3.5,-1.5) .. (5, -1) .. (6,0);
    \draw[dotted] (6,0) .. (7,1);
    \draw (7,1) .. (7.5,1.5) .. (8,3) .. (7.5,3.5) .. (7,3.25) .. (6.5,2.5) .. (5.5, 3);
  \end{tikzpicture}
  \caption{Example of a connection of $n$ visible sets.}
  \label{fig:ex_lemma_n_connection_partial}
\end{figure}
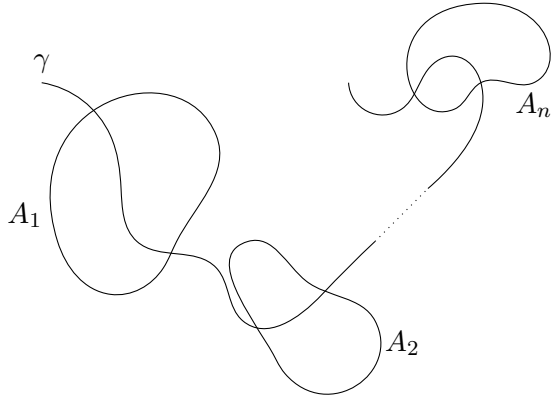

\begin{proof}[Proof of Lemma \ref{lem:t_connection_partial}.]
  For simplicity, we will write $A_i$ for $A_i^{(k_i)}(r,R)$ if $1 \leq i \leq n-1$, and $A_n$ for $A_n^{(1)}(r,R)$. \\
  Let $E = \left\{\exists\,\gamma \in \sigma^u_{r,2R} : \gamma \text{ connects } A_1, \dots, A_n\right\}$. We want to prove that we can find $c = c(u,d) > 0$ such that $\P(E) > c$. We define
  \[
    E_1 = \left\{\exists\,\gamma \in \sigma^u_{2R} : \gamma \text{ connects } A_1, \dots, A_n\right\}, \quad\quad\quad
    E_2 = \left\{\exists\,\gamma \in \sigma^u_{r} : \gamma \text{ intersects } A_n\right\},
  \]
  and note that $E_1 \backslash E_2 \subset E$. We show that for $R$ large enough, $\P(E_1)$ is lower bounded by a positive constant and that $\P(E_1) \geq 2\P(E_2)$, which in turn proves that $\P(E) \geq \P(E_1) - \P(E_2) \geq \P(E_1)/2 \geq c$.

   By definition of the branching interlacement, if we let $N$ be a Poisson random variable with parameter $u\mathrm{BCap}(A_n)$, and $\gamma$ be a random walk indexed by an infinite invariant tree independent from $N$ whose starting point is chosen according to $\tilde{e}_{A_n}$, then
  \[
    \P(E_1) = \E\left[1 - \Big(1 - \P\big(\gamma \text{ hits } A_1 \dots A_{n-1}\big)\Big)^N\right].
  \]
  Observe that if we let $\rho_k = \alpha^kdR$, where $\alpha$ is as in Lemma \ref{lem:t_A_tau_alphaR}, we have
  \[
    \P(\gamma \text{ hits } A_1,\dots,A_{n-1}) \geq \P(t_{A_1} < t_{A_2} < \dots < t_{A_{n-1}}) \geq \P(t_{A_1} < \tau_{\rho_1} < t_{A_2} < \tau_{\rho_2} < \dots < t_{A_{n-1}} < \tau_{\rho_{n-1}}).
  \]
  Then by conditioning with respect to $\sigma(\{u, X_u : n_u \leq \tau_{\rho_{n-2}}\})$, we get
  \[
    \P(t_{A_1} < \tau_{\rho_1} < t_{A_2} < \tau_{\rho_2} < \dots < t_{A_{n-1}} < \tau_{\rho_{n-1}}) = \E\big[\indic{t_{A_1} < \tau_{\rho_1}}\dots \indic{t_{A_{n-2}} < \tau_{\rho_{n-2}}} \overline{\P}_{S_{\tau_{\rho_{n-2}}}}(t_{A_{n-1}} < \tau_{\rho_{n-1}})\big],
  \]
  where $\overline{\P}$ denotes the law of a random walk indexed by Kesten's tree.
  Since for $i \leq n-1$, $A^{(i)} \subset \mathrm{B}(\rho_i)$, by applying Lemma \ref{lem:t_A_tau_alphaR} and iterating, it follows that
  \[
    \P(\gamma \text{ hits } A_1,\dots,A_{n-1}) \gtrsim \E\bigg[\frac{\mathrm{BCap}(A_1)}{R^{d-4}} \dots \frac{\mathrm{BCap}(A_{n-1})}{R^{d-4}}\bigg].
  \]
  Furthermore, remember that $N$ is a Poisson random variable with parameter $u\mathrm{BCap}(A_n)$, and thus we get for some constant $c > 0$,
  \[
    \P(E_1) \geq 1 - \E\bigg[\sum_{p=0}^\infty (u\mathrm{BCap}(A_n))^p \Big(1 - \frac{\E\big[\mathrm{BCap}(A_1)\dots \mathrm{BCap}(A_{n-1})\big]}{R^{(n-1)(d-4)}}\Big)^p\exp(-u\mathrm{BCap}(A_n))\bigg]. \\
  \]
  Finally, Jensen's inequality gives
  \[
    \P(E_1) \geq 1 - \exp\bigg(-c \frac{\E\big[\mathrm{BCap}(A_1)\dots\mathrm{BCap}(A_n)\big]}{R^{(n-1)(d-4)}}\bigg).
  \]
  Then using the Paley-Zigmund inequality together with Lemmas \ref{lem:bcap_moments_bound} and \ref{lem:lower_bound_capacity_final} (remember that $4k_i \leq d$ for all $i \leq n-1$), there exist positive $\delta$, $C$ and $\varepsilon$ such that if $r^{d-4} \leq \varepsilon R$,
  \begin{equation}
    \label{eq:lower_bound_pz}
    \forall\, i \leq n-1,\ \P\big(\mathrm{BCap}(A_i) \geq CR^{4k_i}\big) \geq \delta, \quad\text{and}\quad \P\big(\mathrm{BCap}(A_n) \geq CR^{4}\big) \geq \delta.
  \end{equation}
  Recalling from \eqref{eq:cond_k1_kt} that $4(k_1+\dots+k_{n-1}) \geq (n-1)(d-4)-4$, for some smaller constants $c' > 0$ and $c''>0$, we deduce
  \begin{equation}
    \label{eq:bound_E1}
    \P(E_1) \geq 1 - \exp\bigg(-c' \frac{R^{4(k_1+\dots+k_{n-1})+4}}{R^{(n-1)(d-4)}}\bigg) > c'' > 0.
  \end{equation}

  \noindent On the other hand, if we let again $\gamma$ be a random walk indexed by an infinite invariant tree whose starting point is chosen according to $\tilde{e}_{A_n}$,
  \[
    \P(E_2) = 1 - \E\bigg[\big(1 - \P(\mathrm{Range}(\gamma) \cap \mathrm{B}(r) \neq \varnothing)\big)^N\bigg].
  \]
  Remember that in the definition of $A^{(1)}$ in \eqref{eq:def_A_1}, we have chosen the origin of the branching walk so that it belongs to the boundary of the ball $\mathrm{B}(R)$. Consequently, almost surely $\mathrm{d}(A^{(1)}, \mathrm{B}(r)) \geq R/2-r$, and in particular, $\mathrm{d}(\gamma_\varnothing, \mathrm{B}(r)) \geq R/2-r$. This allows us to apply Theorem \ref{thm:bound_hitting_probability}, and deduce that
  \begin{align*}
    \P(\mathrm{Range}(\gamma) \cap \mathrm{B}(r) \neq \varnothing) &\leq \P(\exists\, u \in \mathcal{T}_- : \gamma_u \in \mathrm{B}(r)) + \P(\exists\, u \in \mathcal{T}_+, \gamma_u \in \mathrm{B}(r)) \\ &\lesssim \frac{\mathrm{BCap}(\mathrm{B}(r))}{(R/2-r)^{d-4}} \lesssim \frac{r^{d-4}}{(R/2-r)^{d-4}}.
  \end{align*}
  It follows, using again $r^{d-4} \leq \varepsilon R$, that
  \[
    \P(\mathrm{Range}(\gamma)\cap\mathrm{B}(r) \neq \varnothing) \lesssim \frac{\varepsilon}{R^{d-5}},
  \]
  which gives by \eqref{eq:lower_bound_pz}
  \[
    \P(E_2) \leq 1 - \E\Bigg[\exp\bigg(-\frac{C \varepsilon}{R^{d-5}} u\mathrm{BCap}(A_n)\bigg)\Bigg] \leq 1 - \exp\bigg(-\frac{C\varepsilon}{R^{d-9}}\bigg).
  \]
  Thus if we choose $\varepsilon$ small enough,
  \begin{equation}
    \label{eq:bound_E2}
    \P(E_2) \leq \frac{\P(E_1)}{2}.
  \end{equation}

  \noindent Finally, from \eqref{eq:bound_E1} and \eqref{eq:bound_E2}, we get that
  \[
    \P(E) \geq \frac{c''}{2} > 0.
  \]
\end{proof}

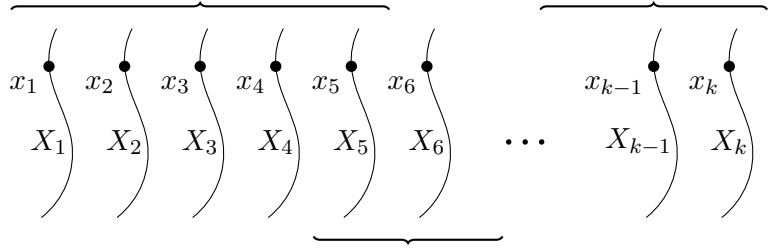
\begin{figure}[H]
  \centering
  \begin{tikzpicture}[use Hobby shortcut]
    \node at (0,0) [circle,fill,inner sep=1.5pt]{};
    \node at (1,0) [circle,fill,inner sep=1.5pt]{};
    \node at (2,0) [circle,fill,inner sep=1.5pt]{};
    \node at (3,0) [circle,fill,inner sep=1.5pt]{};
    \node at (4,0) [circle,fill,inner sep=1.5pt]{};
    \node at (5,0) [circle,fill,inner sep=1.5pt]{};
    \node at (8,0) [circle,fill,inner sep=1.5pt]{};
    \node at (9,0) [circle,fill,inner sep=1.5pt]{};

    \draw[below left] (0,0) node {$x_1$};
    \draw[below left] (1,0) node {$x_2$};
    \draw[below left] (2,0) node {$x_3$};
    \draw[below left] (3,0) node {$x_4$};
    \draw[below left] (4,0) node {$x_5$};
    \draw[below left] (5,0) node {$x_6$};
    \draw[below left] (8,0) node {$x_{k-1}$};
    \draw[below left] (9,0) node {$x_k$};

    \draw (0,-1) node {$X_1$};
    \draw (1,-1) node {$X_2$};
    \draw (2,-1) node {$X_3$};
    \draw (3,-1) node {$X_4$};
    \draw (4,-1) node {$X_5$};
    \draw (5,-1) node {$X_6$};
    \draw (7.8,-1) node {$X_{k-1}$};
    \draw (9,-1) node {$X_k$};
    
    \draw (0.1,0.5) .. (0,0) .. (0.3,-1) .. (-0.1,-2);
    \draw (1.1,0.5) .. (1,0) .. (1.3,-1) .. (0.9,-2);
    \draw (2.1,0.5) .. (2,0) .. (2.3,-1) .. (1.9,-2);
    \draw (3.1,0.5) .. (3,0) .. (3.3,-1) .. (2.9,-2);
    \draw (4.1,0.5) .. (4,0) .. (4.3,-1) .. (3.9,-2);
    \draw (5.1,0.5) .. (5,0) .. (5.3,-1) .. (4.9,-2);
    \draw (8.1,0.5) .. (8,0) .. (8.3,-1) .. (7.9,-2);
    \draw (9.1,0.5) .. (9,0) .. (9.3,-1) .. (8.9,-2);

    \draw[decorate,decoration=brace,thick] (-0.5,0.75) -- (4.5,0.75);
    \draw[decorate,decoration=brace,thick] (6,-2.25) -- (3.5,-2.25);
    \draw[decorate,decoration=brace,thick] (6.5,0.75) -- (9.5,0.75);

    \draw (6.3,-1) node {{\Huge...}};
  \end{tikzpicture}
  \caption{Sketch of proof of the upper bound in Theorem \ref{thm:k_points_branching}}
  \label{fig:sketch_proof_upper_bound}
\end{figure}

\begin{proof}[Proof of the upper bound in Theorem \ref{thm:k_points_branching}]
  We prove that we can connect any familly of $k$ trajectories with $n(k,d)$ trajectories of the interlacement, which proves the upper bound in Theorem \ref{thm:k_points_branching}. Fix $u > 0$ and let $(u_N)_{N\geq 1}$ be an increasing sequence converging to $u$. Let $w_1,\dots,w_k \in \sigma_{u_N}$ intersecting $B_N = \mathrm{B}(0,N)$. By Proposition \ref{prop:desc_interlacement_compact}, they can respectively be described as trajectories of branching random walks $X_1,\dots,X_k$, starting from the boundary of $B_N$. Consider the interlacement $\sigma_{u_N,u}$, restriction of $\sigma_u$ to the trajectories appearing between $u_N$ and $u$, which is independent from $\sigma_{B_N,u_N}$ and thus from $X_1,\dots,X_k$.
  
  Let $m \geq 1$ and $0 \leq n \leq 3$, such that $k = 4(m-1)+n+1$. Cut $\sigma_{u_N,u}$ into independent interlacements $(\sigma^{(i)})_{0\leq i\leq m-1}$ with distribution $\mathrm{Pois}((u-u_N)/m)$. For each $0 \leq i \leq m-2$, apply Lemma \ref{lem:t_connection} to the $5$ branching random walks $X_{4i+1},X_{4i+2},\dots,X_{4i+5}$, and the interlacement $\sigma^{(i)}$. Then there exist $d-8$ trajectories of $\sigma^{(i)}$ that connect $X_{4i+1},\dots,X_{4i+5}$. Since two groups of $5$ walks labeled by consecutive indices share a common walk, we can connect $X_1,\dots,X_{4(m-1)+1}$ with $(m-1)(d-8)$ walks of the interlacement.

  Finally, in the case where $n \geq 1$, we can connect the last $n+1$ walks $X_{4(m-1)+1},\dots,X_{4(m-1)+n+1}$ using again Lemma \ref{lem:t_connection} with $\lceil nd/4\rceil - 2n$ additional trajectories from $\sigma^{(m-1)}$. In total to connect $w_1,\dots,w_k$, we need at most
  \[
    (m-1)(d-8) + \left\lceil\frac{nd}{4}\right\rceil - 2n
  \]
  trajectories of $\sigma_{u_N,u}$. Then observe that,
  \begin{align*}
    (m-1)(d-8) + \left\lceil\frac{nd}{4}\right\rceil - 2n
    &= 2k - 2n - 2 + (m-1)(d-8) + \left\lceil\frac{nd}{4}\right\rceil - (k-2) -k \\
    &= 8(m-1) + (m-1)(d-8) + \left\lceil\frac{nd}{4}\right\rceil - (k-2) -k \\
    &= (m-1)d + \left\lceil\frac{nd}{4}\right\rceil - (k-2) -k = n(k,d) -k.
  \end{align*}
  Hence almost surely any $k$ trajectories of $\sigma_{B_N, u_N}$ are connected using at most $n(k,d)$ trajectories of $\sigma_u$. The proof follows by observing that $\mathrm{Supp}(\sigma_u) = \bigcup_N \mathrm{Supp}(\sigma_{B_N, u_N})$.
\end{proof}

\subsection{Proof of the lower bound}
\label{sec:lower_bound_k}

We prove here the lower bound in Theorem \ref{thm:k_points_branching}. Let $E$ be the event that every family of $k$ points of the interlacement are connected by less than $n(k,d)$ trajectories:
\[
  E = \bigcap_{x_1,\dots,x_k\in\mathcal{I}^u}\left\{\!\!\begin{array}{c} x_1,\dots,x_k \text{ are connected by less than }\\ n(k,d) \text{ trajectories of the interlacement}\end{array}\!\!\!\right\}.
\]

The main goal of this section is to prove the following proposition, from which we directly deduce the theorem.

\begin{proposition}
  \label{prop:lower_bound_k}
  Let $k \geq 2$, $\varepsilon \in (0,1)$. There exists $C = C(k,\varepsilon)$, such that for any ${x_1,\dots,x_k\in\mathbb{Z}^d}$,
  \[
    \P\left(\{x_1,\dots,x_k \in \mathcal{I}\} \cap E\right) \leq C\max_{1\leq i\neq j \leq k} \|x_i-x_j\|^{-1+\varepsilon}.
  \]
\end{proposition}

\noindent We now briefly conclude the proof of our main result and prove the proposition afterwards.

\begin{proof}[Proof of the lower bound in Theorem \ref{thm:k_points_branching}]
  By contradiction, suppose that $\P(E) > 0$. Remember from \eqref{eq:lower_bound_intersect_ball} that there exists $R_0 > 0$, such that for all $R \geq R_0$,
\[
  \P(\mathcal{I}\cap \mathrm{B}(0,R) \neq \varnothing) \geq 1 - \frac{\P(E)}{k+1}.
\]
Then if $z_1,\dots,z_k\in\mathbb{Z}^d$,
\[
  \P\left(E \cap \{\mathcal{I}\cap \mathrm{B}(z_1,R) \neq \varnothing\} \cap \dots \cap \{\mathcal{I}\cap \mathrm{B}(z_k,R) \neq \varnothing\}\right) \geq \P(E) - k\P(\mathcal{I}\cap\mathrm{B}(0,R)=\varnothing) \geq \frac{\P(E)}{k+1}.
\]
However, for every $N > 3R$, we can choose $z_1,\dots,z_k \in \mathbb{Z}^d$ such that $\min_{1\leq i\neq j\leq k}\|z_i - z_j\| > N$. Then we have, by applying Proposition \ref{prop:lower_bound_k} with $\varepsilon = 1/2$ in the last inequality,
\end{proof}

A vertex is said to be \textit{external} if its degree is $1$, conversely, a vertex is \textit{internal} if its degree is at least $2$.
For $k,n \geq 1$, denote by $\mathfrak{T}_{k,n}$ the set of planar trees rooted at an external vertex with $k$ external nodes and $n$ internal nodes. If $T\in\mathfrak{T}_{k,n}$, $x_1,\dots,x_k \in \mathbb{Z}^d$, $w_1,\dots,w_n \in W^*$, and $y_1,\dots,y_{n-1} \in \mathbb{Z}^d$, we write $T(\overline{w}_n,\overline{x}_k, \overline{y}_{n-1})$ for the copy of $T$ where the internal nodes have been replaced by $w_1,\dots,w_n$ and the external ones by $x_1,\dots,x_k$ (in a depth-first search order from the root). Moreover, an edge to an external vertex $x_i$ is labeled by $x_i$, whereas edges between internal vertices are labeled by $(y_i)_{1\leq i\leq n-1}$ (in a depth-first search order from the root).

\begin{lemma}
  \label{lem:lower_bound_event_inclusion}
  Let $k \geq 2$, $x_1,\dots,x_k\in\mathbb{Z}^d$, and $\sigma_u$ with law $\mathrm{Poi}(u, W^*)$. Then
  \begin{align*}
    \left\{\!\!\begin{array}{c} x_1,\dots,x_k \text{ are connected} \\ \text{ with less than } n(k,d) \\ \text{ trajectories of } \sigma_u \end{array}\!\!\right\} 
     = \bigcup_{n=1}^{n(k,d)-1}\!\!\!\bigcup_{w_1,\dots,w_n\in\sigma_u} \bigcup_{T \in \mathfrak{T}_{k,n}} \bigcup_{y_1,\dots,y_{n-1}} \bigcap_{i=1}^n \left\{\!\!\begin{array}{c} w_i \text{ connects the labels}\\ \text{of its adjacent edges } \\ \text{in } T(\overline{w}_n,\overline{x}_k,\overline{y}_{n-1}) \end{array}\!\!\right\}.
  \end{align*}
\end{lemma}
\begin{proof}
  Fix $x_1,\dots,x_k \in \mathbb{Z}^d$. Suppose that these points are connected by a minimal number $n$ of trajectories $w_1,\dots,w_n \in \sigma_u$. For each pair of intersecting trajectories $w_i,w_j$ denote by $y_{i,j}\in\mathbb{Z}^d$ a point in their intersection. Consider the graph $G(\overline{w}_n, \overline{x}_k, \overline{y}_{n^2})$ with $k$ \textit{target} vertices $x_1,\dots,x_k$ and $n$ \textit{trajectory} vertices $w_1,\dots,w_n$, and with the following edge set,
  \[
    \left\{\{w_i,w_j\} : \mathrm{Range}(w_i)\cap \mathrm{Range}(w_j) \neq \varnothing\} \cup \{\{w_i,x_j\} : x_j \in \mathrm{Range}(w_i)\right\}.
  \]
  Label each edge $\{w_i,w_j\}$ by the corresponding point $y_{i,j}$, and each edge $\{w_i,x_j\}$ by $x_j$. This graph is connected under our hypothesis. This yields the following equality of events,
  \[
    \left\{\!\!\begin{array}{c} x_1,\dots,x_k \text{ are connected with less} \\ \text{than } n(k,d) \text{ trajectories of } \sigma_u \end{array}\!\!\right\} = \bigcup_{n<n(k,d)} \bigcup_{w_1,\dots,w_n} \bigcup_{(y_{i,j})_{i,j}} \bigcap_{i=1}^n \left\{\!\!\begin{array}{c} w_i \text{ connects the labels of its} \\ \text{adjacent edges in } G(\overline{w}_n,\overline{x}_k, \overline{y}_{n^2})\end{array}\right\}.
  \]
  If two trajectories $w_i$, $w_j$ intersect in a point $x_\ell$, then $w_i,w_j,x_\ell$ form a triangle in $G$. Therefore, we can choose a spanning tree $T$ of $G(\overline{w}_n, \overline{x}_k, \overline{y}_{n^2})$, such that each point $x_i$ is an external vertex. Moreover, by minimality of $n$, external vertices are exactly $x_1,\dots,x_k$. Since the tree $T\backslash\{x_1,\dots,x_k\}$ has $n$ vertices and thus $n-1$ edges, there are $n-1$ edges labelled by $y_1,\dots,y_{n-1}$ in $T$. This concludes the proof since $T \in \mathfrak{T}_{k,n}$.
\end{proof}

We say that a tree $T'$ is a \textit{strict subtree} of $T$, if it is different from $T$ and its external vertices are external vertices of $T$. A tree $T \in \mathfrak{T}_{k,n}$ \textit{satisfies condition} $(*)$, if for each strict subtree $T'$ of $T$, the number $n' \leq n$ of internal vertices and $k' < k$ of external vertices satisfy
\begin{equation}
  \label{eq:strict_subtree}
  n' \geq n(k',d).
\end{equation}
We denote by $\mathfrak{T}^*_{k,n}$ the set of trees of $\mathfrak{T}_{k,n}$ which satisfy $(*)$.

Remember the definition of the set $\mathbb{T}_n^k$ from Section \ref{sec:n_pts_connection_walk}, and let $\mathbb{T}_\circ^\bullet = \bigcup_{n\geq 1}\bigcup_{k\leq n} \mathbb{T}_n^k$. For $k \geq 2$, $n \geq 1$ and $m \leq k + n - 2$, let $\mathbb{A}_{k,n,m}$ be the set of rooted planar trees with the following properties,
\begin{itemize}
\item the root has degree $1$,
\item the internal nodes are decomposed into two subsets called \textit{intersection} and \textit{branching} vertices with respective cardinality $n-1$ and $m$,
\item there are $k$ external vertices,
\item each edge has a color in $\{1,\dots,n\}$,
\item for each $i \in \{1,\dots,n\}$, the set of edges with color $i$ forms a rooted planar tree $\mathcal{A}_i \in \mathbb{T}_\circ^\bullet$,
\item for $1 \leq i \neq j \leq n$, the intersection of $\mathcal{A}_i$ and $\mathcal{A}_j$ is either empty or an intersection vertex,
\item each strict subtree with $n' \leq n$ different edge colors and $k' < k$ external vertices satisfies \eqref{eq:strict_subtree}.
\end{itemize}

For $\mathcal{A} \in \mathbb{A}_{k,n,m}$, $\overline{x}_k \in (\mathbb{Z}^d)^k$, $\overline{y}_{n-1}\in (\mathbb{Z}^d)^{n-1}$, and $\overline{z}_m \in (\mathbb{Z}^d)^m$, we denote by $\mathcal{A}(\overline{x}_k, \overline{y}_{n-1}, \overline{z}_m)$ the copy of $\mathcal{A}$ where the external vertices have been replaced by $x_1,\dots,x_k$, the intersection vertices by $y_1,\dots,y_{n-1}$, and the branching vertices by $z_1,\dots,z_m$ (replace the vertices in the order of a depth-first search from the root). We also denote by $\mathcal{A}_i(\overline{x}_k, \overline{y}_{n-1}, \overline{z}_m)$ the subtree of color $i$ in $\mathcal{A}(\overline{x}_k, \overline{y}_{n-1}, \overline{z}_m)$.

\begin{definition}[Weight]
  \label{def:tree_weight}
  Given a tree $T$ equiped with functions $\varphi : V(T) \to \mathbb{Z}^d$, and $\alpha : E(T) \to \mathbb{R}$, we define the \textit{weight} of $T$, as
  \[
    p^{\alpha}_\varphi(T) = \prod_{(u,v) \in E(T)} \frac{1}{1 + \|\varphi(u)-\varphi(v)\|^{\alpha(u,v)}}.
  \]
  For a tree $T$ with vertices in $\mathbb{Z}^d$, we write $p^{\alpha}(T)$ for $p^{\alpha}_{\mathrm{Id}}(T)$. When there is no possible ambiguity, we omit the exponent $\alpha$. Furthermore, if the weight is constant equal to $c$, we simply write $p^c_\varphi$.
\end{definition}

\begin{lemma}
  \label{lem:link_T_A}
  Let $n < n(k,d)$. Then, for every $x_1,\dots,x_k \in\mathbb{Z}^d$ and $y_1,\dots,y_{n-1}\in\mathbb{Z}^d$,
  \[
    \sum_{T \in \mathfrak{T}^*_{k,n}} \P\left(\bigcup_{w_1,\dots,w_n\in\sigma_u} \bigcap_{i=1}^n \left\{\!\!\begin{array}{c} w_i \text{ connects the labels} \\ \text{of its adjacent edges} \\ \text{in } T(\overline{w}_n,\overline{x}_k,\overline{y}_{n-1}) \end{array}\!\!\right\} \right) \lesssim \sum_{m=1}^{k+n-2} \sum_{\mathcal{A}\in\mathbb{A}_{k,n,m}}\sum_{z_1,\dots,z_m} p^{d-2}\big(\mathcal{A}(\overline{x}_k,\overline{y}_{n-1},\overline{z}_m)\big).
  \]
\end{lemma}
\begin{proof}
  Let $n < n(k,d)$, $x_1\dots,x_k\in\mathbb{Z}^d$, $y_1,\dots,y_{n-1}\in\mathbb{Z}^d$ and $T\in\mathfrak{T}_{k,n}^*$. For a subset $V$ of $\mathbb{Z}^d$, let $S(V) = \{w \in W^* : V \subset \mathrm{Range}(w)\}$. Denote by $v^i_1,\dots,v^i_{d_i}$ the labels of the adjacent edges of $w_i$ in $T(\overline{w}_n,\overline{x}_k,\overline{y}_{n-1})$. By the Palm-Mecke's Theorem for general Poisson point processes (from \cite[Chapter 13.1]{dvj2008}), recalling the definition of the measure $\nu$ from Section \ref{sec:branching_interlacements},
  \begin{equation}
    \label{eq:palm_mecke_T}
    \E\Bigg[\sum_{w_1,\dots,w_n\sigma_u}\prod_{i=1}^n \indic{w_i \text{ connects } v^i_1,\dots,v^i_{d_i}}\Bigg] = u^n\prod_{i=1}^n \nu\big(S(\{v^i_1,\dots,v^i_{d_i}\})\big).
  \end{equation}
  If $X$ denotes a random walk indexed by an infinite invariant tree $\mathcal{T}$, observe that by definition of $\nu$ one has,
  \[
    \nu\big(S(V)\big) = \sum_{x\in V} \P_x\big(V\subset \mathrm{Range}(X), \mathcal{T}_-^x \cap V = \varnothing\big) \leq \sum_{x\in V} \P_x\big(V \subset \mathrm{Range}(X)\big).
  \]
  Then
  $\nu\big(S(\{v^i_1,\dots,v^i_{d_i}\})\big) \leq \sum_{j=1}^{d_i} \P_{v^i_j}\big(v^i_1,\dots,v^i_{d_i} \in \mathrm{Range}(X)\big)$, thus by Lemma \ref{lem:n_points_connection},
  \begin{equation}
    \label{eq:nu_to_t}
    \prod_{i=1}^n \nu\big(S(\{v^i_1,\dots,v^i_{d_i}\})\big) \lesssim  \prod_{i=1}^n \sum_{m=0}^{d_i-1} \sum_{z_1,\dots,z_m} \sum_{t \in \mathbb{T}_{d_i+m}^{d_i}(\overline{v}^i_{d_i}, \overline{z}_m)} \prod_{\{u,u'\} \in E(t)} \frac{1}{1+\|u-u'\|^{d-2}}.
  \end{equation}
  Given, for each $i \in \{1,\dots,n\}$, $m_i \in \{0,\dots,d_i-1\}$, $z^i_1,\dots,z_{m_i}^i \in \mathbb{Z}^d$, and $t_i \in \mathbb{T}_{d_i+m_i}^{m_i}(\overline{v}^i_{d_i}, \overline{z}_{m_i}^i)$, we replace each vertex $w_i$ in $T(\overline{w}_n,\overline{x}_k,\overline{y}_{n-1})$ and its adjacent edges (which form a star-shaped graph) by $t_i$. We get a planar tree $\mathcal{A}$ with vertices $x_1,\dots,x_k, y_1,\dots,y_{n-1},z_1,\dots,z_m$. Call the points $x_1,\dots,x_{k}$ \textit{target} vertices, $y_1,\dots y_{n-1}$ \textit{intersection} vertices and $z_1,\dots,z_m$ \textit{branching} vertices. Color each tree $t_i$ of $\mathcal{A}$ with color $i$. The number of total branching vertices is
  \[
    m = \sum_{i=1}^n m_i \leq \sum_{i=1}^n (d_i-1) = 2\#E(T) - \sum_{i=1}^k\mathrm{deg}_T(x_i) - n \leq 2(n+k-1) - k - n = n+k-2.
  \]
  Hence $\mathcal{A} \in \mathbb{A}_{k,n,m}(\overline{x}_k,\overline{y}_{n-1},\overline{z}_m)$. Furthermore, since the edges of $\mathcal{A}$ are the edges of the trees $t_i$,
  \begin{equation}
    \label{eq:T_to_A}
    \prod_{i=1}^n \prod_{\{u,u'\} \in E(t_i)} \frac{1}{1+\|u-u'\|^{d-2}} = \prod_{\{u,u'\}\in E(\mathcal{A})} \frac{1}{1+\|u-u'\|^{d-2}} = p^{d-2}(\mathcal{A}).
  \end{equation}
  We conclude by combining \eqref{eq:palm_mecke_T}, \eqref{eq:nu_to_t} and \eqref{eq:T_to_A}.
\end{proof}

\begin{lemma}
  \label{lem:upper_bound_A_star}
  Let $k \geq 2$, $n < n(k,d)$, $m \leq k+n-2$ and $\varepsilon \in (0,1)$. There exists $C > 0$, such that for every $\mathcal{A} \in \mathbb{A}_{k,n,m}$ and $x_1,\dots,x_k \in \mathbb{Z}^d$,
  \[
    \sum_{y_1,\dots,y_{n-1}} \sum_{z_1,\dots,z_{m}} \!\!\!p^{d-2}(\mathcal{A}(\overline{x}_k,\overline{y}_{n-1},\overline{z}_m)) \leq C\max_{1 \leq i \neq j \leq k}\|x_i-x_j\|^{-1+\varepsilon}.
  \]
\end{lemma}

We now have all we need to prove Proposition \ref{prop:lower_bound_k}. As the proof of Lemma \ref{lem:upper_bound_A_star} is long and technical, we postpone it to the end of the section.

\begin{proof}[Proof of Proposition \ref{prop:lower_bound_k}]
  Fix $\varepsilon \in (0,1)$ for the whole proof. We show the proposition by induction on $k$. If $k = 2$, the result has already been proved in Lemma \ref{lem:2_points_lower_bound_probability}. Let now $k \geq 3$, $x_1,\dots,x_k \in \mathbb{Z}^d$, and suppose that the proposition holds for every $k' < k$.

  For $n < n(k,d)$ and $T \in \mathfrak{T}_{k,n}$, let $E_T$ be the event
  \begin{equation}
    \label{def:E_A}
    E_T = \bigcup_{w_1,\dots,w_n \in \sigma_u} \bigcup_{y_1,\dots,y_{n-1}} \bigcap_{i=1}^n \left\{\!\!\begin{array}{c} w_i \text{ connects the labels of its} \\ \text{adjacent edges in } T(\overline{w}_n,\overline{x}_k,\overline{y}_{n-1}) \end{array}\!\!\right\}.
  \end{equation}
  Observe that thanks to Lemma \ref{lem:lower_bound_event_inclusion}, for any $n$ as above, it all boils down to proving that,
  \begin{equation}
    \label{eq:E_A_total}
    \sum_{T \in \mathfrak{T}_{k,n}}\P(E_T) \lesssim \max_{i\neq j} \|x_i-x_j\|^{-1+\varepsilon}.
  \end{equation}

  On the one hand, if $T \notin \mathfrak{T}_{k,n}^*$, then $T$ contains a strict subtree $T'$ that does not satisfy $(*)$. Denote by $n'$ (resp.~$k'$) the number of edge colors (resp.~external vertices) of $T'$. Then by \eqref{eq:strict_subtree}, this means that $n' < n(k',d)$ trajectories connect $k'$ points among $\{x_1,\dots,x_k\}$. We deduce from the induction hypothesis that
  \begin{equation}
    \label{eq:E_A_not_A_star}
    \forall T \notin \mathfrak{T}_{k,n}^*,\quad \P(E_T) \lesssim \max_{i\neq j}\|x_i-x_j\|^{-1+\varepsilon}.
  \end{equation}
  On the other hand, by Lemmas \ref{lem:link_T_A} and \ref{lem:upper_bound_A_star},
  \begin{equation}
    \label{eq:E_A_A_star}
    \sum_{T \in \mathfrak{T}^*_{k,n}} \P(E_T) \lesssim \sum_{m=1}^{k+n-2} \sum_{\mathcal{A}\in\mathbb{A}_{k,n,m}} \sum_{y_1,\dots,y_{n-1}}\sum_{z_1,\dots,z_m} p^{d-2}(\mathcal{A}(\overline{x}_k,\overline{y}_{n-1},\overline{z}_m)) \lesssim \max_{i\neq j} \|x_i-x_j\|^{-1+\varepsilon}.
  \end{equation}

  \noindent Combining \eqref{eq:E_A_not_A_star} and \eqref{eq:E_A_A_star}, we deduce \eqref{eq:E_A_total}. This concludes the proof of the proposition.
\end{proof}

We now prove Lemma \ref{lem:upper_bound_A_star}. We start by introducing two operations (cut and merge), that will serve to reduce a tree $\mathcal{A} \in \mathbb{A}_{k,n,m}$ to a single edge.
For a tree $T$, a mapping $\varphi$  from $V(T)$ to $\mathbb{Z}^d$, and $\partial\notin V(T)$, we denote by $\varphi_{y}$ the extension of $\varphi$ that maps $\partial$ to $y \in \mathbb{Z}^d$. Likewise given $\partial,\partial',y,y'$, we write $\varphi_{y,y'}$ for the extension of $\varphi$ that maps $\partial$ to $y$ and $\partial'$ to $y'$.

\begin{lemma}[Cut operation, see Figure \ref{fig:cut}]
  \label{lem:cut}
  Let $T$ be a tree, and $\alpha_T : E(T) \to \mathbb{R}$. Let $y$ be an internal node of $T$ with at least two external vertices $\partial$ and $\partial'$ as neighbours. Let $T' = T \backslash\{\partial'\}$, and $\alpha_{T'} : E(T') \to \mathbb{R}$ which coincides with $\alpha_T$ except on $\{y,\partial\}$, where
  \[
    \alpha_{T'}(y,\partial) = \alpha_T(y,\partial) + \alpha_T(y,\partial').
  \]
  Then for every function $\varphi : V(T)\backslash\{\partial,\partial'\} \to \mathbb{Z}^d$, and $x,x'\in\mathbb{Z}^d$,
  \[
    p_{\varphi_{x,x'}}(T) \lesssim p_{\varphi_x}(T') + p_{\varphi_{x'}}(T') \lesssim \max_{a\in\{x,x'\}} p_{\varphi_a}(T').
  \]
\end{lemma}
\begin{proof}
  Fix a function $\varphi : V(T)\backslash\{\partial,\partial'\} \to \mathbb{Z}^d$, and $x,x'\in\mathbb{Z}^d$. Observe that the only factor with $\partial$ and $\partial'$ in $p_{\varphi_{x,x'}}(T)$ is $\big((1+\|x-\varphi(y)\|^{\alpha(\partial,y)})(1+\|x'-\varphi(y)\|^{\alpha(\partial',y)})\big)^{-1}$. Since either $\|x-\varphi(y)\| \leq \|x'-\varphi(y)\|$ or $\|x'-\varphi(y)\| \leq \|x-\varphi(y)\|$, we have
\[
  \frac{1}{(1+\|x-\varphi(y)\|^{\alpha(\partial,y)})(1+\|x'-\varphi(y)\|^{\alpha(\partial',y)})} \lesssim \frac{1}{1+ \|x-\varphi(y)\|^{\alpha(\partial,y)+\alpha(\partial',y)}} + \frac{1}{1+\|x'-\varphi(y)\|^{\alpha(\partial,y)+\alpha(\partial',y)}}.
\]
The result follows.
\end{proof}

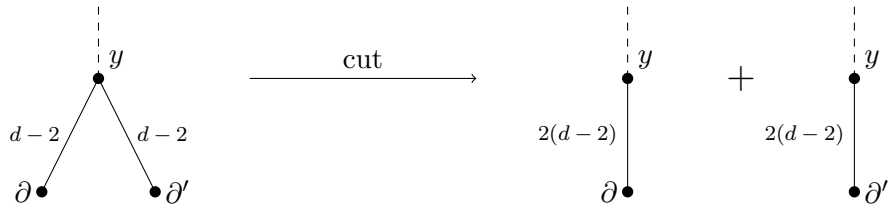
\begin{figure}[H]
  \centering
  \begin{tikzpicture}
    \begin{scope}
      \draw (-3,0) node {\phantom{ }};
      \coordinate (y) at (0,0);
      \coordinate (x1) at (-0.75,-1.5);
      \coordinate (x2) at (0.75,-1.5);

      \node at (y) [circle,fill,inner sep=1.5pt]{};
      \draw[above right] (y) node {$y$};
      \node at (x1) [circle,fill,inner sep=1.5pt]{};
      \draw[left] (x1) node {$\partial$};
      \node at (x2) [circle,fill,inner sep=1.5pt]{};
      \draw[right] (x2) node {$\partial'$};
      
      \draw[thin, dashed] (y) -- (0,1);
      \draw (y) -- (x1) node [midway, left] {\scriptsize$d-2$};
      \draw (y) -- (x2) node [midway, right] {\scriptsize$d-2$};
    \end{scope}
    \draw[->] (2,-0) -- (5,0) node [midway,above] {$\mathrm{cut}$};
    \begin{scope}[xshift = 7cm]
      \coordinate (y1) at (0,0);
      \coordinate (y2) at (3,0);
      \coordinate (x1) at (0,-1.5);
      \coordinate (x2) at (3,-1.5);

      \node at (y1) [circle,fill,inner sep=1.5pt]{};
      \draw[above right] (y1) node {$y$};
      \node at (x1) [circle,fill,inner sep=1.5pt]{};
      \draw[left] (x1) node {$\partial$};
      \node at (y2) [circle,fill,inner sep=1.5pt]{};
      \draw[above right] (y2) node {$y$};
      \node at (x2) [circle,fill,inner sep=1.5pt]{};
      \draw[right] (x2) node {$\partial'$};

      \draw (1.5,0) node {\Large$+$};
      
      \draw (y1) -- (x1) node [midway, left] {\scriptsize$2(d-2)$};
      \draw[thin, dashed] (y1) -- (0,1);
      \draw (y2) -- (x2) node [midway, left] {\scriptsize$2(d-2)$};
      \draw[thin, dashed] (y2) -- (3,1);
    \end{scope}
  \end{tikzpicture}
  \caption{An example of a cut operation, where the weights of the two initial edges are $d-2$.}
  \label{fig:cut}
\end{figure}

\begin{lemma}[Merge operation, see Figure \ref{fig:merge}]
  \label{lem:merge}
    Let $T$ be a tree with an internal node $\partial$ only attached to an external vertex $x$ and another node $z$. Denote by $T'$ the copy of $T$ without the vertex $\partial$ and where the edges $\{x,\partial\}$ and $\{z,\partial\}$ have been merged into a single edge $\{x,z\}$. Let $\alpha_{T'}$ be a weight function for $T'$, which coincides with $\alpha_T$ except that
  \[
    \alpha_{T'}(x,z) = \alpha_T(x,\partial) + \alpha_T(z,\partial) - d.
  \]
  Suppose that $\alpha_T(x,\partial)$ and $\alpha_T(\partial,z) < d$ and $\alpha_T(x,\partial) + \alpha_T(z,\partial) > d$. Then for any function $\varphi : V(T)\backslash\{\partial\} \to \mathbb{Z}^d$,
  \[
    \sum_{y\in\mathbb{Z}^d} p_{\varphi_y}(T) \lesssim p_\varphi(T').
  \]
\end{lemma}
\begin{proof}
  The vertex $y\in\mathbb{Z}^d$ appears only once in $p_{\varphi_y}(T)$. Moreover, with the assumptions made on $\alpha_T$, we get using Lemma \ref{lem:magic_ineq},
  \[
    \sum_{y \in \mathbb{Z}^d} \frac{1}{\big(1+\|\varphi(x)-y\|^{\alpha_T(x,\partial)}\big)\big(1+\|\varphi(z)-y\|^{\alpha_T(z,\partial)}\big)} \lesssim \frac{1}{1+\|\varphi(x)-\varphi(z)\|^{\alpha_T(x,\partial)+\alpha_T(z,\partial)-d}},
  \]
  and the result follows.
\end{proof}

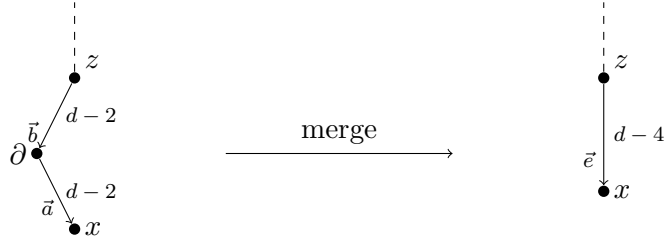
\begin{figure}[H]
  \centering
  \begin{tikzpicture}
    \begin{scope}
      \node (y2) at (0,0) [circle,fill,inner sep=1.5pt]{};
      \draw[above right] (y2) node {$z$};
      \node (y1) at (-0.5,-1) [circle,fill,inner sep=1.5pt]{};
      \draw[left] (y1) node {$\partial$};
      \node (x) at (0,-2) [circle,fill,inner sep=1.5pt]{};
      \draw[right] (x) node {$x$};
      
      \draw[thin, dashed] (y2) -- (0,1);
      \draw[->] (y2) -- (y1) node [midway, right] {\scriptsize$d-2$} node[left, pos=0.75] {\scriptsize$\vec{b}$};
      \draw[->] (y1) -- (x) node [midway, right] {\scriptsize$d-2$} node[left, pos=0.75] {\scriptsize$\vec{a}$};
    \end{scope}
    \draw[->] (2,-1) -- (5,-1) node [midway,above] {$\mathrm{merge}$};
    \begin{scope}[xshift = 7cm]
      \node (y2) at (0,0) [circle,fill,inner sep=1.5pt]{};
      \draw[above right] (y2) node {$z$};
      \node (x) at (0,-1.5) [circle,fill,inner sep=1.5pt]{};
      \draw[right] (x) node {$x$};
      
      \draw[thin, dashed] (y2) -- (0,1);
      \draw[->] (y2) -- (x) node [midway, right] {\scriptsize$d - 4$} node[left, pos=0.75] {\scriptsize$\vec{e}$};
    \end{scope}
  \end{tikzpicture}
  \caption{An example of a merge operation, where the weights of the two initial edges are $d-2$.}
  \label{fig:merge}
\end{figure}

Remember that a tree $\mathcal{A}$ in $\mathbb{A}_{k,n,m}$ is a rooted tree. As such, its edges are naturally oriented. Then, for an oriented edge $\vec{e}$ of $\mathcal{A}$, let $\mathcal{A}(\vec{e})$ be the descendent tree of $\vec{e}$, $x(\vec{e})$ the number of external vertices of $\mathcal{A}(\vec{e})$ and $y(\vec{e})$ (resp.~$z(\vec{e})$) be the number of intersection (resp.~branching) vertices in $\mathcal{A}(\vec{e})$. Denote by $V^b$ and $V^i$ the sets of branching and intersection vertices of $\mathcal{A}$ respectively.

\begin{lemma}
  \label{lem:bound_z_delta}
  Let $\mathcal{A} \in \mathbb{A}_{k,n,m}$. For each $\vec{a} = (u,v) \in \mathcal{A}_i$, let
  \[
    \delta(\vec{a}) = \begin{cases}
      0 & \text{if } \vec{a} \text{ has no parent of if they have different colors}, \\
      1 - \frac{\indic{u \in V^b}}{\mathrm{deg}_{\mathcal{A}_i}(u) - 1}(2 - \indic{\mathrm{deg}_{\mathcal{A}_i}(u) = 2}) &\text{otherwise.}
      \end{cases}
  \]
  Then for $\vec{e} \in \mathcal{A}$,
  \begin{equation}
    \label{eq:bound_z_delta}
    x(\vec{e}) + y(\vec{e}) - 1 \leq z(\vec{e}) + \sum_{\vec{a} \in \mathcal{A}(\vec{e})} \delta(\vec{a}) \leq x(\vec{e}) + y(\vec{e}).
  \end{equation}
\end{lemma}
\begin{proof}
  Let $i$ be an edge color of the tree. For an edge $\vec{e}$ of $\mathcal{A}_i$, let $x^i(\vec{e}), y^i(\vec{e}), z^i(\vec{e})$ be respectively the number of external vertices, intersection and branching vertices in $\mathcal{A}_i(\vec{e}) = \mathcal{A}(\vec{e}) \cap \mathcal{A}_i$. Note that since $\mathcal{A}_i \in \mathbb{T}_\circ^\bullet$, there exists at most one branching vertex $u_i$ in $\mathcal{A}_i$ with degree $2$. We prove by induction, that
  \begin{equation}
    \label{eq:z_plus_delta}
    z^i(\vec{e}) + \sum_{\vec{a} \in \mathcal{A}_i(\vec{e})} \delta(\vec{a}) = x^i(\vec{e}) + y^i(\vec{e}) - 1 + \indic{u_i \in \mathcal{A}_i(\vec{e})},
  \end{equation}
  where if the vertex $u_i$ does not exist, the indicator function above equals $0$. If $\vec{e} = (u,v)$ is connected to an external vertex of $\mathcal{A}_i$, then $\mathcal{A}_i(\vec{e})$ is reduced to $v$, which is either an external vertex of $\mathcal{A}$ or an intersection vertex, thus $\indic{u_i \in \mathcal{A}_i(\vec{e})} = 0$, $x^i(\vec{e}) + y^i(\vec{e}) = 1$ and $z^i(\vec{e}) = 0$. Thus \eqref{eq:z_plus_delta} holds. Suppose now that $\vec{e}$ has $r \geq 1$ descendent edges $(\vec{e}_j)_{1\leq j\leq r}$ in $\mathcal{A}_i$. Note that $\mathcal{A}_i(\vec{e}) = \bigcup_{j=1}^r\big(\mathcal{A}_i(\vec{e}_j) \cup \{\vec{e}_j\}\big)$, and $z^i(\vec{e}) = \sum_{j=1}^r z^i(\vec{e}_j) + \indic{v \in V^b}$, then
  \begin{equation}
    \label{eq:expansion_z_delta}
    z^i(\vec{e}) + \sum_{\vec{a} \in \mathcal{A}_i(\vec{e})} \delta(\vec{a}) = \sum_{j=1}^r \Big(z^i(\vec{e}_j) + \sum_{\vec{a} \in \mathcal{A}_i(\vec{e}_j)}\delta(\vec{a})\Big) + \sum_{j=1}^r\delta(\vec{e}_j) + \indic{v \in V^b}.
  \end{equation}
  Note that $\sum_{j=1}^r \indic{u_i \in \mathcal{A}(\vec{e}_j)} = \indic{u_i \in \mathcal{A}(\vec{e})} - \indic{v = u_i}$. Then, plugging the induction hypothesis into \eqref{eq:expansion_z_delta}, we get
  \begin{align*}
    z^i(\vec{e}) + \sum_{\vec{a} \in \mathcal{A}_i(\vec{e})} \delta(\vec{a})
    &= \sum_{j=1}^r \Big(x^i(\vec{e}_j)+y^i(\vec{e}_j)-1 + \indic{u_i \in \mathcal{A}_i(\vec{e})}\Big) + \sum_{j=1}^r\delta(\vec{e}_j) + \indic{v \in V^b} \\
    &= x^i(\vec{e}) + y^i(\vec{e}) - \indic{v\in V^i} - r + \indic{u_i \in \mathcal{A}_i(\vec{e})} - \indic{v=u_i} + \sum_{j=1}^r\delta(\vec{e}_j) + \indic{v \in V^b}. \\
  \end{align*}
  Since $v$ is an internal node, $v$ is either an intersection or a branching vertex. Thus,
  \[
    z^i(\vec{e}) + \sum_{\vec{a} \in \mathcal{A}_i(\vec{e})} \delta(\vec{a}) = x^i(\vec{e}) + y^i(\vec{e}) - 1 + \indic{u_i \in \mathcal{A}_i(\vec{e})} + \Big(\sum_{j=1}^r\delta(\vec{e}_j) - r + 2 \indic{v\in V^b} - \indic{v = u_i}\Big).
  \]
  Recall that if $u_i$ exists, it is the only branching vertex with degree $2$ in $\mathcal{A}_i$, thus $\indic{v = u_i} = \indic{v \in V^b} \indic{\mathrm{deg}_{\mathcal{A}_i}(v) = 2}$ and $r = \mathrm{deg}_{\mathcal{A}_i}(v)-1$. Consequently $\sum_{j=1}^r \delta(\vec{e}_j) = r - 2\indic{v\in V^b} + \indic{v = u_i}$. This proves \eqref{eq:z_plus_delta}.
  
  Now, fix an edge $\vec{e}$ and let $i$ be its color. Furthermore, for $i\neq j$, let $\vec{f}_j$ be the closest edge from $\vec{e}$ of color $j$ in $\mathcal{A}(\vec{e})$. Note that if $u_j$ exists, it is in $\mathcal{A}_j(\vec{f}_j)$. Then, since $\mathcal{A}(\vec{e}) = \mathcal{A}_i(\vec{e}) \cup \bigcup_{j\neq i} \left(\mathcal{A}_j(\vec{f}_j) \cup \{\vec{f}_j\}\right)$ and $\delta(\vec{f}_j) = 0$ (since the parent of $\vec{f}_j$ has a different color), we have
  \begin{align*}
    z(\vec{e}) + \sum_{\vec{a} \in \mathcal{A}(\vec{e})}\delta(\vec{e})
    &\stackrel{\phantom{\eqref{eq:z_plus_delta}}}{=} \sum_{j \neq i} \Big( z^j(\vec{f}_j) + \sum_{\vec{a} \in \mathcal{A}_j(\vec{f}_j)}\delta(\vec{a})\Big) + z^i(\vec{e}) + \sum_{\vec{a} \in \mathcal{A}_i(\vec{e}_i)} \delta(\vec{a}) \\
    &\stackrel{\eqref{eq:z_plus_delta}}{=}\sum_{j\neq i} \Big(x^j(\vec{f}_j) + y^j(\vec{f}_j)\Big) + x^i(\vec{e}) + y^i(\vec{e}) - 1 + \indic{u_i \in \mathcal{A}_i(\vec{e})} \\
    &\stackrel{\phantom{\eqref{eq:z_plus_delta}}}{=}x(\vec{e}) + y(\vec{e}) - 1 + \indic{u_i \in \mathcal{A}_i(\vec{e})}.
  \end{align*}
  This concludes the proof.
\end{proof}

We now have all we need to prove Lemma \ref{lem:upper_bound_A_star}. The idea is to successively apply cut and merge operations on a tree $\mathcal{A}$ of $\mathbb{A}_{k,n,m}$, while controlling the structure of the modified tree with Lemma \ref{lem:bound_z_delta}. At the end of the process, we obtain a single edge with known weight, which allows us to conclude.

\begin{proof}[Proof of Lemma \ref{lem:upper_bound_A_star}]
  Fix $\varepsilon \in (0,1), \mathcal{A} \in \mathbb{A}_{k,n,m}$, $\alpha_{\mathcal{A}}$ constant equal to $d-2$, and $x_1,\dots,x_k\in\mathbb{Z}^d$. We need to show that
  \[
    \sum_{y_1,\dots,y_{n-1}} \sum_{z_1,\dots,z_m} p\big(\mathcal{A}(\overline{x}_k,\overline{y}_{n-1},\overline{z}_m)\big) \lesssim \max_{1\leq i\neq j\leq k} \|x_i - x_j\|^ {-1+\varepsilon},
  \]
  with a constant only depending on $k,n,m$ and $\varepsilon$.
  First, remember the definition of $\delta$ in Lemma \ref{lem:bound_z_delta} and consider the weight function $\alpha'_{\mathcal{A}}$ coinciding with $\alpha_{\mathcal{A}}$ except on the edges connected to external vertices, where
  \[
    \alpha'_{\mathcal{A}}(u,v) = \alpha_{\mathcal{A}}(\vec{e}) - 2\delta(\vec{e}) - \frac{\varepsilon}{k}.
  \]
  Since $\alpha'_{\mathcal{A}}(\vec{e}) \leq \alpha_{\mathcal{A}}(\vec{e})$ for each edge $\vec{e}$, it follows that
  \begin{equation}
    \label{eq:inequality_p_A_Ap}
    \forall y_1,\dots,y_{n-1}, z_1,\dots,z_m\in\mathbb{Z}^d,\quad p^{\alpha_{\mathcal{A}}}\big(\mathcal{A}(\overline{x}_k,\overline{y}_{n-1},\overline{z}_m)\big) \leq p^{\alpha'_{\mathcal{A}}}\big(\mathcal{A}(\overline{x}_k,\overline{y}_{n-1},\overline{z}_m)\big),
  \end{equation}
  and hence it suffices to prove the result for $\alpha_A'$ instead of $\alpha_A$.

  We recursively construct a finite sequence of trees $(\mathcal{A}^i)_{1\leq i\leq N}$ and weight functions $(\alpha_i)_{1\leq i\leq N}$. Simultaneously, we also define an application $\mathcal{E} : \bigcup_{i=1}^N E(\mathcal{A}^i) \to \mathcal{P}(E(\mathcal{A}))$. Let $\mathcal{A}^0 = \mathcal{A}$, $\alpha_0 = \alpha'_{\mathcal{A}}$ and for $\vec{e} \in E(\mathcal{A})$, $\mathcal{E}(\vec{e}) = \{\vec{e}\}$.

  Fix $i \in \mathbb{N}$, and suppose that $\mathcal{A}^i$ and $\alpha_i$ have been defined. If $\mathcal{A}^i$ has an internal node $\partial$ only connected to one external vertex with an edge $\vec{a}$, and to another vertex with an edge $\vec{b}$, parent of $\vec{a}$ (see Figure \ref{fig:merge}), merge the two into an edge $\vec{e}$ with Lemma \ref{lem:merge}. Let $\mathcal{A}^{i+1}$ and $\alpha$ be the results of this operation. Set $\mathcal{E}(\vec{e}) = \{\vec{b}\}$, and let $\alpha_{i+1}$ be the weight function which coincides with $\alpha$ except on $\vec{e}$, where $\alpha_{i+1}(\vec{e}) = \alpha(\vec{e}) - 2 \delta(\vec{b})$. Recall the notation $\varphi_y$ from before Lemma \ref{lem:cut}. Then for any $\varphi : V(\mathcal{A}^i)\backslash \{\partial\} \to \mathbb{Z}$,
  \begin{equation}
    \label{eq:inequality_Ai_merge}
    \sum_{y}p^{\alpha_i}_{\varphi_y}(\mathcal{A}^i) \lesssim p^{\alpha_{i+1}}_\varphi(\mathcal{A}^{i+1}).
  \end{equation}
  Otherwise, $\mathcal{A}^i$ has a vertex connected to at least two external vertices $\partial, \partial'$ through two edges $\vec{a}$ and $\vec{b}$, with the same orientation. Apply Lemma \ref{lem:cut} to cut them into an edge $\vec{e}$, and let $\mathcal{A}^{i+1}$ and $\alpha_{i+1}$ be the result of the operation. Set $\mathcal{E}(\vec{e}) = \mathcal{E}(\vec{a}) \sqcup \mathcal{E}(\vec{b})$. Then for any $\varphi : V(\mathcal{A}^i)\backslash\{\partial,\partial'\}$, and $x,x'\in\mathbb{Z}^d$,
  \begin{equation}
    \label{eq:inequality_Ai_cut}
    p_{\varphi_{x,x'}}^{\alpha_i}(\mathcal{A}^i) \lesssim \max_{a \in \{x,x'\}} p_{\varphi_a}^{\alpha_{i+1}}(\mathcal{A}^{i+1}).
  \end{equation}
  Stop when no cut or merge operations are anymore possible, and denote by $N$ the total number of steps. If $\mathcal{A}^N$ is reduced to a single edge $\vec{e}$, by \eqref{eq:inequality_p_A_Ap}, \eqref{eq:inequality_Ai_merge} and \eqref{eq:inequality_Ai_cut},
  \begin{equation}
    \label{eq:chain_A_sequence}
    \sum_{y_1,\dots,y_{n-1}} \sum_{z_1,\dots,z_m} p(\mathcal{A}(\overline{x}_k, \overline{y}_{n-1},\overline{z}_m)) \lesssim \max_{1\leq i\neq j\leq k} p^{\alpha_N}_{\varphi_{x_i,x_j}}(\mathcal{A}^N) = \max_{1\leq i\neq j\leq k} \|x_i-x_j\|^{-\alpha_{N}(\vec{e})}.
  \end{equation}
  Hence, provided that $\mathcal{A}^N$ is a single edge and that $\alpha_{N}(e) \geq 1 - \varepsilon$, the lemma is proved. Showing both statements amounts to get an explicit expression of $\alpha_{i}$ for $1 \leq i \leq N$.

  We prove by induction that, for $1\leq i\leq N$, the weight of a newly created edge $\vec{e}$ satisfies
  \begin{equation}
    \label{eq:weight_induction}
    \alpha_{i}(\vec{e}) = \sum_{\vec{a}\in \mathcal{E}(\vec{e})} \bigg((d-2)x(\vec{a})- 2\big(y(\vec{a}) + z(\vec{a})\big) - 2\sum_{\vec{b}\in\mathcal{A}(\vec{a})} \delta(\vec{b}) - 2 \delta(\vec{a}) - \frac{\varepsilon}{k} x(\vec{a})\bigg).
  \end{equation}
  If $\vec{e}$ has been created by cutting two edges $\vec{e}_1$ and $\vec{e}_2$ at step $j \leq i$, then $\mathcal{E}(\vec{e}) = \mathcal{E}(\vec{e}_1) \sqcup \mathcal{E}(\vec{e}_2)$, and $\alpha_{i}(\vec{e}) = \alpha_{j}(\vec{e}) = \alpha_{{j-1}}(\vec{e}_1) + \alpha_{{j-1}}(\vec{e}_2)$. Hence by the induction hypothesis
  \begin{align*}
    \alpha_{i}(\vec{e})
    & = \sum_{\vec{a}\in \mathcal{E}(\vec{e}_1)} \bigg((d-2)x(\vec{a})- 2\big(y(\vec{a}) + z(\vec{a})\big) - 2\sum_{\vec{b}\in\mathcal{A}(\vec{a})} \delta(\vec{b}) - 2 \delta(\vec{a}) - \frac{\varepsilon}{k} x(\vec{a})\bigg)\\
    &\quad\quad\quad+ \sum_{\vec{a}\in \mathcal{E}(\vec{e}_2)} \bigg((d-2)x(\vec{a})- 2\big(y(\vec{a}) + z(\vec{a})\big) - 2\sum_{\vec{b}\in\mathcal{A}(\vec{a})} \delta(\vec{b}) - 2 \delta(\vec{a}) - \frac{\varepsilon}{k} x(\vec{a})\bigg) \\
    & = \sum_{\vec{a}\in \mathcal{E}(\vec{e})} \bigg((d-2)x(\vec{a})- 2\big(y(\vec{a}) + z(\vec{a})\big) - 2\sum_{\vec{b}\in\mathcal{A}(\vec{a})} \delta(\vec{b}) - 2 \delta(\vec{a}) - \frac{\varepsilon}{k} x(\vec{a})\bigg).
  \end{align*}
  If $\vec{e}$ has been created by merging two edges $\vec{e}_1$ and $\vec{e}_2$ with $\vec{e}_1$ being the parent of $\vec{e}_2$ at step $j \leq i$, then $\mathcal{E}(\vec{e}) = \{\vec{e}_1\}$, and for each $\vec{a} \in \mathcal{E}(\vec{e}_2)$, $\vec{a}$ is a child of $\vec{e}_1$ in $\mathcal{A}$. Thus, since the internal vertices of $\mathcal{A}$ are branching or intersection vertices,
  \begin{equation}
    \label{eq:sum_x_y_z_parent}
    x(\vec{e}_1) = \sum_{\vec{a}\in\mathcal{E}(\vec{e}_2)} x(\vec{a}), \quad\quad y(\vec{e}_1) + z(\vec{e}_1) = \sum_{\vec{a}\in\mathcal{E}(\vec{e}_2)} \big(y(\vec{a}) + z(\vec{a})\big) + 1,
  \end{equation}
  \begin{equation}
    \label{eq:sum_delta_parent}
    \sum_{\vec{a}\in\mathcal{E}(\vec{e}_2)}\bigg(\sum_{\vec{b} \in \mathcal{A}(\vec{a})} \delta(\vec{b}) + \delta(\vec{a})\bigg) = \sum_{\vec{b}\in\mathcal{A}(\vec{e}_1)} \delta(\vec{b}).
  \end{equation}
  Observe that $\alpha_{i}(\vec{e}) = \alpha_{j}(\vec{e}) = \alpha_{{j-1}}(\vec{e}_1) + \alpha_{{j-1}}(\vec{e}_2) - d - 2\delta(\vec{e}_1)$, and since $\vec{e}_1$ is an internal edge of $\mathcal{A}^{j-1}$, it belongs to $\mathcal{A}$ and its weight has been unmodified, hence $\alpha_{{j-1}}(\vec{e}_1) = d-2$. Together with \eqref{eq:sum_x_y_z_parent} and \eqref{eq:sum_delta_parent}, we get
  \begin{align*}
    \alpha_{i}(\vec{e})
    &= \sum_{\vec{a}\in \mathcal{E}(\vec{e}_2)} \bigg((d-2)x(\vec{a})- 2\big(y(\vec{a}) + z(\vec{a})\big) - 2\sum_{\vec{b}\in\mathcal{A}(\vec{a})} \delta(\vec{b}) - 2 \delta(\vec{a}) - \frac{\varepsilon}{k} x(\vec{a})\bigg) + d - 2 - d - 2\delta(\vec{e}_1) \\
    &= (d-2)x(\vec{e}_1)- 2\big(y(\vec{e}_1) + z(\vec{e}_1)\big) - 2\sum_{\vec{b}\in\mathcal{A}(\vec{e}_1)} \delta(\vec{b}) - 2 \delta(\vec{e}_1) - \frac{\varepsilon}{k} x(\vec{e}_1).
  \end{align*}
  This concludes the proof of \eqref{eq:weight_induction}.
  
  We now want to prove that $\mathcal{A}^N$ is reduced to a single edge. By contradiction, suppose that this is not the case. Since a cut operation has no condition on weights, this means that we stopped because there exist two consecutive edges $\vec{e}$ and $\vec{f}$, whose common vertex does not have any other adjacent edge, that cannot be merged. We want to prove that this is not the case. Assume that $\vec{f}$ is the parent of $\vec{e}$. Note that by construction $\vec{f}$ has not been modified yet, thus one has $\alpha_N(\vec{f}) = d-2$. Now, observe that \eqref{eq:sum_x_y_z_parent} and \eqref{eq:sum_delta_parent} hold with $\vec{f}$ and $\vec{e}$ in place of $\vec{e}_1$ and $\vec{e}_2$ respectively. Therefore with \eqref{eq:weight_induction}, we get
  \[
    \alpha_N(\vec{e}) = (d-2)x(\vec{f})- 2\big(y(\vec{f}) + z(\vec{f}) - 1\big) - 2\sum_{\vec{b}\in\mathcal{A}(\vec{f})} \delta(\vec{b}) - \frac{\varepsilon}{k} x(\vec{f}).
  \]
  Then applying Lemma \ref{lem:bound_z_delta} with $\vec{f}$, yields
  \begin{equation}
    \label{eq:bound_alpha}
    (d-4)x(\vec{f}) - 4y(\vec{f}) + 2  - \frac{\varepsilon}{k} x(\vec{f}) \leq \alpha_{N}(\vec{e}) \leq (d-4)x(\vec{f})- 4y(\vec{f}) + 4  - \frac{\varepsilon}{k} x(\vec{f}).
  \end{equation}
    
  Since $\vec{e}$ and $\vec{f}$ cannot be merged, it means that
  
    \begin{equation}
      \label{eq:upper_bound_alpha_e}
      \alpha_N(\vec{e}) \geq d,
      \qquad \text{or} \qquad
      \alpha_N(\vec{e}) + \alpha_N(\vec{f}) \leq d.
    \end{equation}
  
  We start by showing that the first condition of \eqref{eq:upper_bound_alpha_e} cannot hold. Note that if $\vec{e}$ has been created by a merge, its weight is smaller than $d$. Thus, $\vec{e}$ has been created by a cut. Consider the descendent tree $\mathcal{A}(\vec{f})$ of $\vec{f}$ in $\mathcal{A}$. Since $\vec{e}$ comes from a cut, $\mathcal{A}(\vec{f})$ is a strict subtree of $\mathcal{A}$. Let $n_f$ be the number of different colors in $\mathcal{A}(\vec{f})$. By definition, an intersection vertex lies at a change of colors, and thus $y(\vec{f}) \geq n_f - 1$. Furthermore, since $\mathcal{A} \in \mathbb{A}_{k,n,m}$, by applying \eqref{eq:strict_subtree} with $\mathcal{A}(\vec{f})$, we get
  \[
    y(\vec{f}) \geq n_f -1 \geq n(x(\vec{f}),d) - 1 = \left\lceil\frac{(x(\vec{f})-1)d}{4}\right\rceil - (x(\vec{f})-2) - 1.
  \]
  This yields $4y(\vec{f}) \geq (d-4)x(\vec{f}) - d + 4$. With \eqref{eq:bound_alpha}, we deduce that $\alpha_N(\vec{e}) \leq d - \frac{\varepsilon}{k} x(\vec{a}) < d$. Note that this is the point where the use of $\alpha_{\mathcal{A}}'$ insted of $\alpha_{\mathcal{A}}$ is important, to make sure that the last inequality is strict.
  
  Now, we show that the second condition of \eqref{eq:upper_bound_alpha_e} cannot hold either. Let $T$ be the parent tree of $\vec{f}$ in $\mathcal{A}$. Denote by $n_T$ the number of different colors in $T$, and by $k_T$ the number of external vertices of $\mathcal{A}$ in $T$. We claim that $k_T + x(\vec{f}) = k$, and $y(\vec{f}) \leq n - n_T$. Indeed, if $y_T$ is the number of intersection vertices in $T$, $y_T \geq n_T-1$ and $y(\vec{f}) + y_T = n-1$. Even if $T$ is not a strict subtree of $\mathcal{A}$, it contains a strict subtree of $\mathcal{A}$ with $k_T$ external vertices, and which satisfies \eqref{eq:strict_subtree}. Thus $n_T \geq n(k_T,d)$. Moreover, recall that $n < n(k,d)$, and thus
  \[
    y(\vec{f}) \leq n - n_T \leq \left\lceil \frac{(k-1)d}{4} \right\rceil - (k-2) - 1 - \left\lceil \frac{(k_T-1)d}{4} \right\rceil - (k_T-2) \leq \frac{(k-1)d}{4} - \frac{(k_T-1)d}{4} - x(\vec{f}) - \frac{1}{4}.
  \]
  We deduce that $4y(\vec{f}) \leq (d-4) x(\vec{f}) - 1$. Injecting this in \eqref{eq:bound_alpha} yields
  \[
    \alpha_N(\vec{e}) \geq 3  - \frac{\varepsilon}{k} x(\vec{f}) \geq 3 - \varepsilon.
  \]
  Since $\alpha_N(\vec{f}) = d-2$, we finally get that $\alpha_N(\vec{e}) + \alpha_N(\vec{f}) \geq d + 1 - \varepsilon > d$, contradicting \eqref{eq:upper_bound_alpha_e} and thus the fact that $\vec{e}$ and $\vec{f}$ cannot be merged. This proves that $\mathcal{A}_N$ is reduced to a single edge. It remains to see that its weight is larger than $1 - \varepsilon$.

  Observe that at step $N-1$, the tree has two consecutive edges $\vec{e}$ and $\vec{f}$, say with $\vec{f}$ the parent of $\vec{e}$ and by construction $\alpha_{N-1}(\vec{f}) = d-2$. They are merged into an edge $\vec{g}$ with weight $\alpha_{N-1}(\vec{e}) + \alpha_{N-1}(\vec{f}) - d$. The previous computation shows that $\alpha_{N-1}(\vec{e}) + \alpha_{N-1}(\vec{f}) \geq d + 1 - \varepsilon$. Thus $\alpha_N(\vec{g}) \geq 1 - \varepsilon$, as wanted.
\end{proof}

\printbibliography

\end{document}